\RequirePackage{fix-cm}
\documentclass{svjour3}                     
\smartqed  
\usepackage{mathptmx}      
\usepackage{amsmath,amssymb,amsfonts}%
\usepackage{graphicx}
\usepackage{subfigure}
\usepackage{stmaryrd}
\usepackage[T1]{fontenc}
\usepackage{latexsym,mathrsfs}
\usepackage{float}
\usepackage{xcolor}
\usepackage{algorithm, algorithmic}
\usepackage{array,bm}
\usepackage[colorlinks,linkcolor=blue, anchorcolor=blue, citecolor=blue]{hyperref}
\usepackage{color,caption,cases}
\usepackage{tabularx,booktabs }
\usepackage{indentfirst}
\usepackage{appendix}
\usepackage[numbers,sort&compress]{natbib}
\numberwithin{equation}{section}
\newcommand{\malA}{\mathcal{A}}
\newcommand{\malB}{\mathcal{B}}
\newcommand{\malC}{\mathcal{C}}
\newcommand{\malD}{\mathcal{D}}
\newcommand{\malM}{\mathcal{M}}
\newcommand{\malE}{\mathcal{E}}

\newcommand{\mbfH}{\mathbf{H}}
\newcommand{\mbfL}{\mathbf{L}}
\newcommand{\mbfV}{\mathbf{V}}

\newcommand{\f}{\frac}
\newcommand{\p}{\partial}

\journalname{Arxiv}
\begin{document}
\title{Error estimates of linear decoupled structure-preserving incremental viscosity splitting methods for the Cahn--Hilliard--Navier--Stokes system} 
\titlerunning{Error estimates of linear decoupled structure-preserving incremental viscosity splitting methods for the Cahn--Hilliard--Navier--Stokes system}        

\author{
	Baolin Kuang    \and
	Hongfei Fu$^*$ \and 
	Xiaoli Li 
}
\authorrunning{ Baolin Kuang \and Hongfei Fu \and Xiaoli Li  } 

\institute{ 
	Baolin Kuang  \at
	School of Mathematical Sciences, Ocean University of China, Qingdao  266100, China\\
	\email{blkuang@stu.ouc.edu.cn} 
	\and 
	Hongfei Fu		\at
	Corresponding author. School of Mathematical Sciences \& Laboratory of Marine Mathematics, Ocean University of China, Qingdao  266100, China\\
	\email{fhf@ouc.edu.cn} 
	\and 
	Xiaoli Li		\at 
	School of Mathematics, Shandong University, Jinan, Shandong 250100, China.\\
	\email{xiaolimath@sdu.edu.cn} 
}

\date{Received: date / Accepted: date}

\maketitle
\begin{abstract} 
 We propose first- and second-order time discretization schemes for the coupled Cahn--Hilliard--Navier--Stokes model, leveraging the incremental viscosity splitting (IVS) method. The schemes combine the scalar auxiliary variable method and the zero-energy-contribution approach, resulting in a linear,  decoupled numerical framework. At each time step, they only require to solve a sequence of constant-coefficient equations, along with a linear equation with one unknown, making the algorithms computationally efficient and easy to implement. In addition, the proposed schemes are proven to be uniquely solvable, mass-conserving, and unconditional energy dissipation. Most importantly, leveraging the mathematical induction method and the regularity properties of the Stokes equation, we perform a rigorous error analysis for the first-order scheme in multiple space dimensions, establishing an unconditional and optimal convergence rate for all relevant variables under different norms. A user-defined, time-dependent parameter plays an important role in the error analysis of the proposed structure-preserving IVS methods. Ample numerical examples are carried out to verify the theoretical findings and to demonstrate the accuracy, effectiveness  and efficiency of the proposed schemes. 
\end{abstract}
%
%
%
%

\section{Introduction}\label{sec:Intro}
The phase field method, also known as the diffusive interface method,  has become a prominent mathematical framework for modeling interfacial dynamics in multi-component systems \cite{anderson1998diffuse, christlieb2014high}. The governing equations, namely the phase field model, are typically derived through a variational approach within an appropriate functional setting aimed at minimizing the total free energy.  Notably, the Cahn--Hilliard--Navier--Stokes (CHNS) system \cite{anderson1998diffuse,gurtin1996two,kay2007efficient} represents the fundamental building blocks of hydrodynamic phase-field models for multiphase fluid flow dynamics.

Despite being extensively studied over the past two decades \cite{chen2020novel,zhao2021second,liu2024simple,qian2023fully,han2015second,kay2007efficient},  numerical research on the CHNS system remains challenging due to its inherent complexities, including: (i) the presence of the stiffness terms linked to interfacial width parameters in the phase equation, the nonlinear convection term in the  momentum equation, and the nonlinear couplings between velocity and phase-field variables through convective transport and stress effects; (ii) the intrinsic velocity-pressure coupling imposed by the incompressibility constraint;   (iii) the requirement for numerical methods to preserve essential physical laws, such as energy dissipation law and mass conservation law; and (iv) the theoretical difficulty in deriving rigorous optimal-order error estimates for fully decoupled linear schemes.

Due to the inherent complexity of the CHNS system, considerable efforts have been devoted to establishing error estimates for efficient numerical schemes. For example, Diegel et al. \cite{diegel2017convergence} proposed a second-order convex splitting scheme using mixed finite elements, achieving optimal convergence rates for the velocity in the $\bm L^2$-norm and for the phase-field variable and the chemical potential in the $ H^1$-norm. However, no theoretical estimate was provided for the pressure. Cai and Shen \cite{cai2018error} conducted a rigorous error analysis for a linear, weakly coupled energy-stable scheme with mixed element discretization. Apart from error estimates for the phase-field variable, the chemical potential and the velocity, they also obtained the optimal-order error estimate for the pressure in the $\ell^2(0,T; L^2)$-norm. Xu et al. \cite{xu2020error} proposed a first-order in time semi-discrete scheme based on projection methods \cite{goda1979multistep}  and stabilization technique \cite{shen2010numerical}, which is fully decoupled and energy-stable. 
They proved the optimal-order accuracy for the phase-field variable in the $\ell^\infty(0,T;H^1)$-norm, for the chemical potential in the $\ell^2(0,T;H^1)$-norm, and for the velocity in the $\ell^\infty(0,T; \bm L^2) \cap \ell^2(0,T; \bm H^1)$-norm, and for the pressure also in the $\ell^2(0,T; L^2)$-norm. Recently, Li et al. \cite{li2022fully} proposed first- and second-order fully decoupled time discretization schemes based on the multiple scalar auxiliary variable (MSAV) approach for gradient systems, combined with (rotational) pressure-correction methods \cite{goda1979multistep,guermond2006overview} for the Navier--Stokes equations. They carried out a rigorous error analysis for the first-order scheme, and optimal convergence rates for the phase-field variable and chemical potential in the $H^1$-norm, and the velocity in the $\bm L^2$-norm were presented. Notably, they also provided a half-order error estimate for the pressure in the stronger $\ell^\infty(0,T; L^2)$-norm. However, extension of the error analysis to three-dimensional model remains unresolved.
Recent efforts by Cai et al. \cite{cai2023optimal} derived optimal $L^2$-norm error estimates for convex-splitting finite element methods using novel quasi-projection techniques, however, again no error estimate for the pressure was presented. Despite these great advancements, the construction of fully decoupled linear schemes with constant coefficients and rigorous error analysis in a stronger norm (especially for the pressure) remains limited. Moreover, most existing works rely on projection methods or mixed finite element discretization to handle the velocity-pressure coupling, which may cause artificial boundary conditions for the velocity or require special finite element space pairs for the velocity and pressure.

Apart from conventional projection methods, the viscosity-splitting (VS) method \cite{blasco1998fractional} provides an efficient alternative for solving the Navier--Stokes equations. This approach consists of two main steps: the first step involves solving a linear elliptic problem without incompressibility constraint and pressure term, and the second step preserves the viscosity and addresses a generalized Stokes problem. In contrast to the projection methods, a key advantage is that it preserves full original velocity boundary conditions throughout the time-stepping process, thereby avoiding artificial boundary conditions. Furthermore, El-Amrani et al. extended its application to incompressible non-Newtonian fluids \cite{el2024error} and to thermally coupled Navier–Stokes equations involving convection–conduction dynamics \cite{el2024time}, where the velocity was shown to achieve first-order accuracy in the $\ell^\infty(0,T; \bm L^2) \cap \ell^2(0,T; \bm H^1)$-norm, while only $1/2$-order accuracy in the $\ell^2(0,T;  L^2)$-norm was proved for the pressure. We point out that \cite{el2024error} also asserted that the pressure approximation can attain full first-order accuracy by employing the same arguments as in \cite{guillen2011new,yakoubi2023enhancing}.  
Inspired by the incremental pressure-projection scheme \cite{goda1979multistep}, Yakoubi \cite{yakoubi2023enhancing} enhanced the VS method by proposing an incremental viscosity-splitting (IVS) scheme for the Navier–Stokes equations, which further improves the accuracy and robustness of the VS method. They demonstrated that the scheme achieves first-order convergence for the velocity in the $\ell^\infty(0,T; \bm L^2) \cap \ell^2(0,T; \bm H^1)$-norm, and for the pressure also in the $\ell^2(0,T;  L^2)$-norm.  More recently, Obbadi et al. \cite{obbadi2025stable} proposed a second-order IVS method for the incompressible Navier–Stokes equations and provided the corresponding stability analysis.   However, the existing VS and IVS methods typically employ a semi-implicit treatment for the nonlinear convection terms, which results in variable-coefficient equations that must be updated at each time step. Recently, by introducing a time-dependent auxiliary variable, Sun et al. \cite{Yang2024error} proposed an efficient first-order auxiliary variable viscosity-splitting method for the incompressible fluid flows, which yields constant-coefficient Poisson-type equations at each time step. They also established an energy dissipation law based on kinetic energy rather than a modified energy involving velocity and additional pressure gradient, while, no error analysis has been conducted for this scheme.  
To the best of our knowledge, research on the IVS method remains limited. In particular, no existing work has considered its application to complex two-phase flow models such as the CHNS system, nor has any rigorous error analysis been carried out for such coupled systems.

To address this significant gap, we propose a novel linear, fully decoupled, and structure-preserving numerical framework for the CHNS system by integrating the incremental viscosity-splitting (IVS) method \cite{yakoubi2023enhancing} with the SAV-ZEC approach \cite{zgdsisc2024}.  The SAV-ZEC approach that integrates the SAV method with the zero energy contribution (ZEC) idea \cite{chen2021fully,yang2021novelSISC} is employed to systematically address the nonlinearities of the CHNS system and to ensure energy dissipation. Simultaneously, the IVS method efficiently resolves the velocity–pressure coupling without requiring the solution of saddle-point problems associated with mixed finite element formulations. In addition, to enhance the accuracy of both the velocity and pressure, a user-defined, time-dependent parameter $\gamma(t)$ shall be introduced ahead of the pressure term in the equivalent CHNS system, see \eqref{model:sav:phi}--\eqref{model:sav:r}. As mentioned in Remarks \ref{rem:ivs:gamma} and \ref{remk:sub_rate:1}, it plays an important role in the error analysis.
In summary, we develop a novel integrated numerical framework that offers the following major advantages:
\begin{itemize} 
	\item Computational efficiency: The proposed schemes avoid solving nonlinear systems. At each time step, only constant-coefficient linear equations and a single scalar linear equation need to be solved, significantly reducing the computational cost.
	
	\item Preservation of physical laws: In addition to mass conservation, the proposed schemes are well-posed and unconditionally energy dissipative. 
	Moreover, by incorporating the relaxation technique \cite{jiang2022improving}, the schemes can better guarantee consistency between the original energy and the modified energy.
	
	\item Rigorous error analysis: 
	Under mild regularity assumptions, we derive comprehensive error estimates for the first-order scheme in multiple space dimensions, encompassing all variables including pressure. Specifically, by virtue of the regularity properties of the Stokes equation, we obtain the optimal-order error estimates for the velocity in the $\ell^{\infty}(0,T; \bm H^1) \cap \ell^{2}(0,T; \bm H^2) $-norm and for the pressure in the $\ell^{2}(0,T; H^1)$-norm, thereby completing the induction process. To the best of our knowledge, these results represent the most accurate theoretical estimates available for the IVS method.  Notably, in contrast to projection-type methods for the incompressible Navier–Stokes equations, the proposed scheme avoids numerical inconsistencies caused by artificial boundary conditions, and thus achieves convergence of both velocity and pressure in stronger norms.
\end{itemize}

The remainder of the paper is organized as follows. In Section \ref{equivalent_eq}, we introduce the CHNS model and present some preliminaries. In Section \ref{sec:scheme}, we construct first- and second-order linear structure-preserving IVS schemes, outline an efficient implementation procedure, and prove that both schemes are unique solvable,  mass-conserving and unconditionally energy-dissipative. In Section \ref{sec:error}, we provide rigorous error estimates for the first-order scheme by dividing the proof into three main steps: the rough estimates, the improved estimates, and completion of the mathematical induction. In Section \ref{sec:test}, we present some numerical experiments to verify the accuracy and the reliability of the scheme. Some conclusions are addressed in Section \ref{conclusions}.

\section{Mathematical model and preliminary lemmas}\label{equivalent_eq}
Let $\Omega \subset \mathbb{R}^d, d=2,3$ be an open bounded convex polygonal or polyhedral domain, and let $\bm n$ denote the unit outward normal to the boundary $\p \Omega$. For simplicity, we denote $J:=(0,T]$. In this paper, we are interested in the construction and analysis of an efficient time discretization scheme for the following CHNS  model:
\begin{align}
	& \phi_t+ \nabla \cdot (\bm u  \phi)=M \Delta \mu 
    	&&\ \ \text { in } \Omega \times J, \label{model:phi} \\
	& \mu=- \lambda \Delta \phi+ \lambda G^{\prime}(\phi) 
	 &&\ \ \text { in } \Omega \times J,  \label{model:mu}\\
	&  \bm u_t+\bm u \cdot \nabla \bm u  -\nu \Delta \bm u+\nabla p =- \phi \nabla \mu
	 &&\ \  \text { in } \Omega \times J, \label{model:u} \\
	& \nabla \cdot \bm u=0  
	 &&\ \  \text { in } \Omega \times J, \label{model:div} 
\end{align}
enclosed with the following no-flux/no-flow boundary conditions and initial conditions:
\begin{align}
	& \p_{\bm n} \phi=\p_{\bm n} \mu=0,~~ \bm u=\bm{0}    &\ \  \text { on } \partial \Omega \times J, \label{model:boundary}\\        
	&\phi(\bm x ,0)=\phi^o(\bm x),  ~ \bm u(\bm x ,0)=\bm{u}^o(\bm x)  & \ \  \text { in } \Omega \times \{0\}. \label{model:ivp}
\end{align}
The CHNS phase field equation \eqref{model:phi}--\eqref{model:ivp} is frequently used as a diffuse interface model for the flow of two-phase incompressible fluids with a matched density \cite{gurtin1996two} and to model phase separation of a binary fluid after a deep quench \cite{anderson1998diffuse}. 

The unknown variables in the system \eqref{model:phi}--\eqref{model:ivp} include the  velocity $\bm u$, the  pressure $p$, the phase field function $\phi$, and the chemical potential $\mu$. Moreover, $M$, $\lambda$ and $\nu$ are positive physical constants, representing the mobility, the mixing coefficient, and the fluid viscosity, respectively. The Ginzburg--Landau potential function  potential $G(\phi):=\frac{1}{4 \epsilon^2}(1-\phi^2)^2$ with $\epsilon$ representing the interfacial width.    
It is well-known that under the mentioned boundary conditions \eqref{model:boundary}, the incompressible CHNS system satisfies the mass conservation law \cite{li2023consistency} and energy dissipation law \cite{li2022fully}, i.e.,
\begin{equation}\label{model:McEc:e1}
	\frac{d \malM(\phi)}{d t}=0, ~~ \frac{d \malE(\phi,  \bm u)}{d t} 	=-M\|\nabla \mu\|^2-\nu\|\nabla  \bm u\|^2 \le  0,
\end{equation}
where 
\begin{equation}\label{model:McEc:e2}
	\malM(\phi) :=\int_{\Omega} \phi d \bm x,~~  
	\malE(\phi,  \bm u) :=\int_{\Omega} \left\{\frac{1}{2}| \bm u|^2+\frac{\lambda}{2}|\nabla  \phi|^2+ \lambda G(\phi)\right\} d  \bm x. 
\end{equation}
Therefore, it is desirable to construct numerical schemes that can preserve such properties in a discrete version of \eqref{model:McEc:e1}.

We end this section by briefly stating the functional setting used throughout the paper, and also introduce some key lemmas. First, for any $m \geq 0$ and $p \ge 1$, we use $W^{m, p}(\Omega)$ to denote the standard Sobolev spaces, and we abbreviate $W^{m, 2}(\Omega)=H^m(\Omega)$ when $p=2$ and $W^{0, p}(\Omega)=L^p(\Omega)$ when $m=0$. The corresponding norms are denoted as $\|\cdot\|_{W^{m, p}}$, $\|\cdot\|_m$ and $\|\cdot\|_{L^p}$. Moreover, denote $(\cdot,\cdot)$ as the scalar or vector $L^2(\Omega)$ inner product, and define $\mbfH^m(\Omega) =(H^m(\Omega))^d$ and $\mbfL^p(\Omega) =(L^p(\Omega))^d$ with $d=2,3$. Furthermore, we introduce the following two vector spaces
\[
\mbfH_0^1(\Omega)=\{\bm v \in \mbfH^1(\Omega): \bm v =\bm 0 \text { on } \partial \Omega\}, ~\mathbf{V}=\{\bm v \in \mbfH_0^1(\Omega): \nabla \cdot  \bm v=0  \text { in } \Omega\}.
\]
We also denote by $L^s(I;W^{m,p}(\Omega))$ the Banach space of all $L^s$ integrable functions from $I:=[a,b]$ into $W^{m,
	p}(\Omega)$ with norm $\|v\|_{L^s(I;W^{m,p}(\Omega))}= \left(\int^b_a \|v\|_{
	W^{m,p}(\Omega)}dt\right)^{\frac{1}{s}}~ \textrm{for}~ s\in
[1,\infty)$ and the standard modification for $s = \infty$.
Similarly, one can define the spaces $W^{k,s}(I;W^{m,p}(\Omega))$. 
The details can be found in \cite{LM72,brenner2008mathematical}.
Next, we revisit the following useful lemmas.
\begin{lemma}[\cite{nochetto2005error}] \label{lem:poncare} 
	If $ \bm v \in \mbfH_0^1(\Omega)$, then    
	there is a positive constant  $c_p$, depending only on $\Omega$, such that
	$$
	\| \bm v\| \leq c_p \|\nabla  \bm v\|   \quad~\text{and}~\quad 	\|\nabla \cdot  \bm v\| \leq\|\nabla  \bm v\|. 
	$$
\end{lemma}
\begin{lemma}[Property of the trilinear form \cite{constantin1988navier}] \label{lem:trilinear} 
	Define the trilinear form 
	$\bm b(\bm u,  \bm v,  \bm w):=(\bm u \cdot \nabla  \bm v,  \bm w)$.  Then 
	\begin{itemize} 
		\item[(i)]  it is skew-symmetric with respect to its last two arguments, i.e.,
		\begin{equation*}
			b(\bm u,  \bm v,  \bm w)=-b(\bm u,  \bm w,  \bm v), ~b(\bm u,  \bm v,  \bm v)=0, \quad  \forall \bm u \in \mbfV, ~  \bm v,  \bm w \in \mbfH^1(\Omega).
		\end{equation*}
		\item[(ii)]
		there is a positive constant  $c_t$, depending only on $\Omega$, such that for $d \le 4$,
		\begin{equation*}
			\bm b( \bm u,  \bm v,  \bm w)   \leq c_t
			\left\{
			\begin{aligned}
				&\|\bm u\|_1 \| \bm v\|_1 \| \bm w\|_1, && \quad \forall  \bm u, \bm v, \bm w \in \mbfH_0^1(\Omega), \\
				&\| \bm u\|_2\| \bm v\|_1\| \bm w\|, && \quad  \forall  \bm u \in \mbfH^2(\Omega) \cap \mbfH_0^1(\Omega), ~  \bm v, \bm w \in \mbfH_0^1(\Omega), \\
				&\| \bm u\|_1\| \bm v\|_2\| \bm w\|, && \quad \forall  \bm v \in \mbfH^2(\Omega) \cap \mbfH_0^1(\Omega), ~ \bm u, \bm w \in \mbfH_0^1(\Omega), \\
				&\| \bm u\|\| \bm v\|_2\| \bm w\|_1, && \quad \forall  \bm v \in \mbfH^2(\Omega) \cap \mbfH_0^1(\Omega), ~  \bm u, \bm w \in \mbfH_0^1(\Omega).
			\end{aligned}
			\right.
		\end{equation*}
 	\end{itemize}
\end{lemma}
\begin{lemma}[Agmon's inequality \cite{temam2012infinite}]  \label{lem:Agmon}
	Define $\bar{v}:=\frac{1}{|\Omega|} \int_{\Omega} v d \bm x$, then there exists a constant $c_{A}$, depending only on $\Omega$, such that
	\begin{equation*}
		\|v-\bar{v}\|_{L^{\infty}} \leq c_{A}\|\nabla v\|^{1/2} \|\Delta v\|^{1/2}. 
	\end{equation*}
\end{lemma}

\section{A class of linear structure-preserving IVS methods} \label{sec:scheme}
In this section, we aim to develop first- and second-order, linear, mass-conserving, and unconditionally energy-stable numerical schemes for the system \eqref{model:phi}--\eqref{model:ivp}. To achieve this, three major challenges must be addressed:
(i) devising an effective strategy to linearize the nonlinear terms;
(ii) ensuring the decoupling of the variables for practical computation; 
and (iii) preserving both mass conservation law and unconditional energy dissipation law at the discrete level.

Given $\beta$ a positive constant, we define $F(\phi) := G(\phi)-\frac{\beta}{2}\phi^2$ and $E_1(\phi):=\int_{\Omega} F(\phi) d \bm x$. 
Moreover, assume that there is a positive constant $\delta_0$ such that $E_1(\phi)+\delta_0>0$. We employ the SAV-ZEC approach \cite{zgdsisc2024} to simultaneously address the nonlinear terms and  coupled terms present in \eqref{model:phi}--\eqref{model:ivp}. To this end, we introduce the scalar auxiliary variable 
$  R(t):= \sqrt{E_1(\phi)+\delta_0}$ and $ \xi(t):=\frac{R(t)}{\sqrt{E_1(\phi)+\delta_0}}\equiv 1$,
and reformulate the original system \eqref{model:phi}--\eqref{model:ivp} into the following equivalent form:
\begin{align}
	& \phi_t+ \xi(t) \nabla \cdot (\bm u \phi) =M \Delta \mu,  \label{model:sav:phi} \\
	& \mu=- \lambda\Delta \phi +\lambda  \beta \phi+ \lambda \xi(t) F^{\prime}(\phi),  \label{model:sav:mu}\\
	& \bm u_t +  \xi(t) \bm u\cdot\nabla \bm u-\nu \Delta \bm u +\nabla p=-  \xi(t) (\phi \nabla \mu + \gamma(t) \nabla p) + \gamma(t)  \nabla p,  \label{model:sav:u}\\
	& \nabla \cdot \bm u = 0,    \label{model:sav:divu} \\
	&  R_t =\frac{1}{2  \sqrt{E_1(\phi)+\delta_0}}   ( F^{\prime}(\phi), \phi_t)
	+  \frac{1}{2 \lambda  \sqrt{E_1(\phi)+\delta_0}} \big((\mu , \nabla \cdot (\bm u \phi) )+  ( {\bm u} ,  \phi  \nabla\mu ) \big) \notag\\
	&\qquad  + \frac{1}{2 \lambda  \sqrt{E_1(\phi)+\delta_0}} \bm b(\bm u, \bm u, \bm u) + \frac{ \gamma(t)}{2 \lambda  \sqrt{E_1(\phi)+\delta_0}}  (\nabla p, \bm u), 
	\label{model:sav:r}
\end{align}
enclosed with boundary and initial conditions \eqref{model:boundary}--\eqref{model:ivp} and $R(0)=\sqrt{E_1(\phi^o)+\delta_0}$.
Here, $\gamma(t)$ is a user-defined, time-dependent function to be given below, and its role and significance shall be illustrated in Remarks \ref{rem:ivs:gamma} and \ref{remk:sub_rate:1}. 

\begin{remark}
	We  point out that some exactly zero terms in the continuous level have been added to \eqref{model:sav:r}, namely, 
	\begin{equation*} 
		(\mu , \nabla \cdot (\bm u \phi) )+  ( {\bm u},  \phi  \nabla\mu ) =0, ~   \bm b(\bm u, \bm u, \bm u) =0, ~ (\nabla p, \bm u) =0.
	\end{equation*}
	This is called  ZEC  (zero energy contribution) in  \cite{chen2021fully,yang2021novelSISC}, which shall  play an important role in developing and analyzing the decoupled and energy stable numerical scheme for the CHNS system. Besides, it is worth noting that \eqref{model:sav:phi}--\eqref{model:sav:r} are more "algorithm-friendly" than the original system \eqref{model:phi}--\eqref{model:div}. This reformulation enables the construction of a linear and fully decoupled structure-preserving numerical scheme. The details of the discretization strategy will be provided in the next subsection.
\end{remark}

\subsection{A first-order numerical scheme} \label{subsec:scheme}
For given positive integer $N_t$, we define a uniform temporal partition $J=\cup_{n=0}^{N_t-1} J_n:=\cup_{n=0}^{N_t-1}[t_{n},t_{n+1}]$ with nodes $t_n=n \tau$ for $n=0, \ldots, N_t$ and $ \tau:=T / N_t$. Denote $\nabla_\tau  v^{n+1}:= v^{n+1}-v^{n}$ and $\p_\tau  v^{n+1}:= \nabla_\tau  v^{n+1}/\tau$.  Let $\psi^n$ be the numerical approximation to the exact function $\psi(t_n)$. With the above notations, a first-order structure-preserving IVS scheme for \eqref{model:sav:phi}--\eqref{model:sav:r} is proposed as the following two steps:

\paragraph{\bf Step 1}  Find $\big(\phi^{n+1}, \mu^{n+1}, \widehat{\bm u}^{n+1}\big)$ such that 
\begin{align}
	& \p_{\tau} \phi^{n+1} + \xi^{n+1}\nabla \cdot({\bm u}^n \phi^n  )=M \Delta \mu^{n+1}, \quad  
	\p_{\bm n} \phi^{n+1}|_{\partial \Omega}=0, \label{scheme:1th:phi}\\
	& \mu^{n+1}=- \lambda \Delta \phi^{n+1} + \lambda \beta \phi^{n+1} + \lambda  \xi^{n+1} F^{\prime}(\phi^n), \quad   
	\p_{\bm n} \mu^{n+1} |_{\partial \Omega}=0, \label{scheme:1th:mu} \\
	& \frac{\widehat {\bm u}^{n+1} -  \bm u^n}{\tau} + \xi^{n+1} {\bm u}^n\cdot\nabla {\bm u}^n - \nu\Delta   \widehat {\bm u}^{n+1} = -\xi^{n+1}  \big( \phi^n \nabla \mu^n+\gamma^n \nabla p^n \big), ~
	\widehat{\bm u}^{n+1}|_{\partial \Omega}=0, \label{scheme:1th:uhat}   
\end{align}
with
\begin{equation} \label{scheme:1th:xi}
	\xi^{n+1}=\frac{R^{n+1}}{\sqrt{E_1(\phi^n)+\delta_0}}.
\end{equation}

\paragraph{\bf Step 2}  Find the SAV $R^{n+1}$ such that 
\begin{equation}\label{scheme:1th:rhat}
	\begin{aligned}
		\p_{\tau}  R^{n+1} 
		&=\frac{1}{2\lambda \sqrt{E_1(\phi^n)+\delta_0}} \big( \lambda (F^{\prime}(\phi^n), \p_\tau  \phi^{n+1} ) + (\mu^{n+1},  \nabla \cdot({\bm u}^n \phi^n)) \big) \\
		&\qquad +\frac{1}{2\lambda \sqrt{E_1(\phi^n)+\delta_0}} \big( (\widehat{\bm u}^{n+1}, \phi^n \nabla \mu^n)
		+ \bm b({\bm u}^n, {\bm u}^n,  \widehat {\bm u}^{n+1}) \big)
		+\frac{\gamma^n}{2\lambda \sqrt{E_1(\phi^n)+\delta_0}}   (\nabla p^n,  \widehat {\bm u}^{n+1}).
	\end{aligned}
\end{equation}

\paragraph{\bf Step 3}  Update $ (\bm u^{n+1}, p^{n+1})$ by the following correction step:
\begin{align}
	& \frac{\bm u^{n+1}- \widehat {\bm u}^{n+1} }{\tau}-\nu \Delta ( \bm u^{n+1}- \widehat {\bm u}^{n+1})+\nabla p^{n+1} = \gamma^n  \nabla p^{n}, \quad  {\bm u}^{n+1}|_{\partial \Omega}=0, \label{scheme:1th:uplus1}   \\
	& \nabla \cdot  \bm u^{n+1} = 0. \label{scheme:1th:div}
\end{align}
Here the parameter $\gamma^n:=\frac{\theta}{\|\nabla p^n \|+1}$
with $\theta$ being a positive artificial parameter, for example, $\theta=1$.

Some remarks are in order.
\begin{remark}\label{rem:e1}   We first explain the strategy behind developing the above scheme.  The  IVS method is used to  decouple the computations of the pressure from the velocity field in the momentum equation. To construct a linear, decoupled, and energy stable scheme, all terms involving $\xi^{n+1}$ are treated  explicitly,  while $\xi^{n+1}$  itself is discretized semi-implicitly. For \eqref{scheme:1th:rhat}, we utilize a combination of implicit and explicit discretization to achieve unconditional energy stability. This approach enables efficient resolution of nonlinearities and coupled terms while preserving a fully decoupled computational framework.
\end{remark}
\begin{remark}  
	Unlike conventional projection methods \cite{chorin1968numerical, temam1969approximation, guermond2006overview}, the IVS method \cite{yakoubi2023enhancing} retains the viscosity term in \eqref{scheme:1th:uplus1}, which preserves the original no-flow boundary condition for the velocity. This is a key advantage over traditional approaches, and it is also crucial for establishing convergence, as discussed in subsection \ref{subsec:full_rate}.
\end{remark}

\begin{remark} \label{rem:ivs:gamma}
	The introduced consistency parameter $\gamma^n$, which only depends on the computed pressure $p^n$, can enhance the accuracy of both the velocity and pressure estimates compared to the traditional VS method. This improvement has been validated through extensive numerical experiments in \cite{yakoubi2023enhancing}. In fact, incorporating $\gamma^n$ in the discretization is a crucial step of our analytical framework, as it leads to the bounded result
	\begin{equation} \label{gamma:bound}
		| \gamma^n \nabla p^n | = \left| \frac{\nabla p^n}{|\nabla p^n| + 1} \right| \leq 1, 
	\end{equation} 
	and thus the optimal-order error estimate, as further detailed in Remark \ref{remk:sub_rate:1}.
\end{remark}

\subsection{ Efficient implementation} \label{subsec:implement}
As pointed out in Remark \ref{rem:e1}, the variables in \eqref{scheme:1th:phi}--\eqref{scheme:1th:uhat} are coupled with the auxiliary variable $R^{n+1}$ through $\xi^{n+1}$ defined in \eqref{scheme:1th:xi}, which may cause huge computational cost by enforcing to solve a large-scale coupled system. 
However, due to the explicit treatment of the nonlinear terms, we can split the unknown variables $\phi^{n+1}, \mu^{n+1}$ and $\widehat{\bm u}^{n+1}$ with respect to $\xi^{n+1}$  as follows:
\begin{equation} \label{scheme:1th:split:all}
	\phi^{n+1}=\phi_0^{n+1}+\xi^{n+1} \phi_1^{n+1}, ~
	\mu^{n+1}=\mu_0^{n+1}+\xi^{n+1} \mu_1^{n+1}, ~ 
	\widehat{\bm u}^{n+1}  = \widehat{\bm u}^{n+1}_{0}+\xi^{n+1} \widehat{\bm u}^{n+1}_{1}. 
\end{equation} 
Then, inserting \eqref{scheme:1th:split:all}  into \eqref{scheme:1th:phi}--\eqref{scheme:1th:uhat}, by collecting terms without and with $\xi^{n+1}$, respectively, the solutions to Step 1 are reduced to find $\{\phi_i^{n+1}, \mu_i^{n+1}, \widehat{\bm u}_i^{n+1}\}_{i=0}^1$ such that

\paragraph{\bf\underline{Step 1.1}} Compute $\{\phi_0^{n+1}, \mu_0^{n+1}\}$ and $\{\phi_1^{n+1}, \mu_1^{n+1}\}$ such that
\begin{equation} \label{scheme:1th:phi0}
	\left\{
	\begin{aligned}
		& \phi_0^{n+1}-\phi^n=M \tau \Delta \mu_0^{n+1}, \quad \mu_0^{n+1}=- \lambda\Delta \phi_0^{n+1} + \lambda \beta \phi_0^{n+1},  \\
		& \p_{\bm n} \phi_0^{n+1}|_{\partial \Omega}=\p_{\bm n} \mu_0^{n+1}|_{\partial \Omega}=0, 
	\end{aligned} \right.
\end{equation}
and
\begin{equation}  \label{scheme:1th:phi1}
	\left\{
	\begin{aligned}
		& \phi_1^{n+1}+\tau \nabla \cdot({\bm u}^n \phi^n )=M \tau \Delta \mu_1^{n+1}, \quad \mu_1^{n+1}=-  \lambda\Delta \phi_1^{n+1} +\lambda \beta \phi_1^{n+1} + \lambda  F^{\prime}(\phi^n), \\
		&   \p_{\bm n} \phi_1^{n+1}|_{\partial \Omega}=\p_{\bm n} \mu_1^{n+1}|_{\partial \Omega}=0.
	\end{aligned} \right.
\end{equation}

\paragraph{\bf \underline{Step 1.2}} Compute $\{\widehat{\bm u}_0^{n+1}, \widehat{\bm u}_1^{n+1}\} $ such that
\begin{align} 
	& \widehat{\bm u}_0^{n+1} - {\bm u}^n -\nu \tau \Delta  \widehat{\bm u}_0^{n+1} = 0,  \quad \widehat{\bm u}_0^{n+1}|_{\partial \Omega}=0, \label{scheme:1th:uhat0}\\
	& \widehat{\bm u}_1^{n+1} +\tau {\bm u}^n \cdot \nabla {\bm u}^n -\nu \tau \Delta \widehat{\bm u}_1^{n+1} =-\tau \phi^n \nabla \mu^n-\tau\gamma^n \nabla p^n,  \quad \widehat{\bm u}_1^{n+1} \big|_{\partial \Omega}=0. \label{scheme:1th:uhat1}
\end{align}

Furthermore, inserting \eqref{scheme:1th:xi} into \eqref{scheme:1th:rhat}, the solution to Step 2 is reduced to find $\xi^{n+1}$ such that
\paragraph{\bf\underline{Step 2.1}} Solve the linear equation about the unknown $\xi^{n+1}$:
\begin{equation} \label{scheme:1th:algebraic_system}
		A \xi^{n+1} =B, 
\end{equation}
where
\begin{equation*}  
	\left\{
	\begin{aligned}
		A &=\sqrt{E_1(\phi^n)+\delta_0}-\frac{1}{2 \lambda   \sqrt{E_1(\phi^n)+\delta_0}}
		\big(\lambda (F^{\prime}(\phi^n), \phi_1^{n+1})+ \tau(\mu_1^{n+1}, \nabla \cdot ( \bm u^n \phi^n))\big) \\
		&\qquad -\frac{\tau}{2\lambda \sqrt{E_1(\phi^n)+\delta_0}}\big( (\widehat{\bm u}_1^{n+1}, \phi^n \nabla \mu^n)+ \bm b({\bm u}^n, {\bm u}^n, \widehat{\bm u}_1^{n+1})+\gamma^n  (\nabla p^n, \widehat{\bm u}_1^{n+1}) \big), \\
		B &= R^n+\frac{1}{2\lambda \sqrt{E_1(\phi^n)+\delta_0}}\lambda( F^{\prime}(\phi^n), \phi_0^{n+1}-\phi^n)\\
		&\qquad + \frac{\tau}{2\lambda \sqrt{E_1(\phi^n)+\delta_0}} \big((\mu_0^{n+1}, \nabla \cdot (  \bm u^n \phi^n))+ ( \widehat{\bm u}_0^{n+1}, \phi^n \nabla \mu^n)\big) \\
		&\qquad +\frac{\tau}{2\lambda \sqrt{E_1(\phi^n)+\delta_0}}\big( \bm b({\bm u}^n, {\bm u}^n, \widehat{\bm u}_0^{n+1})+\gamma^n (\nabla p^n, \widehat{\bm u}_0^{n+1})\big).
	\end{aligned}\right.
\end{equation*}

\paragraph{\bf\underline{Step 2.2}}  Once $\{\phi_i^{n+1}, \mu_i^{n+1}, \widehat{\bm u}_i^{n+1}\}_{i=0}^1$ and $\xi^{n+1}$ are obtained, we first compute $R^{n+1}$ by \eqref{scheme:1th:xi} and then update the variables $\{\phi^{n+1}, \mu^{n+1}, \widehat{\bm u}^{n+1}\}$ via the splitting \eqref{scheme:1th:split:all}.

\paragraph{\bf\underline{Step 3}} With $\widehat{\bm u}^{n+1}$ being computed, we can solve $\bm u^{n+1}$ and $p^{n+1}$ through the correction step \eqref{scheme:1th:uplus1}--\eqref{scheme:1th:div}. 

\begin{remark} \label{rem:solve}
	We remark that combined with suitable spatial discretization methods, at each time level, the total computational cost of fully discrete scheme involves the following processes:
	\begin{itemize}
		\item[(i)] Phase field: Solve two second-order linear systems \eqref{scheme:1th:phi0}--\eqref{scheme:1th:phi1} with the same time-independent constant-coefficient matrix, which can be generated only once and in advance.
		
		\item[(ii)]  Intermediate velocity field: Solve $2d$ Poisson-type scalar equations in \eqref{scheme:1th:uhat0}--\eqref{scheme:1th:uhat1} with the identical time-independent constant-coefficient matrix, which can also be generated only once and in advance.
		
		\item[(iii)] Auxiliary variable: Solve one linear equation \eqref{scheme:1th:algebraic_system} with only one unknown $\xi^{n+1}$ and thus $R^{n+1}$ can be updated by \eqref{scheme:1th:xi} with negligible computational cost.
			
		\item[(iv)] Pressure field and final velocity field: Solve one linear generalized constant-coefficient Stokes problem, that can also be generated only once and in advance in Step 3.
	\end{itemize}
	Therefore, the computational efficiency is totally ensured in practice, especially in large-scale modeling and simulations.
\end{remark}

\begin{remark}
	Although the inclusion of the viscosity term in \eqref{scheme:1th:uplus1} results in a Stokes-like problem, the pressure can be decoupled from the incompressible condition in practical computation \cite{Yang2024error}.
	Specifically, we can compute $p^{n+1}$ via the Poisson equation
	\begin{equation*} 
		-\Delta p^{n+1} =-\gamma^n \Delta p^n - \frac{1}{\tau} \nabla \cdot \widehat{\bm u}^{n+1} 
		+ \nu \nabla \cdot \Delta \widehat{\bm u}^{n+1}, \quad  \p_{\bm n} p^{n+1}|_{\partial \Omega}=0,
	\end{equation*} 
	and then update each component of $\bm u^{n+1}$ through \eqref{scheme:1th:uplus1} by solving some constant-coefficient Poisson equations, which can further increase the computational efficiency. 
\end{remark}

\subsection{Unique solvability, mass conservation and energy stability}
In the following, we first demonstrate that the first-order scheme is uniquely solvable.
\begin{theorem}  \label{thm:solvable}
	The first-order structure-preserving IVS scheme \eqref{scheme:1th:phi}--\eqref{scheme:1th:div} is uniquely solvable.
\end{theorem}
\begin{proof}
	As discussed in Remark \ref{rem:solve}, the linear algebraic equations in \eqref{scheme:1th:phi0}--\eqref{scheme:1th:uhat1} and Step 3 are uniquely solvable, due to the well-behaved coefficient matrix structures. Therefore, it remains to prove the unique solvability of the linear equation given in \eqref{scheme:1th:algebraic_system}, and it suffices to show that $A \neq 0$.
	
	Taking the inner products of the first equation in \eqref{scheme:1th:phi1}, the second equation in \eqref{scheme:1th:phi1} and \eqref{scheme:1th:uhat1}  with $\mu_1^{n+1}, \phi_1^{n+1}$ and $\widehat{\bm u}_1^{n+1}$, respectively, we obtain
	\begin{equation} \label{lem1:eq:1}
		\begin{aligned}
			& \tau(\nabla \cdot(\bm u^n  \phi^n ), \mu_1^{n+1})=-(\phi_1^{n+1}, \mu_1^{n+1})-M \tau\|\nabla \mu_1^{n+1}\|^2, \\
			& \lambda(F^{\prime}(\phi^n), \phi_1^{n+1})=(\mu_1^{n+1}, \phi_1^{n+1})-\lambda\|\nabla \phi_1^{n+1}\|^2- \lambda \beta\|\phi_1^{n+1}\|^2, \\
			& \tau ( \phi^n \nabla \mu^n +\gamma^n\nabla p^n + \bm u^n \cdot \nabla \bm u^n, \widehat{\bm u}_1^{n+1})=-\|\widehat{\bm u}_1^{n+1}\|^2-\nu \tau\|\nabla \widehat{\bm u}_1^{n+1}\|^2.
		\end{aligned}
	\end{equation} 
	Then, substituting \eqref{lem1:eq:1} into the formula for $A$ in \eqref{scheme:1th:algebraic_system}, we deduce
	\begin{equation*}
		\begin{aligned}
			A& = \sqrt{E_1(\phi^n)+\delta_0}+\frac{1}{2\lambda \sqrt{E_1(\phi^n)+\delta_0}}\big(\lambda\|\nabla \phi_1^{n+1}\|^2+\lambda \beta\|\phi_1^{n+1}\|^2\big) \\
			& \qquad +\frac{1}{2  \lambda \sqrt{E_1(\phi^n)+\delta_0}}\big(M\tau\|\nabla \mu_1^{n+1}\|^2+\|\widehat{\bm u}_1^{n+1}\|^2+\nu\tau\|\nabla \widehat{\bm u}_1^{n+1}\|^2\big) >0,
		\end{aligned}
	\end{equation*}
	which implies the desired result.
\end{proof}

Then, we show that the scheme is mass-conserving and unconditionally energy-dissipative at the discrete level, i.e., a discrete version of \eqref{model:McEc:e1} holds.
\begin{theorem} \label{thm:energy}
	The first-order  structure-preserving IVS scheme \eqref{scheme:1th:phi}--\eqref{scheme:1th:div}  satisfies the discrete mass conservation law:
	\begin{equation*}
		(\phi^{n+1},1)=(\phi^n,1), \quad \forall~  0 \le n \le N_t-1.
	\end{equation*}
	Moreover,   the discrete modified energy dissipation law also holds: 
	\begin{align*} 
		\begin{gathered}
			E^{n+1}(\phi,  \bm u, R)-E^n(\phi,  \bm u, R) \leq-  M  \tau\|\nabla \mu^{n+1}\|^2 
			- \frac{\nu \tau}{2}({\|\nabla \widehat{\bm u}^{n+1} \|^2}+\|\nabla  \bm u^{n+1} \|^2),
		\end{gathered}
	\end{align*}
	for $0 \le n \le N_t-1$, where
	\begin{align} \label{energy:mod:1st}
		E^{n+1}(\phi, \bm u, R):=\frac{1}{2} \|\bm u^{n+1}\|^2+ \frac{\lambda}{2}\|\nabla \phi^{n+1}\|^2+\frac{\lambda \beta}{2}\|\phi^{n+1}\|^2+ \lambda |R^{n+1}|^2.
	\end{align}
\end{theorem}
\begin{proof}
	First, multiplying  \eqref{scheme:1th:phi} with $\tau$, and then taking the inner product with 1 and applying the zero boundary conditions, we have
	\begin{equation*}
		(\nabla_\tau  \phi^{n+1},1)=0.  
	\end{equation*}
	Therefore, the discrete mass conservation law is proved.
	
	Next, we proceed to show that the scheme also satisfies the discrete energy dissipation law.
	Taking the inner products of \eqref{scheme:1th:phi}, \eqref{scheme:1th:mu}, \eqref{scheme:1th:uhat} and \eqref{scheme:1th:uplus1} with $\mu^{n+1}$, $\p_\tau \phi^{n+1}$, $ \widehat {\bm u}^{n+1}$ and $ \bm u^{n+1}$, respectively, and multiplying \eqref{scheme:1th:rhat} with $2\lambda  R^{n+1}$, and then adding these results together
	and multiplying it with $\tau$, we have
	\begin{equation*}
		\begin{aligned}
			& \frac{1}{2}(\| \bm u^{n+1}\|^2-\| \bm u^{n}\|^2) + \frac{\lambda}{2}(\|\nabla \phi^{n+1}\|^2- \|\nabla \phi^n\|^2 ) +\frac{\lambda \beta}{2}(\|\phi^{n+1}\|^2- \|\phi^n\|^2 ) + \lambda (| R^{n+1}|^2-|R^n|^2) \\
			&\quad = -M\tau\|\nabla \mu^{n+1}\|^2 - \frac{\nu \tau}{2}(\|\nabla  \bm u^{n+1} \|^2+\|\nabla \widehat {\bm u}^{n+1} \|^2)\\
			& \qquad - \frac{\nu \tau}{2} \|\nabla  (\bm u^{n+1}- \widehat {\bm u}^{n+1})\|^2 - \frac{1}{2}(\|\bm u^{n+1}- \widehat {\bm u}^{n+1}\|^2+\|\widehat {\bm u}^{n+1}- \bm u^n\|^2) \\
			& \qquad -\frac{\lambda}{2}\|\nabla_\tau  (\nabla \phi^{n+1})\|^2 -\frac{\lambda \beta}{2}\|\nabla_\tau  \phi^{n+1}\|^2-\lambda |\nabla_\tau R^{n+1}|^2,
		\end{aligned}
	\end{equation*}
	where  the equality $a(a-b)=\frac{1}{2}(a^2-b^2)+\frac{1}{2}(a-b)^2$  and the zero boundary conditions have been used. Then we complete the proof with the energy defined by \eqref{energy:mod:1st}.
\end{proof}

\begin{remark} \label{err:energy_estimate}
	Notably, in comparison to the discrete modified energy defined in \cite{li2022fully, qian2023fully, li2023consistency, xu2020error}, the developed scheme includes only contributions from the phase-field variable, the velocity, and the auxiliary variable, excluding extraneous terms such as the pressure gradient. 
	Additionally, from the stability result of Theorem \ref{thm:energy},  we immediately deduce that  there exists a positive constant $K_1$ independent of $\tau$ such that
	\begin{equation} \label{bound:solu}
		\|\bm u^{n+1}\|^2 +\|\phi^{n+1}\|_1^2 +|R^{n+1}|^2  +   {\nu \tau\sum_{k=0}^n\|\nabla {\bm u}^{k+1} \|^2}\leq K_1,
	\end{equation}
	which is crucial for performing the subsequent error analysis.
\end{remark}

\subsection{A second-order numerical scheme}\label{sec:scheme:2th}
In this subsection, we extend the developed methodology to second-order accurate method based on the BDF2 formula for the incompressible CHNS model, and demonstrate its mass conservation and unconditional energy stability. 
Denote $w^* :=2w^n-w^{n-1}$ as the second-order linear extrapolation and let $\malD_2 w^n := (3 w^{n+1}-4w^n +w^{n-1})/{2\tau}$ for $n \ge 2$. 

Then, the second-order  IVS scheme for the CHNS system is proposed as follows:
For $n=1$, we use the first-order scheme to compute $\big(\phi^{1}, \mu^{1}, \widehat{\bm {u}}^{1},  \bm u^{1}, p^{1},  R^{1}\big)$; and for $n \ge 2$, given $\big(\phi^{\ell}, \mu^{\ell}, \widehat{\bm {u}}^{\ell}, \bm u^{\ell}, p^{\ell}, R^{\ell}\big)$ ($\ell=n-1, n$), we solve $\big(\phi^{n+1}, \mu^{n+1}, \widehat{\bm {u}}^{n+1}, \bm u^{n+1}, p^{n+1}, R^{n+1} \big)$ via 
\paragraph{\bf Step 1}  Find $\big(\phi^{n+1}, \mu^{n+1}, \widehat{\bm u}^{n+1}\big)$ such that 
\begin{align}
	& \malD_2 \phi^{n+1}  + \xi^{n+1}\nabla \cdot({\bm u}^* \phi^*  )=M \Delta \mu^{n+1}, \quad  
	\p_{\bm n} \phi^{n+1}|_{\partial \Omega}=0, \label{scheme:phi:2th}\\
	& \mu^{n+1}=- \lambda \Delta \phi^{n+1} + \lambda \beta \phi^{n+1} + \lambda  \xi^{n+1} F^{\prime}(\phi^*), \quad   
	\p_{\bm n} \mu^{n+1} |_{\partial \Omega}=0, \label{scheme:mu:2th} \\
	& \widehat{\malD}_2 {\bm u}^{n+1}  + \xi^{n+1} {\bm u}^*\cdot\nabla {\bm u}^* - \nu\Delta   \widehat {\bm u}^{n+1} = -\xi^{n+1}  (\phi^* \nabla \mu^* + \gamma^* \nabla p^*)  , ~
	\widehat{\bm u}^{n+1}|_{\partial \Omega}=0, \label{scheme:uhat:2th}   
\end{align}
with
\begin{equation*} 
	\xi^{n+1}=\frac{R^{n+1}}{\sqrt{E_1(\phi^*)+\delta_0}}, \quad \widehat{\malD}_2 {\bm u}^{n+1} = \frac{3 \widehat{\bm u}^{n+1}-4{\bm u}^{n} +{\bm u}^{n-1}}{2\tau}.
\end{equation*}

\paragraph{\bf Step 2}  Find the SAV $R^{n+1}$ such that 
\begin{equation}\label{scheme:rhat:2th}
	\begin{aligned}
		\malD_2 R^{n+1}
		&=\frac{1}{2\lambda \sqrt{E_1(\phi^*)+\delta_0}} \big( \lambda (F^{\prime}(\phi^*), \malD_2 \phi^{n+1} ) + (\mu^{n+1},  \nabla \cdot({\bm u}^* \phi^*)) \big) \\
		&\qquad +\frac{1}{2\lambda \sqrt{E_1(\phi^*)+\delta_0}} \big( ( \widehat{\bm u}^{n+1}, \phi^* \nabla \mu^*)
		+ \bm b({\bm u}^*, {\bm u}^*,  \widehat{\bm u}^{n+1}) \big) 
		  +\frac{\gamma^*}{2\lambda \sqrt{E_1(\phi^*)+\delta_0}}   (\nabla p^*,  \widehat {\bm u}^{n+1}).
	\end{aligned}
\end{equation}

\paragraph{\bf Step 3}  Update $ (\bm u^{n+1}, p^{n+1})$ by the following correction step:
\begin{align}
	&  \frac{3}{2\tau}({\bm u}^{n+1}- \widehat {\bm u}^{n+1}) -\nu \Delta ( \bm u^{n+1}- \widehat {\bm u}^{n+1})+\nabla p^{n+1} = \gamma^*  \nabla p^{*}, \quad  {\bm u}^{n+1}|_{\partial \Omega}=0, \label{scheme:uplus1:2th}   \\
	& \nabla \cdot  \bm u^{n+1} = 0. \label{scheme:div:2th}
\end{align}

\begin{remark} Analogous to the first-order method, the second-order scheme also permits efficient implementation by solving only constant-coefficient linear systems. In fact, through substitution of \eqref{scheme:1th:split:all} into \eqref{scheme:phi:2th}–\eqref{scheme:uhat:2th}, segregating terms by their $\xi_1^{n+1}$- and $\xi_2^{n+1}$-dependence,  we can obtain constant-coefficient linear systems
	for $\{\phi_i^{n+1}, \mu_i^{n+1}\}_{i=0}^1$ that similar to \eqref{scheme:1th:phi0}--\eqref{scheme:1th:phi1}, for $\{\widehat{\bm u}_i^{n+1}\}_{i=0}^1$ that similar to \eqref{scheme:1th:uhat0}--\eqref{scheme:1th:uhat1}, and a linear equation for
	$\xi^{n+1}$ that similar to \eqref{scheme:1th:algebraic_system}.
\end{remark}
\begin{remark}  Following the same process as in Theorem \ref{thm:solvable}, one can also show that the second-order scheme \eqref{scheme:phi:2th}--\eqref{scheme:div:2th}  admits unique solutions. We omit the detailed proof for the sake of brevity.
\end{remark}

Next, we shall show that the second-order  scheme is also mass-conserving and unconditionally energy stable. 
\begin{theorem} \label{thm:energy:2th}
	The second-order  structure-preserving  scheme \eqref{scheme:phi:2th}--\eqref{scheme:div:2th}  satisfies the discrete mass conservation law:
	\begin{equation*}
		(\phi^{n+1},1)=(\phi^n,1), \quad \forall~  0 \le n \le N_t-1.
	\end{equation*}
	Moreover,   the discrete modified energy dissipation law also holds: 
	\begin{equation*}
		\widetilde{E}^{n+1}(\phi,  \bm u, R) - \widetilde{E}^{n}(\phi,  \bm u, R) \leq -M\tau \|\nabla \mu^{n+1}\|^2 -\f{\nu \tau}{6} \big(\|\nabla({\bm u}^{n+1}- \widehat{\bm u}^{n+1})\|^2 + \|\nabla \widehat{\bm u}^{n+1}\|^2\big),
	\end{equation*}
	for $0 \le n \le N_t-1$, where
	\begin{equation}\label{energy:mod:2th}
		\begin{aligned}
			\widetilde{E}^{n+1}(\phi,  \bm u, R) 
			& := \frac{1}{4}\left(\|\bm u^{n+1}\|^2+\|\bm u^{*}\|^2 \right)
			+\frac{\lambda}{4}\left(\|\nabla \phi^{n+1}\|^2+\|\nabla \phi^{*}\|^2 \right)
			+\frac{\lambda\beta}{4} \left(\|\phi^{n+1}\|^2+ \|\phi^{*}\|^2 \right)  \\
			&\qquad   + \frac{\lambda}{2}\left( |R^{n+1}|^2   +  |R^{*}|^2 \right)
			+  \frac{\nu \tau}{6} \left\|\nabla \bm u^{n+1}\right\|^2.
		\end{aligned}
	\end{equation}
\end{theorem}
\begin{proof}
	Thanks to Theorem \ref{thm:energy}, we have	$(\phi^1,1) = (\phi^0,1)$. Then for $n\geq 1$, 
	multiplying \eqref{scheme:phi:2th} with $2\tau$, and then taking the inner product with 1 and applying the zero boundary conditions and the incompressibility condition, we have
	\begin{equation*}
		(3 \phi^{n+1}-4\phi^{n}+\phi^{n-1},1)=0,
	\end{equation*}
	which implies the discrete mass conservation law.
	
	Next, we proceed to show that the scheme is unconditionally stable.
	Taking the inner products of \eqref{scheme:phi:2th}--\eqref{scheme:uhat:2th} and \eqref{scheme:uplus1:2th} with $2\tau \mu^{n+1}$, $2\tau \malD_2 \phi^{n+1}$, $2\tau \widehat{\bm u}^{n+1}$ and $2\tau  {\bm u}^{n+1}$ respectively, and multiplying \eqref{scheme:rhat:2th}  with $4\lambda R^{n+1}$, then adding these results together yields 
	\begin{equation} \label{stable:2th:CH}
		\begin{aligned}
			& \frac{\lambda}{2}\big(\|\nabla \phi^{n+1}\|^2 - \|\nabla \phi^{n}\|^2 + \|2\nabla \phi^{n+1}-\nabla \phi^n\|^2-\|2\nabla \phi^n-\nabla \phi^{n-1}\|^2 \big) + \frac{\lambda}{2} \|\nabla_\tau (\nabla \phi^{n+1} - \nabla \phi^n)\|^2 \\
			& \qquad + \frac{\lambda \beta}{2}\big(\|\phi^{n+1}\|^2  -\|\phi^{n}\|^2+ \|2 \phi^{n+1}-\phi^n\|^2-\|2\phi^n-\phi^{n-1}\|^2 \big) +  \frac{\lambda \beta }{2} \|\nabla_\tau (\phi^{n+1} -  \phi^n)\|^2 \\ 
			& \qquad + \frac{1}{2}\Big(\|{\bm u}^{n+1}\|^2-\|{\bm u}^{n}\|^2  + \|2\widehat{\bm u}^{n+1} - {\bm u}^{n}\|^2-\|2\bm u^{n} - {\bm u}^{n-1}\|^2 \Big) + \frac{1}{2} \|{\bm u}^{n+1}- \widehat{\bm u}^{n+1}\|^2  \\
			& \qquad + \frac{1}{2} \|\widehat{\bm u}^{n+1} - 2 {\bm u}^{n}+{\bm u}^{n-1}  \|^2  + \f{\nu \tau}{3} \big(\| \nabla {\bm u}^{n+1}\|^2   + \| \nabla  {\bm u}^{n+1} - \nabla \widehat{\bm u}^{n+1}\|^2  \big)   \\
			& \qquad + \lambda \big( |R^{n+1}|^2- |R^{n}|^2+|2R^{n+1}-R^n|^2 -|2R^{n}-R^{n-1}|^2 \big) + \lambda |\nabla_\tau (R^{n+1} -  R^n)|^2 \\
			& \quad = -2 M \tau \|\nabla \mu^{n+1}\|^2- \frac{5\nu \tau}{3}   \| \nabla \widehat{\bm u}^{n+1}\|^2, 
		\end{aligned}
	\end{equation}
	where we have utilized the equality $2a(3 a-4 b+c)=|a|^2+|2 a-b|^2-|b|^2-|2 b-c|^2+|a-2 b+c|^2.$

	Furthermore, thanks to \eqref{scheme:uplus1:2th}, we have
	\begin{equation}\label{stable:2th:u_uhat}
		2 {\bm u}^{n+1}-{\bm u}^n-\f{4\nu \tau}{3} \Delta {\bm u}^{n+1}=2 \widehat{\bm u}^{n+1}-{\bm u}^n-\f{4\nu \tau}{3} \Delta \widehat{\bm u}^{n+1}-\frac{4 \tau}{3} \nabla p^{n+1}+ \frac{4 \tau}{3} \gamma^{*}\nabla p^*.
	\end{equation}
	Taking the inner product of \eqref{stable:2th:u_uhat} with  $2 {\bm u}^{n+1}-{\bm u}^n$ and utilizing \eqref{scheme:div:2th} gives
	\begin{equation*}
		\begin{aligned}
			& \|2 {\bm u}^{n+1}-{\bm u}^n\|^2+\f{2\nu \tau}{3} \|2\nabla {\bm u}^{n+1}-\nabla {\bm u}^n\|^2 \\
			& \quad = (2 \widehat{{\bm u}}^{n+1}-{\bm u}^n, 2 {\bm u}^{n+1}-{\bm u}^n )+\f{2\nu \tau}{3} (\nabla (2\widehat{\bm u}^{n+1} -{\bm u}^{n} ), \nabla (2 {\bm u}^{n+1}-{\bm u}^n ))  \\
			& \quad \leq \frac{1}{2} \|2 \widehat{{\bm u}}^{n+1}-{\bm u}^n\|^2 + \frac{1}{2}\| 2 {\bm u}^{n+1}-{\bm u}^n \|^2 +\f{\nu \tau}{6} \|\nabla (2\widehat{\bm u}^{n+1} -{\bm u}^{n}) \|^2 + \f{2\nu \tau}{3} \|\nabla (2 {\bm u}^{n+1}-{\bm u}^n) \|^2,
		\end{aligned}
	\end{equation*}
	which can be simplified into
	\begin{equation}  \label{stabel:2th:uplus:2}
		\begin{aligned}
			\frac{1}{2}\|2 {\bm u}^{n+1}-{\bm u}^n\|^2  
			& \leq \frac{1}{2} \|2 \widehat{{\bm u}}^{n+1}-{\bm u}^n\|^2 +\f{\nu \tau}{6} \|\nabla (2\widehat{\bm u}^{n+1} -{\bm u}^{n}) \|^2 \\
			& \leq \frac{1}{2} \|2 \widehat{{\bm u}}^{n+1}-{\bm u}^n\|^2 + \f{4\nu \tau}{3} \|\nabla  \widehat{\bm u}^{n+1}\|^2 +\f{\nu \tau}{3}\|\nabla {\bm u}^{n} \|^2.
		\end{aligned}      
	\end{equation}
	
	Combing \eqref{stable:2th:CH} and \eqref{stabel:2th:uplus:2}, we have
	\begin{equation*}  
		\begin{aligned}
			& \frac{\lambda}{2}\big(\|\nabla \phi^{n+1}\|^2 -\|\nabla \phi^{n}\|^2+ \|2\nabla \phi^{n+1}-\nabla \phi^n\|^2-\|2\nabla \phi^n-\nabla \phi^{n-1}\|^2 \big) + \frac{\lambda}{2} \|\nabla_\tau (\nabla \phi^{n+1} - \nabla \phi^n)\|^2\\
			& \qquad + \frac{\lambda \beta}{2}\big(\|\phi^{n+1}\|^2-\|\phi^{n}\|^2 + \|2 \phi^{n+1}-\phi^n\|^2-\|2\phi^n-\phi^{n-1}\|^2 \big) +  \frac{\lambda \beta }{2} \|\nabla_\tau (\phi^{n+1} -  \phi^n)\|^2 \\ 
			& \qquad + \frac{1}{2}\Big(\|{\bm u}^{n+1}\|^2-\|{\bm u}^{n}\|^2  + \|2\bm u^{n+1} - {\bm u}^{n}\|^2-\|2\bm u^{n} - {\bm u}^{n-1}\|^2 \Big) + \frac{1}{2} \|{\bm u}^{n+1}- \widehat{\bm u}^{n+1}\|^2  \\
			& \qquad + \frac{1}{2} \|\widehat{\bm u}^{n+1} - 2 {\bm u}^{n}+{\bm u}^{n-1}  \|^2 + \f{\nu \tau}{3}\big( \|\nabla {\bm u}^{n+1} \|^2-\|\nabla {\bm u}^{n} \|^2\big)  \\
			& \qquad + \lambda \big( |R^{n+1}|^2- |R^{n}|^2+|2R^{n+1}-R^n|^2 -|2R^{n}-R^{n-1}|^2 \big)+ \lambda |\nabla_\tau (R^{n+1} -  R^n)|^2 \\
			& \quad = -2 M \tau \|\nabla \mu^{n+1}\|^2 -\f{\nu \tau}{3} \big(\|\nabla({\bm u}^{n+1}- \widehat{\bm u}^{n+1})\|^2 + \|\nabla \widehat{\bm u}^{n+1}\|^2\big),
		\end{aligned}
	\end{equation*}
	which completes the proof.
\end{proof}
	
\section{Error estimates}\label{sec:error}
In this section, we focus on the error estimates for the structure-preserving IVS scheme \eqref{scheme:1th:phi}--\eqref{scheme:1th:div} in multiple space dimensions. For simplicity, we shall drop the dependence on $\bm x$ for all functions when there is no confusion.

Let $\{ \phi(t_{n+1}), \mu(t_{n+1}), \bm u(t_{n+1}),p(t_{n+1}), R(t_{n+1})\}$ be the exact solutions of \eqref{model:sav:phi}--\eqref{model:sav:r} at $t=t_{n+1}$.
Denote
\begin{equation*}  
	\begin{aligned}
		& e_\phi^{n+1}= \phi(t_{n+1})-\phi^{n+1}, \quad 
		e_\mu^{n+1}=\mu(t_{n+1})-\mu^{n+1}, \quad e_R^{n+1}=R(t_{n+1})- R^{n+1},\\ 
		& \widehat e_{\bm u}^{n+1}=\bm u(t_{n+1})-\widehat {\bm 
			u}^{n+1}, \quad  
		e_{\bm u}^{n+1}= \bm u(t_{n+1})-\bm u^{n+1}, \quad 
		e_p^{n+1}=p(t_{n+1})-p^{n+1}. 
	\end{aligned}
\end{equation*}
Besides, we assume that the exact solutions  are smooth enough such that
\begin{equation} \label{assump:e1}
	\begin{aligned}
		& \phi \in H^{2}(0,T; L^2(\Omega)) \cap W^{1, \infty}(0,T; H^1(\Omega)) \cap L^{\infty}(0,T; H^2(\Omega)), ~  \mu \in W^{1, \infty}(0,T; H^2(\Omega)),\\
		& \bm u \in H^{2}(0,T; \mbfL^2(\Omega)) \cap W^{1, \infty}(0,T;\mbfH^2(\Omega)\cap   \mbfH_0^1(\Omega)), ~ p \in W^{1, \infty}(0,T; L^2(\Omega))\cap L^{\infty}(0,T;H^1(\Omega)).
	\end{aligned}  
\end{equation}
Meanwhile, we assume that there exists a positive constant $K_{\bm u}$ such that
\begin{equation} \label{assump:e2}
	\max_{t \in J} \|\bm u(t)\|_{\mbfH^2(\Omega)} \leq K_{\bm u}.
\end{equation}
Moreover, throughout the paper, we use $C$ with or without subscript, to denote a general positive constant that may depend on the norms of the exact solutions and have different values in different circumstances, but is independent of the mesh size.

Based on the stability result \eqref{bound:solu}, the following $H^2$-bound for $\phi^n$ and $L^2$-bound for $\mu^n$ can be similarly proved as Ref. \cite{li2023global}, without the Lipschitz continuous assumption on the nonlinear potential function $G(\phi)$. 

\begin{lemma} \label{lem:laplace_phi}
	Assume that $(\phi^{n+1}, \mu^{n+1})$ are the solutions of \eqref{scheme:1th:phi}--\eqref{scheme:1th:mu}. 
	Then, there exists a positive constant $K_2$ independent of  $\tau$ such that 
	\begin{equation*} 
		\|\Delta \phi^{n+1}\|+\|\mu^{n+1}\| \leq  K_2, \quad 0 \leq n \leq N_t-1.
	\end{equation*}
\end{lemma}
Lemma \ref{lem:laplace_phi} combined with Theorem \ref{thm:energy} and Lemma \ref{lem:Agmon} implies that 
\begin{equation} \label{bound_phi}
	\|\phi^n\|_{L^\infty} \le C_{\phi}, \quad 0 \leq n \le N_t.
\end{equation}
This conclusion shall play a key role in the latter analysis. 

Due to the strong nonlinearity and coupling of the CHNS system,  it turns out to be difficult to directly present optimal-order error estimates. Instead, the analysis will be divided into several parts based upon the assumption that 
\begin{equation} \label{hypothesis:n:1}
	\|e_{\bm u}^{m}\|_2 \le 1 \quad \text{for sufficiently small } \tau,
\end{equation}     
which further means that  $\|\bm u^{m}\|_2 \le 1+K_{\bm u}$
for $m=0,1,\dots, N_t$, due to \eqref{assump:e2}.
This hypothesis shall be validated using an induction process based on a bootstrap argument. Firstly, it is straightforward to see that \eqref{hypothesis:n:1} holds for $m = 0$. Below we suppose \eqref{hypothesis:n:1} holds for all $m \le n \le N_t-1$,
and we shall complete the induction assumption in subsection \ref{subec:induct} for $m=n+1$.

\subsection{Rough error estimates} \label{subsec:reduce_rate}
In this subsection, we first present a sub-optimal convergence result for the phase-field variable, chemical potential, and velocity under the regularity assumption \eqref{assump:e1}. This intermediate result will serve as a foundation for realizing improved optimal-order error estimates in subsection \ref{subec:induct}.
\begin{theorem} \label{err:sub_rate}
	Under the assumption \eqref{assump:e1}, 
	there exists a sufficiently small constant $\tau_0$, such that for $\tau \le \tau_0$, we have
	\begin{equation*} 
		\|e_\phi^{n+1}\|_1^2+|e_R^{n+1}|^2 +\|e_{\bm u}^{n+1}\|^2+ \tau (\|e_\mu^{n+1}\|_1^2 + \nu \|\nabla e_{\bm u}^{n+1}\|^2) + \sum_{m=0}^{n} \|e_{\bm u}^{m+1}-\widehat{e}_{\bm u}^{m+1}\|^2 		\leq C  \tau,
	\end{equation*}
	where the constant $C$ is independent of $\tau$ and $n$, but depends on $T$.
\end{theorem}
The proof of Theorem \ref{err:sub_rate} is split into several parts. We will address each part step by step.

\paragraph{\bf \indent Part I. Estimates for phase-field equation}
In the first part, we present an estimate for the phase-field numerical scheme \eqref{scheme:1th:phi}--\eqref{scheme:1th:mu}.

\begin{lemma} \label{lem:err:CH}
	Under the assumption \eqref{assump:e1}, we have
	\begin{equation*}  
		\begin{aligned}
			&\frac{\lambda }{2}\p_\tau \|\nabla e_\phi^{n+1}\|^2+\frac{1+\beta}{2}\p_\tau \|e_\phi^{n+1}\|^2
			+\frac{3M}{4}\|e_\mu^{n+1}\|_1^2 \\
			& \leq  -\lambda  \frac{e_R^{n+1}}{\sqrt{E_1(\phi^{n})+\delta_0}} \big(F^{\prime}(\phi^{n}),\p_\tau  e_\phi^{n+1}\big)
			+C|e_R^{n+1}|^2 + C\|e_\phi^{n+1}\|_1^2 + C\|e_\phi^n\|_1^2			\\
			& \quad    +C\|e_{\bm u}^n\|^2+C \tau \big(\|\bm u_t\|_{L^2(J_n ; \mbfL^2(\Omega))}^2+\|\phi_t\|_{L^2(J_n ; H^1(\Omega))}^2+\|\phi_{tt}\|_{L^2(J_n ; L^2(\Omega))}^2\big).
		\end{aligned}
	\end{equation*}
\end{lemma}
\begin{proof} First, subtracting \eqref{scheme:1th:phi} from \eqref{model:sav:phi} at $t=t_{n+1}$, we obtain the error equation
	\begin{equation} \label{err:ephi}
		\p_\tau  e_\phi^{n+1}-M \Delta e_\mu^{n+1}=R_\phi^{n+1}+E_N^{n+1},
	\end{equation}
	where  
	\begin{equation*} 
		\begin{aligned}
			& R_\phi^{n+1}:=\p_\tau \phi(t_{n+1})- \phi_t(t_{n+1})=\frac{1}{\tau} \int_{t_n}^{t_{n+1}}(t-t_n) \phi_{tt} d t, \\
			& E_N^{n+1}:=\frac{R^{n+1}}{\sqrt{E_1(\phi^n)+\delta_0}} \nabla \cdot({\bm u}^n \phi^n) -\frac{R(t_{n+1})}{\sqrt{E_1(\phi(t_{n+1}))+\delta_0}}\nabla \cdot  (\bm u(t_{n+1})\phi(t_{n+1})).
		\end{aligned} 
	\end{equation*}
	Taking the inner products of \eqref{err:ephi} with $e_\mu^{n+1}$ and $e_\phi^{n+1}$, respectively, we have
	\begin{equation} \label{err:ephi:1}
		\big(\p_\tau  e_\phi^{n+1}, e_\mu^{n+1}\big)+M\|\nabla e_\mu^{n+1}\|^2= (R_\phi^{n+1}, e_\mu^{n+1})+(E_N^{n+1}, e_\mu^{n+1}),
	\end{equation}
	\begin{equation} \label{err:ephi:2}
		\frac{1}{2}\p_\tau \| e_\phi^{n+1}\|^2+ \frac{\tau}{2}\|\p_\tau  e_\phi^{n+1}\|^2=(R_\phi^{n+1}, e_\phi^{n+1})+(E_N^{n+1}, e_\phi^{n+1})-M(\nabla e_\mu^{n+1}, \nabla e_\phi^{n+1}).
	\end{equation}
	
	Next, subtracting \eqref{scheme:1th:mu} from \eqref{model:sav:mu} at $t=t_{n+1}$, we get
	\begin{equation}  \label{err:emu}
		e_\mu^{n+1}=- \lambda\Delta e_\phi^{n+1}+ \lambda \beta e_\phi^{n+1} + \lambda \frac{e_R^{n+1}}{\sqrt{E_1(\phi^{n}))+\delta_0}} F^{\prime}(\phi^{n})+E_F^{n+1},
	\end{equation} 
	where  
	\begin{equation*}  
		E_F^{n+1}:=\lambda  \Big(\frac{F^{\prime}(\phi(t_{n+1}))}{\sqrt{E_1(\phi(t_{n+1}))+\delta_0}}-\frac{F^{\prime}(\phi^n)}{\sqrt{E_1(\phi^n)+\delta_0}}\Big) R(t_{n+1}).
	\end{equation*} 
	Taking the inner products of \eqref{err:emu} with $M e_\mu^{n+1}$ and $\p_\tau e_\phi^{n+1}$, respectively, we obtain
	\begin{equation} \label{err:emu:1}
		\begin{aligned}
			M\|e_\mu^{n+1}\|^2 
			& =    M \lambda(\nabla e_\mu^{n+1}, \nabla e_\phi^{n+1}) +M  \lambda 
			\beta(e_\phi^{n+1},e_\mu^{n+1}) \\
			&\qquad +M \lambda  \frac{ e_R^{n+1}}{\sqrt{E_1(\phi^{n})+\delta_0}}(F^{\prime}(\phi^{n}), e_\mu^{n+1})+M(E_F^{n+1}, e_\mu^{n+1}),
		\end{aligned}
	\end{equation}
	\begin{equation} \label{err:emu:2}
		\begin{aligned}
			\big( e_\mu^{n+1},\p_\tau  e_\phi^{n+1}\big)
			& =  \frac{\lambda }{2}\big( \p_\tau \|\nabla e_\phi^{n+1}\|^2 + \tau\|\p_\tau  \nabla e_\phi^{n+1}\|^2\big)  
			+  \frac{\lambda \beta}{2}\big( \p_\tau\| e_\phi^{n+1}\|^2 + \tau\|\p_\tau   e_\phi^{n+1}\|^2\big)\\
			& \qquad  +\big(E_F^{n+1}, \p_\tau  e_\phi^{n+1}\big) + \lambda \frac{e_R^{n+1}}{\sqrt{E_1(\phi^{n})+\delta_0}} \big(F^{\prime}(\phi^{n}),\p_\tau  e_\phi^{n+1}\big)   .
		\end{aligned}
	\end{equation}

	Now, adding \eqref{err:ephi:1}--\eqref{err:ephi:2} and \eqref{err:emu:1}--\eqref{err:emu:2} together, we obtain
	\begin{equation} \label{err:CH:all:e1}
		\begin{aligned}
			&\frac{\lambda }{2}  \p_\tau \|\nabla e_\phi^{n+1}\|^2 
			+\frac{\lambda (1+\beta)}{2}  \p_\tau\| e_\phi^{n+1}\|^2 
			+ M\| e_\mu^{n+1}\|_1^2 \\
			&  \le - \lambda \frac{e_R^{n+1}}{\sqrt{E_1(\phi^{n})+\delta_0}} \big(F^{\prime}(\phi^{n}),\p_\tau  e_\phi^{n+1}\big)+\lambda M   \frac{e_R^{n+1}}{\sqrt{E_1(\phi^{n})+\delta_0}}(F^{\prime}(\phi^{n}), e_\mu^{n+1}) \\
			& \quad +\big(E_F^{n+1}, M e_\mu^{n+1}-\p_\tau  e_\phi^{n+1}\big)+ (E_N^{n+1}, e_\mu^{n+1}+e_\phi^{n+1}) \\
			& \quad  +(R_\phi^{n+1}, e_\mu^{n+1})+(R_\phi^{n+1}, e_\phi^{n+1})+ \lambda M\beta(e_\phi^{n+1},e_\mu^{n+1})
			=:\sum\limits_{i=1}^7\malA_i.
		\end{aligned}
	\end{equation}
	
	We pay attention to the estimates of the right-hand side of \eqref{err:CH:all:e1}. First, the second term can be bounded by
	\begin{equation*} 
		\malA_2\leq   \frac{M}{16}\|e_\mu^{n+1}\|^2+C|e_R^{n+1}|^2,
	\end{equation*}
	where the uniform boundedness of $\phi^n$ is utilized. Second, the last three terms can be estimated in a standard way that 
	\begin{equation*} 
		\malA_{5} + \malA_{6} +  \malA_{7}
		\le \frac{M}{16}\|e_\mu^{n+1}\|^2+ C\| e_\phi^{n+1}\|^2+ C\tau \|\phi_{tt}\|_{L^2(J_n ; L^2(\Omega))}^2,
	\end{equation*}
	where	$\|R_\phi^{n+1}\|^2 \le \tau \|\phi_{tt}\|_{L^2(J_n ; L^2(\Omega))}^2$ 	is applied.
	
	Third, owing to \eqref{err:ephi}, the third term can be rewritten as
	\begin{equation*}
		\begin{aligned}
			\malA_3
			&=\big(E_F^{n+1}, M e_\mu^{n+1}-M \Delta e_\mu^{n+1}-R_\phi^{n+1}-E_N^{n+1}\big) \\
			& =\big(E_F^{n+1}, M e_\mu^{n+1}-R_\phi^{n+1}-E_N^{n+1}\big)+M\big(\nabla E_F^{n+1}, \nabla e_\mu^{n+1}\big). 
		\end{aligned}
	\end{equation*}
	Note that $E_N^{n+1}$ is equivalent to  
	\begin{equation} \label{err:EN:all}
		\begin{aligned}
			E_N^{n+1} 
			& =  -\frac{R(t_{n+1})}{\sqrt{E_1(\phi(t_{n+1}))+\delta_0}}  
			\nabla \cdot\Big( \phi(t_n)\int_{t_n}^{t_{n+1}} \bm u_t dt+\bm u (t_{n+1}) \int_{t_n}^{t_{n+1}} \phi_t dt \Big) \\
			& \qquad  + \malE_R^{n+1} \nabla \cdot{\left(\bm u (t_{n}) \phi(t_{n})\right)} -\frac{R^{n+1}}{\sqrt{E_1(\phi^n)+\delta_0}} \nabla \cdot\left(\bm u (t_{n}) e_\phi^n+e_{\bm u }^n \phi^n\right),
		\end{aligned}
	\end{equation}
	where 
	\begin{equation} \label{err:Er}
		\malE_R^{n+1}:=\frac{R^{n+1}}{\sqrt{E_1(\phi^n)+\delta_0}}-\frac{R(t_{n+1})}{\sqrt{E_1(\phi(t_{n+1}))+\delta_0}} 
		\le C|e_R^{n+1}|+ C|e_\phi^n|  +C\int_{t_n}^{t_{n+1}}|\phi_t| dt.
	\end{equation}
	Besides, the boundedness of $\phi(t_{n+1})$ and $\phi^n$ imply that
	\begin{equation}  \label{err:E_F}
		\|E_F^{n+1}\|_1^2 \le		C\|e_\phi^n\|_1^2  +C \tau \|\phi_t\|^2_{L^2(J_n ; H^1(\Omega))}.
	\end{equation}
	Therefore, using the integration by parts formula and \eqref{err:EN:all}--\eqref{err:E_F}, we see 
	\begin{equation} \label{err:EN:EF}
		\begin{aligned}
			 (E_F^{n+1}, E_N^{n+1}) 	& = \frac{R(t_{n+1})}{\sqrt{E_1(\phi(t_{n+1}))+\delta_0}}
			\big(\nabla E_F^{n+1},  \phi(t_n)\int_{t_n}^{t_{n+1}} \bm u_t dt+\bm u (t_{n+1}) \int_{t_n}^{t_{n+1}} \phi_t dt\big) \\
			& \quad - \malE_R^{n+1} \big(\nabla E_F^{n+1},  \bm u(t_{n}) \phi(t_{n})\big) 
			+\frac{R^{n+1}}{\sqrt{E_1(\phi^n)+\delta_0}}\big(\nabla E_F^{n+1}, \bm u(t_{n}) e_\phi^n +e_{\bm u}^n \phi^n\big)\\
			&  \leq C|e_R^{n+1}|^2 +C\|e_\phi^n\|_1^2+C\|e_{\bm u}^n\|^2
			+C\tau\big( \|\phi_t\|_{L^2(J_n ; H^1(\Omega))}^2 + \|\bm u_t\|_{L^2(J_n; \mbfL^2(\Omega))}^2 \big),
		\end{aligned}
	\end{equation}
	where the boundedness of \eqref{bound:solu}  and \eqref{bound_phi} are also applied in the last step.
	
	Thus, combing \eqref{err:E_F} and \eqref{err:EN:EF} together and applying the Cauchy-Schwarz inequality, it yields
	\begin{equation*} 
		\begin{aligned}
			\malA_{3} 
			&  \leq C|e_R^{n+1}|^2+\frac{M}{8}\|e_\mu^{n+1}\|_1^2
			+C\|e_{\bm u}^n\|^2+C\|e_\phi^n\|_1^2\\ 
			& \quad +C \tau \big(\|\phi_t\|_{L^2(J_n; H^1(\Omega))}^2
			+\|\phi_{tt}\|_{L^2(J_n; L^2(\Omega))}^2+\|\bm u_t\|_{L^2(J_n; \mbfL^2(\Omega))}^2\big).
		\end{aligned}
	\end{equation*}
	
	Finally, by using integration by parts formula again and similar to the estimate of \eqref{err:EN:EF}, we have
	\begin{equation*} 
		\begin{aligned}
			\malA_{4} 
			& \leq C|e_R^{n+1}|^2+C\|\nabla e_\phi^{n+1}\|^2+ \frac{M}{8}\|\nabla e_\mu^{n+1}\|^2 +C\|e_\phi^n\|^2+C\|e_{\bm u}^n\|^2\\
			& \quad +C \tau \big(\|\phi_t\|_{L^2(J_n ; L^2(\Omega))}^2+\|\bm u_t\|_{L^2(J_n ; \mbfL^2(\Omega))}^2\big).
		\end{aligned}
	\end{equation*}
	
Now, inserting the estimates of $\malA_2-\malA_{7}$ into \eqref{err:CH:all:e1}, we obtain the desired result.
\end{proof}

\paragraph{\bf \indent Part II. Estimates for auxiliary variable}
Next, we continue with the estimates for the auxiliary variable $R^{n+1}$ defined by \eqref{scheme:1th:rhat}. The SAV-ZEC approach and the introduced $\gamma^n$ shall play a key role in the error analysis. 
\begin{lemma} \label{lem:err:r}
	Under the assumption \eqref{assump:e1}, we have
	\begin{equation*}  
		\begin{aligned}
			\p_\tau |e_R^{n+1}|^2 
			& \leq  \frac{e_R^{n+1}}{\sqrt{E_1(\phi^n)+\delta_0}} \big(\lambda(F^{\prime}(\phi^n), \p_\tau   e_\phi^{n+1})+(\widehat{e}_{\bm u}^{n+1}, \phi^n \nabla \mu^n) +  \bm b \big(\bm u^n, \bm u^n, \widehat{e}_{\bm u}^{n+1})\big) \\
			& \quad  + \frac{M}{4}\|\nabla e_{\mu}^{n+1}\|^2  
			+  \frac{M}{4}\|\nabla e_{\mu}^{n}\|^2
			+ C\|\widehat{e}_{\bm u}^{n+1}-e_{\bm u}^{n+1}\|^2
			+ \frac{\nu}{4}\|\nabla e_{\bm u}^n\|^2 + C|e_R^{n+1}|^2  +C\| e_{\phi}^n\|^2  \\
			& \quad+  C \tau \big(
			\|\phi_{t}\|_{L^2(J_n ; L^2(\Omega))}^2 +\|\phi_{tt}\|_{L^2(J_n ; L^2(\Omega))}^2 +\|\mu_t\|_{L^2(J_n ; H^1(\Omega))}^2 
			    +  |R_{tt}|_{L^2(J_n)}^2+\|\bm u_t\|_{L^2(J_n ; \mbfH^1(\Omega))}^2\big).
		\end{aligned}
	\end{equation*}
\end{lemma}
\begin{proof}
	Subtracting \eqref{scheme:1th:rhat} from \eqref{model:sav:r} at $t=t_{n+1}$ and multiplying the resulting equation by $2 \lambda  e_R^{n+1}$ leads to
	\begin{equation} \label{err:r:all}
		\begin{aligned}
			&	\p_\tau |e_R^{n+1}|^2 + \tau |\p_\tau   e_R^{n+1}|^2 \\
			& =   \lambda \frac{e_R^{n+1}}{\sqrt{E_1(\phi^n)+\delta_0}}\big( (F^{\prime}(\phi(t_{n+1})), \phi_t(t_{n+1})) - (F^{\prime}(\phi^n), \p_\tau   \phi^{n+1})\big) \\
			&\quad  +  \frac{e_R^{n+1}}{\sqrt{E_1(\phi^n)+\delta_0}}\big( (\mu(t_{n+1}), \nabla \cdot (\bm u(t_{n+1}) \phi(t_{n+1}))) - (\mu^{n+1},  \nabla \cdot( \bm u^n \phi^n))\big)\\
			&\quad  +  \frac{e_R^{n+1}}{\sqrt{E_1(\phi^n)+\delta_0}}\big((\bm u(t_{n+1}), \phi(t_{n+1}) \nabla \mu(t_{n+1}))-(\widehat{\bm u}^{n+1}, \phi^n \nabla \mu^n)\big) \\
			&\quad  +  \frac{e_R^{n+1}} {\sqrt{E_1(\phi^n)+\delta_0}}
			\big( \bm b \big(\bm u(t_{n+1}), \bm u(t_{n+1}), \bm u(t_{n+1}))-\bm b(\bm u^n, \bm u^n, \widehat{\bm u}^{n+1})\big) \\
			& \quad + \frac{e_R^{n+1}}
			{\sqrt{E_1(\phi^n)+\delta_0}}\gamma^n \big( (\nabla p(t_{n+1}), \bm u(t_{n+1}))- ( \nabla p^n, \widehat{\bm u}^{n+1})\big)
			+2 e_R^{n+1}  R_r^{n+1}, 
		\end{aligned}
	\end{equation}
	where 
	\begin{equation*} 
		R_r^{n+1}:=\p_\tau R(t_{n+1}) -R_t(t_{n+1})=\frac{1}{\tau} \int_{t_n}^{t_{n+1}}(t-t_n) R_{tt} d t. 
	\end{equation*}
	
	We denote the right-hand side terms of \eqref{err:r:all} by $\malB_i$ ($1\le i \le 6$). Then, by using the Cauchy-Schwarz inequality and the boundedness of $\phi(t_{n+1})$ and $\phi^n$, the first term can be recast into
	\begin{align*} 
		\begin{aligned}
			\malB_1 & = \lambda {\frac{e_R^{n+1}}{\sqrt{E_1(\phi^{n})+\delta_0}} \big(F^{\prime}(\phi^{n}),\p_\tau   e_\phi^{n+1}\big)} - \lambda \frac{e_R^{n+1}}{\sqrt{E_1(\phi^n)+\delta_0}}\big(F^{\prime}(\phi^n), R_\phi^{n+1}\big) \\
			& \qquad + \lambda \frac{e_R^{n+1}}{\sqrt{E_1(\phi^n)+\delta_0}}\big(F^{\prime}(\phi(t_{n+1}))-F^{\prime}(\phi^n), \phi_t(t_{n+1})\big) \\
			&  \leq  \lambda \frac{e_R^{n+1}}{\sqrt{E_1(\phi^{n})+\delta_0}} \big(F^{\prime}(\phi^{n}),\p_\tau   e_\phi^{n+1}\big)+ C|e_R^{n+1}|^2+C\|e_\phi^n\|^2
			+C\tau\|\phi_{tt}\|_{L^2(J_n ; L^2(\Omega))}^2. 
		\end{aligned}
	\end{align*}
	
	Note that by integration by parts formula, we have
	\begin{equation} \label{err:zec1}
		\begin{aligned}
			& (\mu(t_{n+1}), \nabla \cdot (\bm u(t_{n+1}) \phi(t_{n+1}))) - (\mu^{n+1},  \nabla \cdot( \bm u^n \phi^n)) \\
			&  = -\big( \nabla e_\mu^{n+1},  \bm u^{n}\phi^n\big) 
			-\big(\nabla{\mu(t_{n+1})}, (e^{n}_{\bm u} + \int_{t_n}^{t_{n+1}}{\bm u}_t dt) \phi^n \big) 
		  -  \big(\nabla {\mu(t_{n+1})},  \bm u(t_{n+1})(e^{n}_\phi +\int_{t_n}^{t_{n+1}}\phi_t dt ) \big).
		\end{aligned}
	\end{equation}
	Then, combining  \eqref{bound:solu} and  \eqref{bound_phi} with \eqref{err:zec1}, the second term on the right-hand side of  \eqref{err:r:all} can be bounded by
	\begin{equation*} 
		\begin{aligned}
			\malB_2       & \le 
			\frac{ M}{4} \|\nabla e_{\mu}^{n+1}\|^2 +  C|e_R^{n+1}|^2 + C \| e_{\phi}^n\|^2 + C\| e_{\bm u}^{n}\|^2 + C\tau\big( \|\bm u_t\|_{L^2(J_n ; \mbfL^2(\Omega))}^2  +  \|\phi_t\|_{L^2(J_n ; L^2(\Omega))}^2\big).
		\end{aligned}
	\end{equation*}
	
	Furthermore, we have
	\begin{equation} \label{err:zec2}
		\begin{aligned} 
			& (\bm u(t_{n+1}), \phi(t_{n+1}) \nabla \mu(t_{n+1}))-(\widehat{\bm u}^{n+1}, \phi^n \nabla \mu^n) \\
			&  =   (\widehat{e}_{\bm u}^{n+1}, \phi^n \nabla \mu^n) 
			+ \big(\bm u(t_{n+1}),\big(e_\phi^n+\int_{t_n}^{t_{n+1}} \phi_t d t\big)\nabla \mu(t_{n+1})  \big)  
			 +  \big(\bm u(t_{n+1}), \phi^n \nabla \big(e_\mu^n+\int_{t_n}^{t_{n+1}} \mu_t d t\big) \big).
		\end{aligned}
	\end{equation}
	Then, similar to the estimate of $ \malB_2 $, it follows from \eqref{err:zec2} that
	\begin{equation*}
			\malB_3       \le \frac{e_R^{n+1}}  {\sqrt{E_1(\phi^n)+\delta_0}} (\widehat{e}_{\bm u}^{n+1}, \phi^n \nabla \mu^n) + C|e_R^{n+1}|^2+  \frac{M}{4}\|\nabla e_{\mu}^{n}\|^2 + C\|e_{\phi}^n\|^2
			    +C\tau \big(\|\phi_t\|_{L^2(J_n ; L^2(\Omega))}^2 + \|\mu_t\|_{L^2(J_n ; H^1(\Omega))}^2\big).
	\end{equation*}
	
	Additionally, using (i) of Lemma \ref{lem:trilinear}, we obtain 
	$$\bm b( \bm u^n,  \bm u(t_{n+1}), \bm u(t_{n+1}) )= \bm b \big(\bm u(t_{n+1}), \bm u(t_{n+1}), \bm u(t_{n+1}))=0,$$
	which yields
	\begin{equation*}  
		\begin{aligned}
			 \bm b \big(\bm u(t_{n+1}), \bm u(t_{n+1}), \bm u(t_{n+1}))-\bm b(\bm u^n, \bm u^n, \widehat{\bm u}^{n+1}) 
			  =  \bm b \big(\bm u^n, \bm u^n, \widehat{e}_{\bm u}^{n+1})
			+ \bm b\big(\bm u^n, e_{\bm u}^{n}+\int_{t_n}^{t_{n+1}} \bm u_t d t,\bm u(t_{n+1}) \big),
		\end{aligned}
	\end{equation*}
	and thus, by (ii) of Lemma \ref{lem:trilinear} and the stability result \eqref{bound:solu}, the fourth term 
	can be estimated by
	\begin{equation*}
			\malB_4     \le \frac{e_R^{n+1}} {\sqrt{E_1(\phi^n)+\delta_0}} \bm b \big(\bm u^n, \bm u^n, \widehat{e}_{\bm u}^{n+1}) +\frac{\nu}{4}\|\nabla e_{\bm u}^{n}\|^2  + C|e_R^{n+1}|^2 +C\tau\|\bm u_t\|_{L^2(J_n ; \mbfH^1(\Omega))}^2.
	\end{equation*}
	
	Due to \eqref{model:sav:divu} and \eqref{scheme:1th:div}, we see
	\begin{equation}\label{zec:e2}
		(\nabla p(t_{n+1}), \bm u(t_{n+1}))		= \big(\nabla p^n, e_{\bm u}^{n+1}\big)=0.
	\end{equation}
	Thanks to \eqref{gamma:bound} and \eqref{zec:e2}, the fifth term on the right-hand side of \eqref{err:r:all} can be bounded by
	\begin{equation*} 
			\malB_5     = \frac{e_R^{n+1}}  {\sqrt{E_1(\phi^n)+\delta_0}} \gamma^n\big(\nabla p^n, \widehat{e}_{\bm u}^{n+1}-e_{\bm u}^{n+1}\big)
			\le C|e_R^{n+1}|^2 + C\|\widehat{e}_{\bm u}^{n+1}-e_{\bm u}^{n+1}\|^2.
	\end{equation*}
	
	Finally, for the last term on the right-hand side of \eqref{err:r:all}, we have
	\begin{equation*}
		\malB_6 \le C|e_R^{n+1}|^2+C \tau |R_{tt}|_{L^2(J_n)}^2.
	\end{equation*}
	
	In summary, substituting the estimates of $\malB_1-\malB_6$ into  \eqref{err:r:all}, we then complete the proof.
\end{proof}
\paragraph{\bf \indent Part III. Estimates for velocity variable}
In this part, we consider the error analysis of the velocity field defined by \eqref{scheme:1th:uhat} and \eqref{scheme:1th:uplus1}--\eqref{scheme:1th:div}. For clarity, we firstly present three error equations that will be utilized in the proof.

First, subtracting \eqref{scheme:1th:uhat} from \eqref{model:sav:u} at $t=t_{n+1}$ and recalling $\xi(t_{n+1})=\frac{R(t_{n+1})}{\sqrt{E_1(\phi(t_{n+1}))+\delta_0}}\equiv 1$, we obtain
\begin{equation*} 
	(E_1): \quad  \left\{
	\begin{aligned}
		& \frac{\widehat{e}_{\bm u}^{n+1}-e_{\bm u}^{n}}{\tau}-\nu \Delta \widehat{e}_{\bm u}^{n+1}  = \bm \chi_1 +\bm \chi_2 + \bm \chi_3 +R_{\bm u}^{n+1}, 
		\quad \text { in } \Omega, \\ 
		& \left. \widehat{e}_{\bm u}^{n+1}\right|_{\partial \Omega}=\mathbf{0},
	\end{aligned}\right. 
\end{equation*}
where  
\begin{equation*}
	\begin{aligned}
		\bm \chi_1 & := \xi^{n+1}\bm u^n \cdot \nabla \bm u^n - \bm u(t_{n + 1}) \cdot \nabla \bm u(t_{n+1}), \\ 
		& = -\frac{e_R^{n+1}}{\sqrt{E_1(\phi^n)+\delta_0}} \bm u^n \cdot \nabla \bm u^n- \frac{R(t_{n+1})}{\sqrt{E_1(\phi^n)+\delta_0}} \big(e_{\bm u}^n +\int_{t_n}^{t_{n+1}} \bm u_t d t\big) \cdot \nabla \bm u(t_{n + 1})  \\
		& \quad - \frac{R(t_{n+1})}{\sqrt{E_1(\phi^n)+\delta_0}} \bm u^n \cdot \nabla \big(e_{\bm u}^n +\int_{t_n}^{t_{n+1}} \bm u_t d t\big) \\
		& \quad + \big(\frac{R(t_{n+1})}{\sqrt{E_1(\phi^n)+\delta_0}} - \frac{R(t_{n+1})}{\sqrt{E_1(\phi(t_{n+1})))+\delta_0}} \big) \bm u(t_{n + 1}) \cdot \nabla \bm u(t_{n+1}), \\
		\bm \chi_2  &:= \xi^{n+1} \phi^n  \nabla  \mu^n - \phi(t_{n+1}) \nabla \mu(t_{n+1}), \\
		& = - \frac{e_R^{n+1}}{\sqrt{E_1(\phi^n)+\delta_0}} \phi^n \nabla \mu^n  -\frac{R(t_{n+1}) }{\sqrt{E_1(\phi^n)+\delta_0}}  \phi^n \big(\nabla  e_\mu^n+ \int_{t_n}^{t_{n+1}} \nabla \mu_t dt\big) \\
		& \quad -\frac{R(t_{n+1}) }{\sqrt{E_1(\phi^n)+\delta_0}} \big(e_\phi^n +\int_{t_n}^{t_{n+1}}\phi_t dt \big)\nabla \mu(t_{n+1})  \\
		& \quad  +  \Big( \frac{R(t_{n+1})}{\sqrt{E_1(\phi^n)+\delta_0}}-\frac{R(t_{n+1})}{\sqrt{E_1(\phi(t_{n+1}))+\delta_0}}\Big) \phi(t_{n+1}) \nabla \mu(t_{n+1}),\\
		\bm \chi_3 &:=  \nabla \big(\xi^{n+1} \gamma^n p^n-p(t_{n+1})\big),\
		R_{\bm u}^{n+1}:=\p_{\tau} \bm u(t_{n+1}) - \bm u_t(t_{n+1})=\frac{1}{\tau} \int_{t_n}^{t_{n+1}}\left(t-t_n\right) \bm u_{tt} d t.
	\end{aligned}
\end{equation*}

Next, we derive from \eqref{scheme:1th:uplus1}--\eqref{scheme:1th:div} that 
\begin{equation*} 
	(E_2): \quad   \left\{
	\begin{aligned}
		& \frac{e_{\bm u}^{n+1}-\widehat{e}_{\bm u}^{n+1}}{\tau} -\nu \Delta(e_{\bm u}^{n+1}-\widehat{e}_{\bm u}^{n+1})-\nabla p^{n+1}=-\gamma^n \nabla p^n, \quad \text { in } \Omega, \\ 
		& \nabla \cdot e_{\bm u}^{n+1}=0 \quad \text { in } \Omega,\quad e_{\bm u}^{n+1}|_{\partial \Omega}=\mathbf{0}. 
	\end{aligned}
	\right. 
\end{equation*}

Finally, adding ($E_1$) and ($E_2$) together, we arrive at the following equation
\begin{equation*} 
	(E_3): \quad   \left\{
	\begin{aligned}
		& \p_\tau e_{\bm u}^{n+1}-\nu\Delta e_{\bm u}^{n+1}+\nabla e_p^{n+1}=\bm \chi_1 +\bm \chi_2 + \bm \chi_4 +R_{\bm u}^{n+1},  \quad \text { in } \Omega, \\
		& \nabla \cdot e_{\bm u}^{n+1}=0 \quad \text { in } \Omega, \quad e_{\bm u}^{n+1}|_{\partial \Omega}=\mathbf{0}, 
	\end{aligned}
	\right. 
\end{equation*}
where 
$
\bm \chi_4 := \big(\xi^{n+1}-1)\gamma^n \nabla p^n= \malE_R^{n+1}\gamma^n \nabla p^n.
$

Next, we shall provide a rough estimate for $e_{\bm u}^{n+1}$ and $e_{\bm u}^{n+1}-\widehat{e}_{\bm u}^{n+1}$ using $(E_1)$ and $(E_2)$ under the assumptions stated in \eqref{hypothesis:n:1}.
\begin{lemma} \label{lem:NS:E1}
	Under the assumption \eqref{assump:e1}, for $\tau$ sufficiently small, we have
	\begin{equation*}  
		\begin{aligned}
			& \frac{1}{2}\p_\tau \|e_{\bm u}^{n+1}\|^2 
			+\frac{1}{2\tau}\|e_{\bm u}^{n+1}-\widehat{e}_{\bm u}^{n+1}\|^2 + \frac{\nu}{2} \|\nabla e_{\bm u}^{n+1}\|^2 \\
			&  \le  - \frac{e_R^{n+1}}{\sqrt{E_1(\phi^n)+\delta_0 } }\big( \phi^n \nabla \mu^n,\widehat e_{\bm u}^{n+1}) + \bm b(\bm u^n,  \bm u^n, \widehat{e}_{\bm u}^{n+1})  \big) + \frac{\nu}{4} \|\nabla e_{\bm u}^n\|^2 +\frac{M}{4}\|\nabla e_\mu^n\|^2 \\
			& \quad  +   C\|\widehat e_{\bm u}^{n+1}-e_{\bm u}^{n+1}\|^2 + C\|e_{\bm u}^{n+1}\|^2  + C|e_R^{n+1}|^2 +C\|e_\phi^n\|_1^2+ C \|e_{\bm u}^n\|^2  \\
			& \quad +  C\tau \big(1+ \|\phi_t\|_{L^2(J_n ; L^6(\Omega))}^2    +\|\mu_t\|_{L^2(J_n ; H^1(\Omega))}^2 
			+ \|\bm u_t\|_{L^2(J_n; \mbfL^2(\Omega))}^2 +\|\bm u_{tt}\|_{L^2(J_n; \mbfL^2(\Omega))}^2\big).
		\end{aligned}
	\end{equation*}
\end{lemma}
\begin{proof}
	Taking the inner products of ($E_1$) with $\widehat{e}_{\bm u}^{n+1}$, ($E_2$) with $e_{\bm u}^{n+1}$ and using \eqref{zec:e2}, and then adding the resulting equations together, we obtain
	\begin{equation} \label{err:euhat:1}
		\begin{aligned}
			& \frac{1}{2}\p_\tau \|e_{\bm u}^{n+1}\|^2 +\frac{1}{2\tau}(\|\widehat{e}_{\bm u}^{n+1}-e_{\bm u}^n\|^2+\|e_{\bm u}^{n+1}-\widehat{e}_{\bm u}^{n+1}\|^2) \\
			& \quad +\frac{\nu}{2}\big(\|\nabla e_{\bm u}^{n+1}\|^2+\|\nabla \widehat{e}_{\bm u}^{n+1}\|+ \|\nabla(e_{\bm u}^{n+1}-\widehat{e}_{\bm u}^{n+1})\|^2\big) \\
			&  = (\bm \chi_1, \widehat{e}_{\bm u}^{n+1}) +(\bm \chi_2, \widehat{e}_{\bm u}^{n+1}) + (\bm \chi_3, \widehat{e}_{\bm u}^{n+1}) +(R_{\bm u}^{n+1}, \widehat{e}_{\bm u}^{n+1})
			=:\sum\limits_{i=1}^4 {\malC}_i.
		\end{aligned}
	\end{equation}
	
	For the first term on the right-hand side of \eqref{err:euhat:1}, utilizing Lemmas \ref{lem:poncare}--\ref{lem:trilinear}, \eqref{bound:solu}  and \eqref{hypothesis:n:1}, we have
	\begin{equation*} 
		\begin{aligned}
			{\malC}_1 
			& = -\frac{e_R^{n+1}}{\sqrt{E_1(\phi^n)+\delta_0}} \bm b(\bm u^n,  \bm u^n, \widehat{e}_{\bm u}^{n+1} ) \\
			& \quad - \frac{R(t_{n+1})}{\sqrt{E_1(\phi^n)+\delta_0}} \Big[ \bm b(e_{\bm u}^n +\int_{t_n}^{t_{n+1}} \bm u_t d t, \bm u(t_{n + 1}), \widehat{e}_{\bm u}^{n+1}) +  \bm b (\bm u^n, e_{\bm u}^n +\int_{t_n}^{t_{n+1}} \bm u_t d t, \widehat{e}_{\bm u}^{n+1})  \Big]\\
			& \quad + \big(\frac{R(t_{n+1})}{\sqrt{E_1(\phi^n)+\delta_0}} - \frac{R(t_{n+1})}{\sqrt{E_1(\phi(t_{n+1})))+\delta_0}} \big) \bm b(\bm u(t_{n + 1}), \bm u(t_{n+1}),\widehat{e}_{\bm u}^{n+1})  \\ 
			& \leq -\frac{e_R^{n+1}}{\sqrt{E_1(\phi^n)+\delta_0}} \bm b(\bm u^n,  \bm u^n, \widehat{e}_{\bm u}^{n+1} ) + C  \|\bm u(t_{n+1})\|_2   \big(\|e_{\bm u}^n\|+\big\|\int_{t_n}^{t_{n+1}} \bm u_t d t\big\| \big) \|\nabla \widehat{e}_{\bm u}^{n+1}\|  \\
			& \quad + C\|{\bm u}^n\|_2\, \big\|\int_{t_n}^{t_{n+1}} \bm u_t d t\big\| \,  
			\|\nabla \widehat{e}_{\bm u}^{n+1}\|+   C\|{\bm u}^n\|_2 \|\nabla e_{\bm u}^n\| \|\widehat{e}_{\bm u}^{n+1}-e_{\bm u}^n\|   \\
			& \quad + \big|\frac{R(t_{n+1})}{\sqrt{E_1(\phi^n)+\delta_0}} - \frac{R(t_{n+1})}{\sqrt{E_1(\phi(t_{n+1})))+\delta_0}} \big|  \|\bm u(t_{n + 1})\|_1^2 \|\nabla \widehat{e}_{\bm u}^{n+1}\| \\
			& \leq -\frac{e_R^{n+1}}{\sqrt{E_1(\phi^n)+\delta_0}} \bm b(\bm u^n,  \bm u^n, \widehat{e}_{\bm u}^{n+1} ) +  \frac{\nu}{4}\|\nabla \widehat{e}_{\bm u}^{n+1}\|^2 + \frac{\nu}{4} \|\nabla e_{\bm u}^n\|^2 + C_1 \|\widehat{e}_{\bm u}^{n+1}-e_{\bm u}^n\|^2\\ 
			& \quad  +C\|e_{\phi}^n\|^2 + C\|e_{\bm u}^n\|^2
			+ C \tau \big( \|\phi_t\|_{L^2(J_n ; L^2(\Omega))}^2  + \|\bm u_t\|_{L^2(J_n;\mbfL^2(\Omega))}^2 \big),
		\end{aligned}
	\end{equation*}
	where the conclusion that
	\begin{equation*}
		\bm b (\bm u^n, e_{\bm u}^n, \widehat{e}_{\bm u}^{n+1}) = \bm b (\bm u^n, e_{\bm u}^n,\widehat{e}_{\bm u}^{n+1}-e_{\bm u}^n) 
	\end{equation*}
	due to $\bm b (\bm u^n, e_{\bm u}^n,e_{\bm u}^n)=0$, see (i) of Lemma  \ref{lem:trilinear}, has been applied. 
	
	By using \eqref{bound_phi},  the second term  on the right-hand side of \eqref{err:euhat:1} can be bounded by
	\begin{equation*}  \label{err:chi2}
		\begin{aligned}
			{\malC}_2  
			& \leq   -\frac{e_R^{n+1}}{\sqrt{E_1(\phi^n)+\delta_0 } } ( \phi^n \nabla \mu^n ,\widehat e_{\bm u}^{n+1})  + \frac{M}{4}\|\nabla e_\mu^n\|^2 + C\|\widehat{e}_{\bm u}^{n+1}-e_{\bm u}^{n+1}\|^2 +C\|e_{\bm u}^{n+1}\|^2\\
			&  \quad   +C\|e_\phi^n\|_1^2+C\tau \big( \|\phi_t\|_{L^2(J_n ; L^6(\Omega))}^2+\|\mu_t\|_{L^2(J_n ; H^1(\Omega))}^2\big).
		\end{aligned}
	\end{equation*}
	Here we have used the fact $\mu(t_{n+1}) \in  H^2(\Omega)$ and
	\begin{equation*}
		\begin{aligned}
			\big( \big(e_\phi^n +\int_{t_n}^{t_{n+1}}\phi_t dt \big)\nabla \mu(t_{n+1}), \widehat e_{\bm u}^{n+1} \big)  
			&\le  C \|e_\phi^n + \int_{t_n}^{t_{n+1}} \phi_t dt\|_{L^6} \|\nabla \mu(t_{n+1})\|_{L^3} \|\widehat e_{\bm u}^{n+1}\|\\
			& \le   C\|\widehat{e}_{\bm u}^{n+1}-e_{\bm u}^{n+1}\|^2 +C\|e_{\bm u}^{n+1}\|^2 
			+C\|e_\phi^n\|_1^2  +C\tau  \|\phi_t\|_{L^2(J_n ; L^6(\Omega))}^2,
		\end{aligned}
	\end{equation*}
	where, the Sobolev embedding $H^1(\Omega) \hookrightarrow L^6(\Omega)$ is utilized in the last step.
	
	Noticing \eqref{gamma:bound} and $\nabla \cdot e_{\bm u}^{n} =0$ with no-flow zero boundary condition, the third term  on the right-hand side of \eqref{err:euhat:1} can be estimated by
	\begin{equation*}  
		\begin{aligned}
			{\malC}_3
			& = (\nabla(\xi^{n+1} \gamma^n  p^n-p(t_{n+1})),  \widehat{e}_{\bm u}^{n+1}-e_{\bm u}^n) \\
			& \leq C\tau(\|p\|_{L^{\infty}(J;  H^1(\Omega))}^2+1) 
			+ \frac{1}{4\tau}\|\widehat{e}_{\bm u}^{n+1}-e_{\bm u}^n\|^2.
		\end{aligned}
	\end{equation*}
	
	For the last term on the right-hand side of \eqref{err:euhat:1}, it follows from Cauchy-Schwarz inequality and triangle inequality that
	\begin{equation*}  
		{\malC}_4 \le C\|\widehat{e}_{\bm u}^{n+1}-e_{\bm u}^{n+1}\|^2 +C\|e_{\bm u}^{n+1}\|^2  +C\tau\| \bm u_{tt}\|_{L^2(J_n; \mbfL^2(\Omega))}^2.
	\end{equation*}
	
	Therefore, inserting the estimates of ${\malC}_1-{\malC}_4$ into \eqref{err:euhat:1}, for $\tau$ sufficiently small such as
	$C_1\tau \le 1/4$, we arrive at the conclusion of the lemma.
\end{proof}
\paragraph{\bf \indent Part IV. Proof of Theorem \ref{err:sub_rate}}
Denote
\[
W^n:= \lambda \|\nabla e_\phi^n\|^2 + (1+\beta)\| e_\phi^n\|^2 + 2|e_R^n|^2+ \|e_{\bm u}^n\|^2 + M\tau\|e_\mu^n\|_1^2 + \nu \tau \|\nabla e_{\bm u}^n\|^2.
\]
It then follows from Lemmas \ref{lem:err:CH}--\ref{lem:NS:E1} that 
\begin{equation} \label{err:gronwall:1}
	\begin{aligned}
		& W^{n+1} - W^n +  \|e_{\bm u}^{n+1}-\widehat{e}_{\bm u}^{n+1}\|^2   \\
		& \le  C_2 \tau \big(\|\widehat{e}_{\bm u}^{n+1}-e_{\bm u}^{n+1}\|^2+  W^{n+1}+ W^n\big)   +C \tau^2   \\
		& \qquad  
		 +  C \tau^2 \big(\|\phi_t\|_{L^2(J_n ; H^1(\Omega))}^2  + \|\phi_{tt}\|_{L^2(J_n ; L^2(\Omega))}^2 + \|\mu_t\|_{L^2(J_n ; H^1(\Omega))}^2\big)\\
		&  \qquad  +  C \tau^2 \big(    \|\bm u_t\|_{L^2(J_n ; \mbfH^1(\Omega))}^2 + \|\bm u_{tt}\|_{L^2(J_n;  \mbfL^2(\Omega))}^2+|R_{tt}|_{L^2(J_n)}^2\big).
	\end{aligned}
\end{equation} 

Replacing  $n$ with  $m$ in \eqref{err:gronwall:1}, and summing it from 0 to $n$ ($n \ge 0$),    for $\tau$ sufficiently small such that  $\tau \le  \tau_0:= \min\{\frac{1}{4C_1}, \frac{1}{2C_2}\}$,  the application of  discrete Gronwall inequality leads to 
\begin{equation*}
	W^{n+1} + \frac{1}{2}\sum_{m=0}^{n} ( \|e_{\bm u}^{n+1}-\widehat{e}_{\bm u}^{n+1}\|^2 ) \leq \tau \exp(2C_2T),
\end{equation*}
which completes the proof. \qed

\begin{remark} \label{remk:sub_rate:1}
	It can be observed that the loss of convergence rate arises from the estimate of the pressure term in ${\malC}_3$. Besides, as seen, the parameter $\gamma$ is introduced to ensure that \eqref{gamma:bound} holds; without it, even a suboptimal estimate would not be attainable. In the next subsection, we will leverage the error equation ($E_3$) and the property $(\nabla p^n, \bm u^{n+1}) = 0$ to derive an optimal-order error estimate. This approach is based on the suboptimal results, which are crucial for providing uniform estimates in the proofs of Lemmas \ref{err:full:eu} and \ref{err:full:stokes:NS}.
	In fact, Theorem \ref{err:sub_rate} shows that $\|e_\mu^k\|_1 + \| e_{\bm u}^k \|_1 \le  C$ holds  uniformly for all $0\leq k\leq n+1$, 
	which further implies that there exists  positive constants $ C_{\mu}$ and $C_{\bm u}$ such that 
	\begin{equation}\label{bound:mu}
		\|\mu^k\|_1 \le  C_{\mu}, ~~ \|\bm u^k\|_1 \le  C_{\bm u}, \quad 0\leq k\leq n+1.
	\end{equation}
\end{remark}

\subsection{Improved error estimates } \label{subsec:full_rate}
Based on the estimates established in Theorem \ref{err:sub_rate} and the assumption \eqref{hypothesis:n:1}, we will give an improved error estimate in this subsection.

\begin{lemma}\label{err:full:eu}
	Under the assumption \eqref{assump:e1}, 	for $\tau$ sufficiently small, we have
	\begin{equation*}   
		\begin{aligned}
			& \frac{1}{2}\p_\tau \|e_{\bm u}^{n+1}\|^2  +\frac{\nu}{2}\|\nabla e_{\bm u}^{n+1}\|^2 \\
			&  \le  - \frac{e_R^{n+1}}{\sqrt{E_1(\phi^n)+\delta_0 } } \big( (\phi^n \nabla \mu^n,\widehat e_{\bm u}^{n+1}) + \bm b(\bm u^n,  \bm u^n, \widehat e_{\bm u}^{n+1})\big) + \frac{M}{8}\|\nabla e_\mu^n\|^2 \\
			& \quad + \frac{\nu}{4}\|\nabla e_{\bm u}^{n}\|^2 +  C|e_R^{n+1}|^2 + C\|e_{\bm u}^{n+1}\|^2  + C\|e_{\bm u}^n\|^2+C\|e_\phi^n\|_1^2 + C\|e_{\bm u}^{n+1}-\widehat{e}_{\bm u}^{n+1}\|^2 \\
			& \quad  +C\tau \big(\|\phi_t\|_{L^2(J_n ; L^6(\Omega))}^2 + \|\mu_t\|_{L^2(J_n ; H^1(\Omega))}^2 +\|\bm u_t\|_{L^2(J_n ; \mbfL^2(\Omega))}^2+\|\bm u_{tt}\|_{L^2(J_n; \mbfL^2(\Omega))}^2\big).
		\end{aligned}
	\end{equation*}
\end{lemma}
\begin{proof}
	Taking inner product of $({E_3})$ of $e_{\bm u}^{n+1}$, we have
	\begin{equation} \label{err:eu:eq}
			\frac{1}{2}\p_\tau \|e_{\bm u}^{n+1}\|^2 + \frac{1}{2\tau} \|\nabla_\tau e_{\bm u}^{n+1}\|^2 +\nu \|\nabla e_{\bm u}^{n+1}\|^2   =(\bm \chi_1,  e_{\bm u}^{n+1})+(\bm \chi_2, e_{\bm u}^{n+1})+(R_{\bm u}^{n+1}, e_{\bm u}^{n+1}). 
	\end{equation}
	We denote the right-hand side terms of \eqref{err:eu:eq} by $\malD_i$ ($1\le i \le 3$).
	Next, similar  to the estimates of ${\malC}_1-{\malC}_3$,	we estimate the right-hand side terms sequentially as follows:
	\begin{equation*} 
		\begin{aligned}
			\malD_1  
			& \leq -\frac{e_R^{n+1}}{\sqrt{E_1(\phi^n)+\delta_0}} \bm b(\bm u^n,  \bm u^n, e_{\bm u}^{n+1} ) + C  \|\bm u(t_{n+1})\|_2   \big(\|e_{\bm u}^n\|+\big\|\int_{t_n}^{t_{n+1}} \bm u_t d t\big\| \big) \|\nabla e_{\bm u}^{n+1}\|  \\
			& \quad + C\|{\bm u}^n\|_1 \big\|\int_{t_n}^{t_{n+1}} \bm u_t d t\big\|_1   \|\nabla e_{\bm u}^{n+1}\|
			         +   C\|{\bm u}^n\|_2 \|\nabla e_{\bm u}^n\|\|e_{\bm u}^{n+1}-e_{\bm u}^n\|   \\
			& \quad + \big|\frac{R(t_{n+1})}{\sqrt{E_1(\phi^n)+\delta_0}} - \frac{R(t_{n+1})}{\sqrt{E_1(\phi(t_{n+1})))+\delta_0}} \big|  \|\bm u(t_{n + 1})\|_1^2  
			         \|\nabla e_{\bm u}^{n+1}\| \\
			& \leq -\frac{e_R^{n+1}}{\sqrt{E_1(\phi^n)+\delta_0}} \bm b(\bm u^n,  \bm u^n, \widehat e_{\bm u}^{n+1} )  + \frac{\nu}{2} \|\nabla e_{\bm u}^{n+1}\|^2    
			        +C_3\|\nabla_\tau e_{\bm u}^{n+1}\|^2 +  \frac{\nu}{4}\|\nabla e_{\bm u}^{n}\|^2\\
			& \quad  + C|e_R^{n+1}|^2 + C\|\widehat{e}_{\bm u}^{n+1}-e_{\bm u}^{n+1}\| + C\|e_{\bm u}^{n}\|^2  
			  +C\|e_{\phi}^n\|^2  +C\tau \big(\|\bm u_t\|_{L^2(J_n;\mbfH^1(\Omega))}^2
		      	+\|\phi_t\|_{L^2(J_n ; L^2(\Omega))}^2\big),
		\end{aligned}
	\end{equation*}
	where the estimates that derived from (i)-(ii) of Lemma  \ref{lem:trilinear} and \eqref{bound:mu} that
	\begin{equation*}
		\begin{aligned}
			&  -\frac{e_R^{n+1}}{\sqrt{E_1(\phi^n)+\delta_0}} \bm b(\bm u^n,  \bm u^n, e_{\bm u}^{n+1} ) \\
			& = -\frac{e_R^{n+1}}{\sqrt{E_1(\phi^n)+\delta_0}} \bm b(\bm u^n,  \bm u^n, \widehat e_{\bm u}^{n+1}) 
			-\frac{e_R^{n+1}}{\sqrt{E_1(\phi^n)+\delta_0}} \bm b(\bm u^n,  \bm u^n, e_{\bm u}^{n+1}-\widehat e_{\bm u}^{n+1})\\
			& \leq -\frac{e_R^{n+1}}{\sqrt{E_1(\phi^n)+\delta_0}} \bm b(\bm u^n,  \bm u^n, \widehat e_{\bm u}^{n+1} ) + C|e_R^{n+1}|^2 
			     + C\|\widehat{e}_{\bm u}^{n+1}-e_{\bm u}^{n+1}\|, 
		\end{aligned}
	\end{equation*}
	and
	\begin{equation*}
			\bm b(\bm u^n, e_{\bm u}^n,  e_{\bm u}^{n+1} )  
			=   \bm b(\bm u^n,  e_{\bm u}^n,   e_{\bm u}^{n+1}-e_{\bm u}^n)  
			\leq   C_3\|\nabla_\tau e_{\bm u}^{n+1}\|^2 + \frac{\nu}{4} \|\nabla e_{\bm u}^n\|^2.
	\end{equation*}
	have been applied. 
	
	Furthermore, due to \eqref{bound_phi} and \eqref{bound:mu}, we have
	\begin{equation*}
			 -\frac{e_R^{n+1}}{\sqrt{E_1(\phi^n)+\delta_0 } }(\phi^n \nabla \mu^n,  e_{\bm u}^{n+1}) 
			 		 \le -\frac{e_R^{n+1}}{\sqrt{E_1(\phi^n)+\delta_0 } }(\phi^n \nabla \mu^n, \widehat e_{\bm u}^{n+1}) + C\|e_{\bm u}^{n+1}-\widehat{e}_{\bm u}^{n+1}\|^2+C|e_R^{n+1}|^2,
	\end{equation*}
	and using the Sobolev embedding $H^1(\Omega) \hookrightarrow L^6(\Omega)$, we further obtain
	\begin{equation*}   
		\begin{aligned}
			\malD_2 
			& \leq  -\frac{e_R^{n+1}}{\sqrt{E_1(\phi^n)+\delta_0 } } ( \phi^n \nabla \mu^n ,\widehat e_{\bm u}^{n+1}) 
			+ \frac{M}{8}\|\nabla e_\mu^n\|^2
			+   C\|e_{\bm u}^{n+1}-\widehat{e}_{\bm u}^{n+1}\|^2 \\
			&  \quad  +C\|e_\phi^n\|_1^2 +  C|e_R^{n+1}|^2 + C\|e_{\bm u}^{n+1}\|^2 +C\tau(\|\mu_t\|_{L^2(J_n ; H^1(\Omega))}^2 +\|\phi_t\|_{L^2(J_n ; L^6(\Omega))}^2).
		\end{aligned}
	\end{equation*}
	
	Finally, for the last term of the right-hand side of \eqref{err:eu:eq}, we have  
	\[
	\malD_3  \le  C\|e_{\bm u}^{n+1}\|^2    +C\tau\| \bm u_{tt}\|_{L^2(J_n; \mbfL^2(\Omega))}^2.
	\]
	
Therefore, inserting the estimates of $\malD_1-\malD_3$ into \eqref{err:eu:eq}, for $\tau$ sufficiently small such that  $\tau \leq \min\{\tau_0, \frac{1}{2C_3}\}$, we arrive at the conclusion.
\end{proof}  

Now, based on Lemmas \ref{lem:err:CH}--\ref{lem:err:r} and \ref{err:full:eu}, we refine the error estimates established in Theorem \ref{err:sub_rate} as below.
\begin{theorem} \label{thm:error}
	Under the assumption \eqref{assump:e1},  there exists a sufficiently small constant $\tau_1$, such that for $\tau \le \tau_1$ we have
	\begin{equation*}   
		\|e_\phi^{n+1}\|_1^2+|e_R^{n+1}|^2 +\|e_{\bm u}^{n+1}\|^2 +\nu \tau \|\nabla e_{\bm u}^{n+1}\|^2+  \tau \sum_{m=0}^{n} \|e_\mu^{m+1}\|_1^2  
		\le C\tau^2,
	\end{equation*}
	where the constant $C$ is independent of $\tau$ and $n$, but depends on $T$.
\end{theorem}
\begin{proof}
	Denote 
	\[
	Z^n:= \lambda \|\nabla e_\phi^n\|^2 + (1+\beta)\| e_\phi^n\|^2 + 2|e_R^n|^2+  \|e_{\bm u}^n\|^2 +\nu \tau \|\nabla e_{\bm u}^{n}\|^2 
	+ \frac{3M\tau}{4}\| e_\mu^n\|_1^2.
	\]
	Then, combing  Lemmas \ref {lem:err:CH}--\ref{lem:err:r} and Lemma \ref{err:full:eu}, it yields  
	\begin{equation}  \label{err:full:all}
		\begin{aligned}
			 Z^{n+1} - Z^n + \frac{M\tau}{4} \| e_\mu^{n+1}\|_1^2 
			& \leq C_4\tau \big(Z^{n+1}+   Z^{n}\big) + C\tau \|e_{\bm u}^{n+1}-\widehat{e}_{\bm u}^{n+1}\|^2    \\
			& \quad + C\tau^2\big(\|\phi_{tt}\|_{L^2(J_n ; L^2(\Omega))}^2+\|\phi_t\|_{L^2(J_n ; H^1(\Omega))}^2  +  \|\mu_t\|_{L^2(J_n ; H^1(\Omega))}^2\big)\\
			&  \quad + C\tau^2 \big(
			\|\bm u_t\|_{L^2(J_n ; \mbfH^1(\Omega))}^2+\|\bm u_{tt}\|_{L^2(J_n; \mbfL^2(\Omega))}^2+ |R_{tt}|_{L^2(J_n)}^2 \big).
		\end{aligned}
	\end{equation} 
	
	Replacing  $n$ with  $m$ in \eqref{err:full:all}, summing it from 0 to $n$ ($n \ge 0$),  for $\tau$ sufficiently small such that  $\tau \le  \tau_1:=\min\{\tau_0,\frac{1}{2C_3},\frac{1}{2C_4}\}$,  by using the discrete Gronwall inequality, we obtain 
	\begin{equation*}
			Z^{n+1} + \frac{M\tau}{4}\sum_{m=0}^{n}  \| e_\mu^{n+1}\|_1^2 
			\le C\exp(2C_4T) \big(\tau \sum_{m=0}^{n}  \|e_{\bm u}^{m+1}-\widehat{e}_{\bm u}^{m+1}\|^2 + \tau^2 \big) \le C \tau^2,
	\end{equation*}
	where the result established in Theorem \ref{err:sub_rate} that
	\begin{equation} \label{err:ehat_e}
		\tau\sum_{m=0}^{n} \|\widehat{e}_{\bm u}^{m+1}-{e}_{\bm u}^{m+1}\|^2  \leq C  \tau^2,
	\end{equation}
	has been used. Thus, we complete the proof of Theorem \ref{thm:error}.
\end{proof}
\begin{remark} It is clear that \eqref{err:ehat_e} 	plays a crucial role in obtaining the improved convergence results. Without this estimate, we may only be able to prove the boundedness of the errors. On the other hand, the proven result, $\tau \sum_{m=0}^{n} \|e_\mu^{m+1}\|_1^2 \le C \tau^2$, is also essential to resolve our induction hypothesis \eqref{hypothesis:n:1}.
\end{remark}
\subsection{Completion of the mathematical induction} \label{subec:induct}
In this subsection, we shall prove that $\|e_{\bm u}^{n+1}\|_2 \le 1$ to complete the proof of the entire error analysis.
We first revisit the regularity of the Stokes equation.
\begin{lemma} [Regularity of $(\bm w, r)$  \cite{temam2001navier}] \label{lem:stokes}
	There exist unique solution pairs $(\bm w, r) \in \mbfH_0^1(\Omega) \times$ $L_0^2(\Omega)$ for the stationary Stokes equation 
	\begin{equation*} 
		\left\{
		\begin{aligned}
			& -\nu \Delta \bm w+\nabla r=\mathbf{g}, \quad\quad {\rm in }~ \Omega, \\
			& \nabla \cdot \bm w=0, \qquad \qquad \quad ~~ {\rm in }~ \Omega, \\
			& \bm w=\mathbf{0},  \qquad \qquad \qquad \quad  {\rm on }~ \partial \Omega,
		\end{aligned}
		\right.
	\end{equation*}
	such that  
	\begin{equation*}
		\|\bm w\|_2+\|r\|_1 \leq C_r\|\mathbf{g}\|,
	\end{equation*}
	where  $C_r$ is a positive constant depending on $\Omega$ and $\nu$. 
\end{lemma}

Next, using the assumption \eqref{hypothesis:n:1}, Lemma \ref{lem:stokes} and the estimates of $\|\nabla e_\mu^{n+1}\|$ in Theorem \ref{thm:error}, we shall give an estimate for $\|e_{\bm u}^{n+1}\|_2$ and subsequently complete the induction process \eqref{hypothesis:n:1}.
\begin{lemma} \label{err:full:stokes:NS}
	Under the assumption \eqref{assump:e1},  for $\tau \le \tau_0$  we have
	\begin{equation*}  
		\begin{aligned}
			&   C_r^2\nu\,  \p_\tau \|\nabla e_{\bm u}^{n+1}\|^2   +  \|e_{\bm u}^{n+1}\|_2^2+\|e_p^{n+1}\|_1^2\\
			&\le  C \|\nabla e_{\bm u}^{n+1}\|^2+  C |e_R^{n+1}|^2+ C\|e_{\phi}^n\|_1^2   + C\|\nabla e_\mu^n\|^2 + C\|\nabla e_{\bm u}^{n}\|^2\\
			& \quad +   C\tau \big(\|\phi_t\|_{L^2(J_n ; L^6(\Omega))}^2 + \|\mu_t\|_{L^2(J_n ; H^2(\Omega))}^2+\|\bm u_t\|_{L^2(J_n ; \mbfH^2(\Omega))}^2
			                  +\|\bm u_{tt}\|_{L^2(J_n; \mbfL^2(\Omega))}^2\big).
		\end{aligned}
	\end{equation*}
\end{lemma}
\begin{proof}
	The key idea for the proof is the $ \mbfH^2\times H^1$-regularity of ($E_3$) and the estimates of $(E_3, \p_\tau e_{\bm u}^{n+1})$. 
	Indeed, due to the  regularity theory stated in Lemma \ref{lem:stokes}, we have
	\begin{equation} \label{err:stokes}
		\|e_{\bm u}^{n+1}\|_2^2+\|e_p^{n+1}\|_1^2\leq  C_r^2 \big(\|\p_\tau e_{\bm u}^{n+1}\|^2+\|\bm \chi_1\|^2 +\|\bm \chi_2\|^2 +\|R_{\bm u}^{n+1}\|^2+ \|\bm \chi_4\|^2 \big),
	\end{equation}
	where, it is easy to see $\|R_{\bm u}^{n+1}\|^2 \le \tau\| \bm u_{tt}\|_{L^2(J_n; \mbfL^2(\Omega))}^2$, and moreover, by \eqref{gamma:bound} we have
	\begin{equation*} 
			\|\bm \chi_4\|^2  \le C|e_R^{n+1}|^2+C\|e_\phi^{n}\|^2+ C\tau\|\phi_t\|_{L^2(J_n; L^2(\Omega))}^2.
	\end{equation*}
	Furthermore, using Lemmas \ref{lem:poncare}--\ref{lem:trilinear}, \eqref{assump:e1}--\eqref{bound_phi}, \eqref{bound:mu} and Sobolev embedding $H^1(\Omega) \hookrightarrow L^6(\Omega)$, we obtain
	\begin{equation*}  
		\begin{aligned}
			\|\bm \chi_1\|^2 
			& \le  C|e_R^{n+1}|^2 + C\|e_{\phi}^n\|^2 +  C \|\nabla e_{\bm u}^{n}\|^2 + C\tau \big(\|\phi_t\|_{L^2(J_n ; L^2(\Omega))}^2+\|\bm u_t\|_{L^2(J_n ; \mbfH^1(\Omega))}^2\big),\\
			\|\bm \chi_2\|^2
			& \le   C|e_R^{n+1}|^2 + C\|e_{\phi}^n\|^2  + C \|\nabla e_\mu^n\|^2  +  C\tau \big( \|\phi_t\|_{L^2(J_n ; L^6(\Omega))}^2+\|\mu_t\|_{L^2(J_n ; H^1(\Omega))}^2\big),
		\end{aligned}
	\end{equation*}
	Thus, inserting the above estimates into \eqref{err:stokes}, we arrive at 
	\begin{equation} \label{err:stokes:2}
		\begin{aligned}
			&  \|e_{\bm u}^{n+1}\|_2^2+\|e_p^{n+1}\|_1^2  \\
			& \leq  C_r^2\|\p_\tau e_{\bm u}^{n+1}\|^2 +  C|e_R^{n+1}|^2+ C\|e_{\phi}^n\|_1^2 +   C\|\nabla e_\mu^n\|^2  + C\|\nabla e_{\bm u}^{n}\|^2 \\
			& \quad    +  C\tau \big(\|\phi_t\|_{L^2(J_n ; L^6(\Omega))}^2+\|\mu_t\|_{L^2(J_n ; H^1(\Omega))}^2+\|\bm u_t\|_{L^2(J_n ; \mbfH^1(\Omega))}^2 +\| \bm u_{tt}\|_{L^2(J_n; \mbfL^2(\Omega))}^2\big).
		\end{aligned}
	\end{equation}
	
	On the other hand, taking the inner product of $(E_3)$  with $\p_\tau e_{\bm u}^{n+1}$, and using the above estimates for $\bm \chi_1$, $\bm \chi_2$ and $R_{\bm u}^{n+1}$, we obtain
	\begin{equation} \label{err:eq:dte}
		\begin{aligned}
			& \|\p_\tau e_{\bm u}^{n+1}\|^2 +\frac{\nu}{2}  \p_\tau\|\nabla e_{\bm u}^{n+1}\|^2
			\le (\bm \chi_1 + \bm \chi_2 + R_{\bm u}^{n+1},  \p_\tau e_{\bm u}^{n+1})\\
			& \le \frac{1}{2}\|\p_\tau e_{\bm u}^{n+1}\|^2 +   C |e_R^{n+1}|^2 + C \|\nabla e_{\bm u}^{n+1}\|^2 + C\|e_{\phi}^n\|_1^2+ C \|\nabla e_\mu^n\|^2+  C\|\nabla e_{\bm u}^{n}\|^2  \\ %
			& \quad + C\tau(\|\phi_t\|_{L^2(J_n ; L^6(\Omega))}^2 +\|\mu_t\|_{L^2(J_n ; H^1(\Omega))}^2+\|\bm u_t\|_{L^2(J_n ; \mbfH^2(\Omega))}^2+\|\bm u_{tt}\|_{L^2(J_n; \mbfL^2(\Omega))}^2).
		\end{aligned}
	\end{equation}
	
	Finally, we multiply \eqref{err:eq:dte}  by $2C_r^2$ and add the result to \eqref{err:stokes:2}. Then, the lemma is proved.
\end{proof}

Now, based on Lemmas \ref{lem:err:CH}--\ref{lem:err:r} and Lemma \ref{err:full:stokes:NS}, our aim is to complete the induction process by establishing an error bound that $\|e_{\bm u}^{n+1}\|_2 \le 1$, and meanwhile, the first-order error estimates for $\|e_{\bm u}^{n+1}\|_1$  and $\tau \sum_{m=0}^n  \|e_p^{m+1}\|_1^2$ are also proved.
\begin{theorem} \label{err:full:stokes:all}
	Under the assumption \eqref{assump:e1},   there exists a sufficiently small constant $\tau_2$, such that for $\tau \le \tau_2$ we have
	\begin{equation*}
		\begin{aligned}
			\|e_\phi^{n+1}\|_1^2 + \|e_{\bm u}^{n+1}\|_1^2 +|e_R^{n+1}|^2  +\tau \sum_{m=0}^n (\|e_{\bm u}^{m+1}\|_2^2+\|e_p^{m+1}\|_1^2+\|e_\mu^{m+1}\|_1^2)   \le C \tau^2,
		\end{aligned}
	\end{equation*}
where the constant $C$ is independent of $\tau$ and $n$, but depends on $T$.
\end{theorem}
\begin{proof}
	To begin with, we estimate the SAV terms in Lemma \ref{lem:err:r}. Recalling the boundedness results of \eqref{bound_phi}--\eqref{hypothesis:n:1}, \eqref{bound:mu}  and using (ii) of Lemma \ref{lem:trilinear} and triangle inequality, we obtain
	\begin{equation}  \label{err:sav:1} 
	  \frac{e_R^{n+1}}{\sqrt{E_1(\phi^n)+\delta_0}} \big( (\widehat{e}_{\bm u}^{n+1}, \phi^n \nabla \mu^n)+ \bm b(\bm u^n,  \bm u^n, \widehat e_{\bm u}^{n+1} ) \big)
			 \le C |e_R^{n+1}|^2 + C(\|\widehat{e}_{\bm u}^{n+1}-e_{\bm u}^{n+1}\|^2 + \|e_{\bm u}^{n+1}\|^2).
	\end{equation}
	
	Denote 
	\[
	V^n:= \lambda \|\nabla e_\phi^n\|^2 + (1+\beta)\| e_\phi^n\|^2 + 2|e_R^n|^2+ 2C_r^2\nu \|\nabla  
	e_{\bm u}^n\|^2 +  M\tau \|e_\mu^n\|_1^2.
	\]
	Then, combining \eqref{err:sav:1} with Lemmas \ref{lem:err:CH}--\ref{lem:err:r} and Lemma \ref{err:full:stokes:NS},  we have
	\begin{equation*} \label{err:full_rate:1}
		\begin{aligned}
			& V^{n+1} -V^n + 2\tau(\|e_{\bm u}^{n+1}\|_2^2+ \|e_p^{n+1}\|_1^2)\\
			& \leq C_5\tau (V^{n+1} + V^n) +C\tau( \|\widehat{e}_{\bm u}^{n+1}-e_{\bm u}^{n+1}\|^2 + \|\nabla e_\mu^{n}\|^2)\\
			& \quad + C\tau^2 \big(\|\phi_t\|_{L^2(J_n ; H^1(\Omega))}^2 +\|\phi_{tt}\|_{L^2(J_n ; L^2(\Omega))}^2+  \|\mu_t\|_{L^2(J_n ; H^1(\Omega))}^2\big)\\
			& \quad +C\tau^2 \big( \|\bm u_t\|_{L^2(J_n ; \mbfH^2(\Omega))}^2+\|\bm u_{tt}\|_{L^2(J_n; \mbfL^2(\Omega))}^2+ |R_t|_{L^2(J_n)}^2\big).
		\end{aligned}
	\end{equation*}
	Therefore, similar to the proof of Theorem \ref{thm:error}, for $\tau$ sufficiently small such that  $\tau \le  \tau_2:=\min \{\tau_1, \frac{1}{2C_{5}}\}$, an application of the discrete Gronwall inequality leads to 
	\begin{equation} \label{err:reasult:1}
		\begin{aligned}
			& V^{n+1} +2\tau \sum_{m=0}^n (\|e_{\bm u}^{m+1}\|_2^2+\|e_p^{m+1}\|_1^2) 
			  \le C \tau \sum_{m=0}^n\big(\|\widehat{e}_{\bm u}^{m+1}-{e}_{\bm u}^{m+1}\|^2  + \|e_\mu^{m+1}\|_1^2 \big) + C \tau^2  \le  C\tau^2,
		\end{aligned}
	\end{equation}
	where   the results in Theorems \ref{err:sub_rate} and \ref{thm:error} that
	\begin{equation*} 
		\tau \sum_{m=0}^n\big(\|\widehat{e}_{\bm u}^{m+1}-{e}_{\bm u}^{m+1}\|^2  + \|e_\mu^{m+1}\|_1^2 \big) 	\leq  C \tau^2,
	\end{equation*}
	have been employed. Now, combing \eqref{err:reasult:1} with Theorem  \ref{thm:error}, we complete the proof.
\end{proof} 

\begin{remark} \label{remk:full_rate:1}
	Theorem \ref{err:full:stokes:all} implies that there exists a positive constant $C_{6}$ independent of $\tau$ and $n$ such that  
	\begin{equation*}
		\|e_{\bm u}^{n+1} \|_{2} \le C_{6} \tau^{\frac{1}{2}} \le 1,
	\end{equation*}
	if $\tau \leq \min\{\tau_2, 1/C_{6}^2\}$.
	Then, we complete the induction process \eqref{hypothesis:n:1} and thus Theorem \ref{err:full:stokes:all} holds for all $n=0,1,\ldots, N_t-1$.
\end{remark}

\section{Numerical experiments}\label{sec:test}
In this section, we present several numerical examples to verify both the convergence and reliability of the structure-preserving IVS scheme. 
Note that although the equality $R(t)=\sqrt{E_1(\phi)+\delta_0}$ holds at the continuous level,  $R^{n+1}$ is usually not equal to the discrete value $\sqrt{E_1(\phi^{n+1})+\delta_0}$, due to the possible truncation errors introduced in the scheme. This discrepancy may cause the modified energy dissipation law to lose consistency, particularly for long-term simulations.
To address this issue, we adopt the relaxation approach \cite{jiang2022improving} that strengthens the consistency between the modified and the original energy. Specifically, we rename the auxiliary variable computed in \eqref{scheme:1th:rhat} as $\widehat{R}^{n+1}$, and introduce another correction step to update $\widehat{R}^{n+1}$ in subsection \ref{subsec:scheme} below Step 3 as follows:
\paragraph{\bf Step 4}
Update the scalar auxiliary variable $\widehat R^{n+1}$  via a relaxation step as
\begin{equation*}
	R^{n+1}=\kappa_0 \widehat{R}^{n+1}+\left(1-\kappa_0\right) \sqrt{E_1(\phi^{n+1})+\delta_0},
\end{equation*}
where $\kappa_0$ is determined by the following optimization problem
\begin{equation*} 
	\kappa_0=\min _{\kappa \in[0,1]} \kappa, \quad \text { s.t. } |R^{n+1}|^2  - |\widehat R^{n+1}|^2 \leq \tau \eta M\|\nabla \mu^{n+1}\|^2,
\end{equation*}
with $\eta \in (0,1]$ an artificial relaxation parameter that can be manually assigned.  In fact,
the solution to the quadratic optimization problem is represented as 
\begin{equation*} 
	\kappa_0=\max \big\{0, (-b-\sqrt{b^2-4 a c}) / (2 a)\big\},
\end{equation*}
where the coefficients are
\begin{equation*} 
	\left\{
	\begin{aligned}
		& a= \big(\widehat{R}^{n+1}-\sqrt{E_1(\phi^{n+1})+\delta_0}\big)^2, \\
		& b=  2\sqrt{E_1(\phi^{n+1})+\delta_0}\big(\widehat{R}^{n+1}-\sqrt{E_1(\phi^{n+1})+\delta_0}\big), \\
		& c=  E_1(\phi^{n+1})+\delta_0  - |\widehat{R}^{n+1} |^2  -\tau \eta M \|\nabla \mu^{n+1} \|^2.
	\end{aligned}
	\right.
\end{equation*}
For the second-order scheme  in subsection \ref{sec:scheme:2th} below Step 3,  the corresponding relaxation steps can be formulated in a similar manner following Step 3. For more details on the relaxation method, we refer the readers to \cite{jiang2022improving}, see also  \cite{liu2024novel}  for the energy-optimized SAV method and \cite{zhang2024novel} for the generalized energy-optimized exponential SAV method.   

Our time discretization schemes are designed to be compatible with a wide range of spatial discretization methods, enabling the development of practical fully discrete numerical schemes. For simplicity, we employ the well-established Marker-And-Cell (MAC) spatial discretization approach \cite{li2022fully,harlow1965numerical} on staggered spatial grids. In this configuration, scalar variables - including the phase-field variable, chemical potential, and pressure - are discretized at cell centers, while velocity components are approximated at cell faces, ensuring a robust and stable numerical framework. In the simulation, we always set $N_x=N_y\equiv N$ (i.e., $h_x=h_y\equiv h$) for uniform spatial partitions, and the errors are measured in the following norms
\begin{equation*}
		\|e_{\bm g}\|_{\ell^2(X)}:=\big(\sum_{n=1}^{N_t} \tau \|\bm g^{n}-\bm g(t_{n}) \|_{X}^2\big)^{1 / 2}, \quad 
		\|e_{\bm g}\|_{\ell^\infty(X)}:=\max _{1 \leq n \leq {N_t}} \|\bm g^{n}-\bm g(t_{n})\|_{X}, \quad 
		|e_R|_{\ell^\infty}:=\max _{1 \leq n \leq N_t} |R^{n}-R(t_{n})|.
\end{equation*}

\subsection{Convergence tests}\label{subsec:e1}
In this example, we consider the accuracy  of the structure-preserving IVS schemes \eqref{scheme:1th:phi}--\eqref{scheme:1th:div} and \eqref{scheme:phi:2th}--\eqref{scheme:div:2th}. By adding two artificial source terms  to the right-hand side of the Navier--Stokes equation  \eqref{model:u} and  the Cahn--Hilliard equation \eqref{model:phi}, we construct the following manufactured exact solutions for the CHNS system as
\begin{equation*}
	\begin{aligned}
		& \phi(x, y, t)=\sin (t) \cos (\pi x) \cos (\pi y), \\
		& \bm u(x, y, t)=\cos (t)(\sin ^2(\pi x) \sin (2\pi y),-\sin (2\pi x) \sin ^2(\pi y)), \\
		& p(x, y, t)=\sin (t) \cos (\pi x) \sin (\pi y),
	\end{aligned}
\end{equation*}
where the parameters are chosen as
$M=  10^{-3}$, $\lambda = 1$, $\nu=10^{-2}$, $\epsilon= 9 \times 10^{-2}$,  and $\delta_0 = \beta =0$.
The computational domain $\Omega=[0,1]^2$ is discretized using $h=\sqrt{\tau/2}$ for the first-order scheme and $h=\tau/2$ for the second-order scheme with different time steps, ensuring that the spatial discretization error is comparable to the temporal discretization error.
\begin{table}[!htbp] 
	\caption{\small Errors and convergence rates of the first-order scheme.}   \label{tab:rate:1th:1th}
	{\footnotesize\begin{tabular*}{\columnwidth}{@{\extracolsep\fill}ccccc@{\extracolsep\fill}}
		\toprule 
		$\tau $   & $\frac{1}{32}$  &  $\frac{1}{64}$   & $\frac{1}{128}$  &  $\frac{1}{256}$\\  
		\midrule 
		 {$\|e_{\phi}\|_{\ell^\infty(L^2)}$} 		&  1.5506 E-3   & 7.7828 E-4  &  3.8988 E-4   &  1.9513 E-4 \\
		Rate & 		---         & 0.99       & 1.00        & 1.00 \\
		{$\| e_{\phi}\|_{\ell^\infty(H^1)}$} 		&  8.0993 E-3  & 4.0721 E-3   &  2.0419 E-3   & 1.0220 E-3 \\
		Rate 		& ---        & 0.99        & 1.00        & 1.00 \\ 
		{$\| e_{\bm u}\|_{\ell^\infty(H^{1})}$} 		&   1.5453 E-2  &  7.7321 E-3   & 3.8165 E-3   &   1.8749 E-3 \\
		Rate 		& ---        & 1.00        & 1.02        & 1.03 \\
		{ $\|e_{\bm u}\|_{\ell^{2}(H^{2})} $ } 		& 8.1322 E-2  & 4.4169 E-2   & 2.3055 E-2   & 1.1759 E-2 \\
		Rate		& ---      & 0.88        & 0.94        & 0.97 \\ 
		{ $\|e_{\bm u}\|_{\ell^{\infty}(L^{\infty})} $ } 		& 1.4096 E-3  & 7.2485 E-4   & 3.6672 E-4   & 1.8432 E-4 \\
		Rate		& ---      & 0.96        & 0.98        & 0.99 \\ 
		$\|e_{p}\|_{\ell^{\infty}(L^{2})}$		& 3.3944 E-2  &  1.7080 E-2  & 8.5676 E-3  &   4.2888 E-3 \\
		Rate		& ---      & 0.99        & 1.00        & 1.00 \\ 
		$\|e_{p}\|_{\ell^2(H^1)}$		& 6.2117 E-2  & 3.1233 E-2  & 1.5662 E-2  &  7.8391 E-3 \\
		Rate		& ---      & 0.99        & 1.00        & 1.00 \\ 
		{$|e_R|_{\ell^{\infty}}$} 		& 2.1119 E-4  &  1.0708 E-4   & 5.3905 E-5   & 2.7043 E-5 \\
		Rate 		& ---       & 0.98        & 0.99       & 1.00 \\  
		\bottomrule
	\end{tabular*}}
\end{table}

In Tables \ref{tab:rate:1th:1th}--\ref{tab:rate:1th:2th}, we list the errors and convergence rates of the two schemes at $T=0.2$. The results confirm that each scheme achieves the anticipated temporal accuracy, in full agreement with the theoretical estimates established in Theorem \ref{err:full:stokes:all}. In particular, Tables \ref{tab:rate:1th:1th}--\ref{tab:rate:1th:2th} demonstrate that the velocity $\bm u$ achieves optimal accuracy in the $\ell^\infty(0, T;\bm L^{\infty}) \cap \ell^2(0, T;\bm H^{2})$-norm, and the pressure $p$ converges optimally in the $\ell^\infty(0, T;L^{2}) \cap \ell^2(0, T;H^{1})$-norm. Notably, the classical pressure-correction methods \cite{guermond2006overview,li2022fully} often suffer from order reduction, particularly for the velocity in the $\ell^{2}(0, T; \bm H^1)$-norm and for the pressure in the $\ell^{2}(0, T; L^2)$-norm. Moreover, these methods rarely yield optimal convergence rates in stronger norms, such as the $\ell^2(0, T;\bm H^{2})$-norm for velocity or the $\ell^{2}(0, T; H^1)$-norm for pressure. In contrast, the proposed schemes circumvent the requirement for artificial boundary conditions and fully preserve the original velocity boundary conditions, thereby attaining optimal first- or second-order temporal accuracy.
\begin{table} [!htbp] 	 
	\caption{\small Errors and convergence rates of the second-order scheme.}   \label{tab:rate:1th:2th}
	{\footnotesize\begin{tabular*}{\columnwidth}{@{\extracolsep\fill}ccccc@{\extracolsep\fill}}
		\toprule 
		$\tau $   & $\frac{1}{20}$  &  $\frac{1}{40}$   & $\frac{1}{80}$  &  $\frac{1}{160}$ \\
		\midrule 
		 {$\|e_{\phi}\|_{\ell^\infty(L^2)}$} 		&  8.0353 E-3   & 2.1979 E-3  &  5.6626 E-4   &  1.4302e E-4 \\
		Rate & 		---         & 1.87       & 1.96        & 1.99 \\
		{$\| e_{\phi}\|_{\ell^\infty(H^1)}$} 		&  5.1221 E-2  & 1.4802 E-2   &  3.8982 E-4   & 9.9678 E-4 \\
		Rate 		& ---        & 1.79        & 1.93        & 1.97 \\ 
		{$\| e_{\bm u}\|_{\ell^\infty(H^{1})}$} 		&   7.9893 E-2  &  2.2071 E-2   & 5.9266 E-3   &   1.5215 E-3 \\
		Rate 		& ---        & 1.86        & 1.90        & 1.96 \\
		{ $\|e_{\bm u}\|_{\ell^{2}(H^{2})} $ } 		& 6.1627 E-1  & 1.5842 E-1   & 4.0338e E-2   & 1.0189 E-2 \\
		Rate		& ---      & 1.96        & 1.97        & 1.99 \\ 
		{ $\|e_{\bm u}\|_{\ell^{\infty}(L^{\infty})} $ } 		& 7.4744 E-3  & 2.6066 E-3   & 6.6708 E-4   & 1.7156 E-4 \\
		Rate		& ---      & 1.52        & 1.97        & 1.96 \\ 
		$\|e_{p}\|_{\ell^{\infty}(L^{2})}$		& 1.6444 E-2  &  4.2393 E-2  & 1.0734 E-2  &   2.6987 E-3 \\
		Rate		& ---      & 1.96        & 1.98        & 1.99 \\ 
		$\|e_{p}\|_{\ell^2(H^1)}$		& 3.0631 E-1  & 6.3874 E-2  & 1.3115 E-2  &  2.7450 E-3 \\
		Rate		& ---      & 2.26        & 2.28        & 2.26 \\ 
		{$|e_R|_{\ell^{\infty}}$} 		& 8.6334 E-3  &  2.3125 E-3   & 5.9549 E-4   & 1.4881 E-4 \\
		Rate 		& ---       & 1.90        & 1.96       & 2.00 \\   
		\bottomrule
	\end{tabular*}}
	  \vspace{-8pt}
\end{table}

\begin{figure}  [!htbp] \small
	\centering 
	\subfigure[{$\phi$ at $t$ = 0, 0.2, 0.4, 0.8}]
	{
		\includegraphics[width=0.23\textwidth]{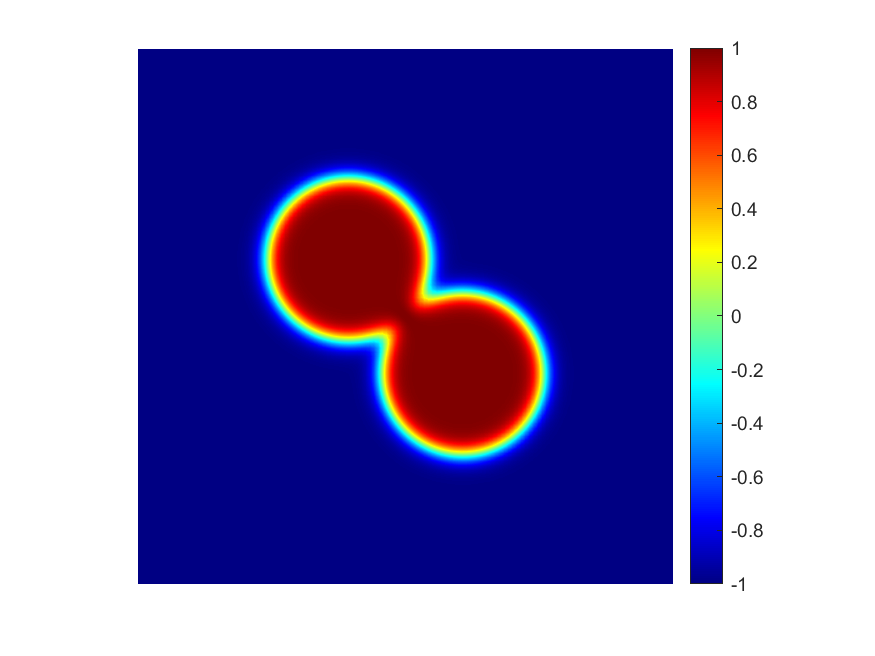}
		\includegraphics[width=0.23\textwidth]{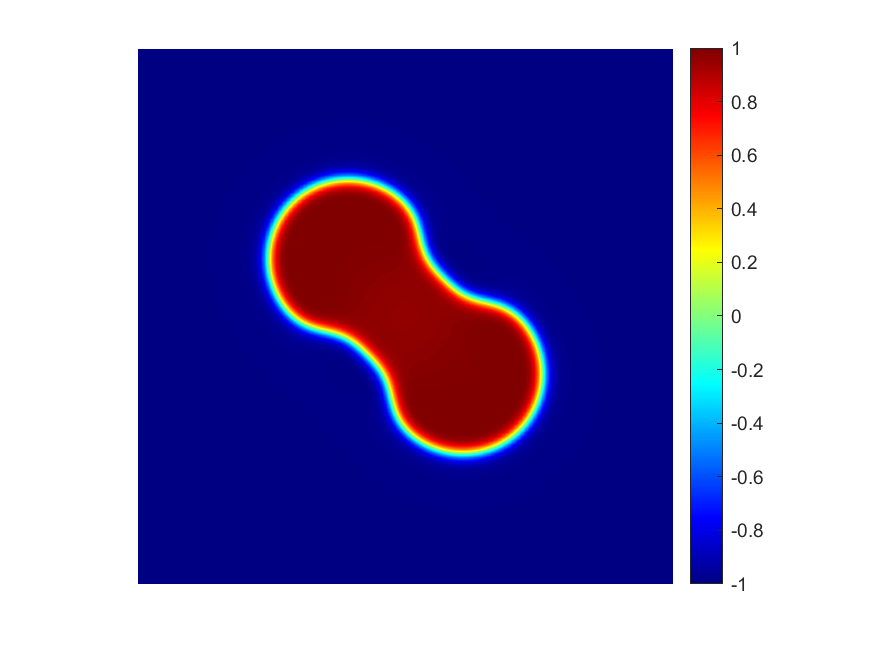}
		\includegraphics[width=0.23\textwidth]{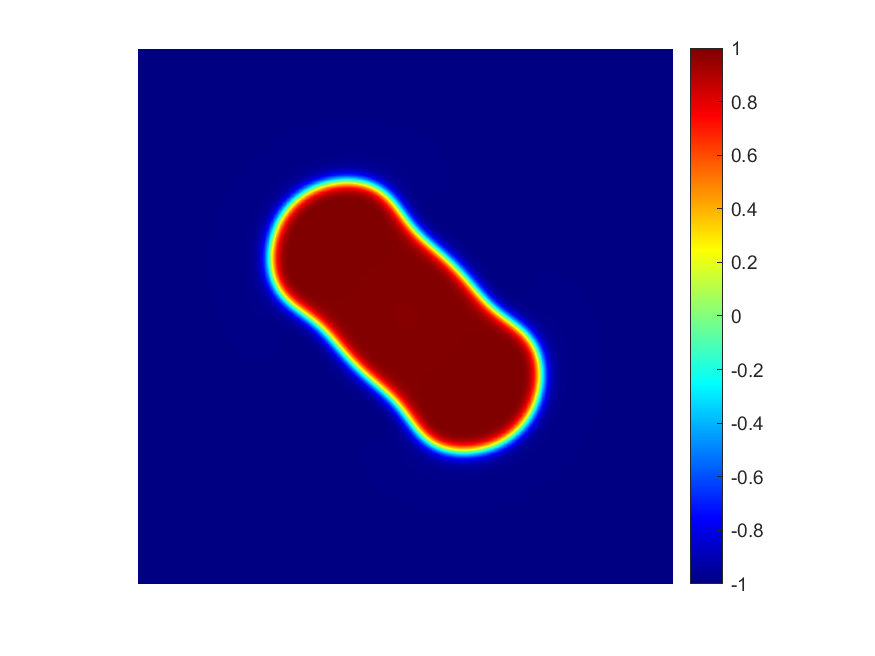}
		\includegraphics[width=0.23\textwidth]{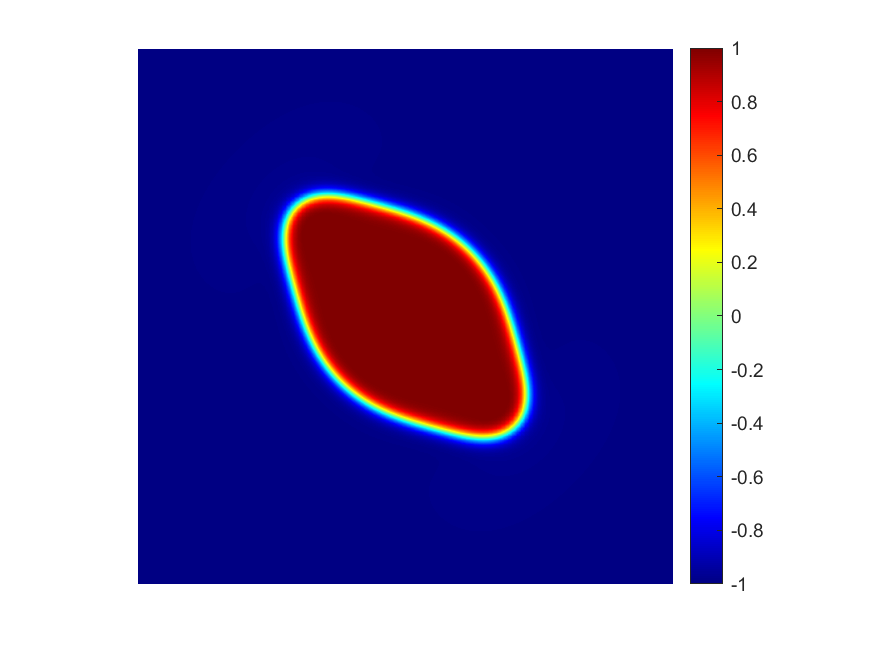}
	}
	\subfigure[{$\phi$ at $t$ = 1.2, 1.5, 2, 3 }]
	{
		\includegraphics[width=0.23\textwidth]{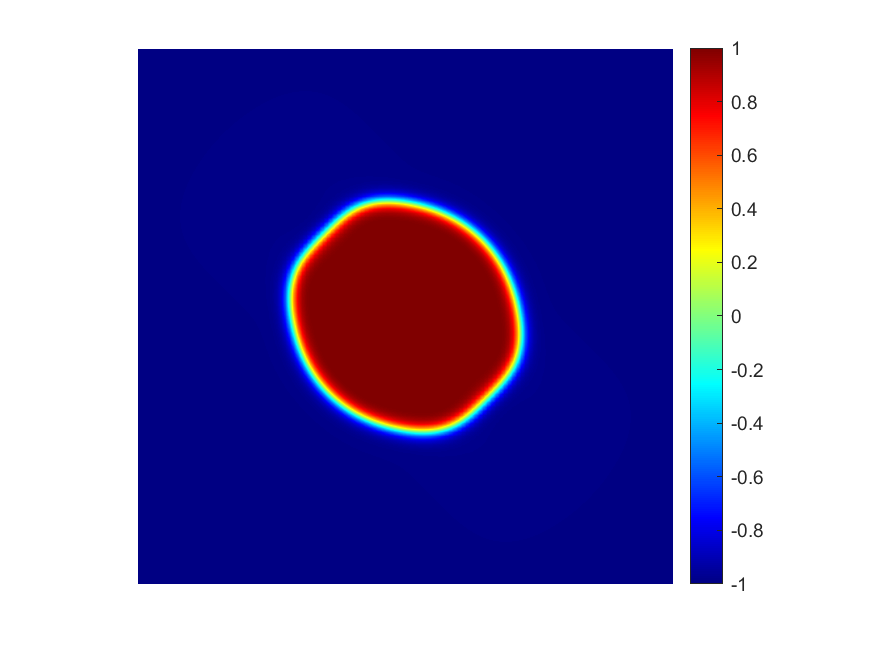}
		\includegraphics[width=0.23\textwidth]{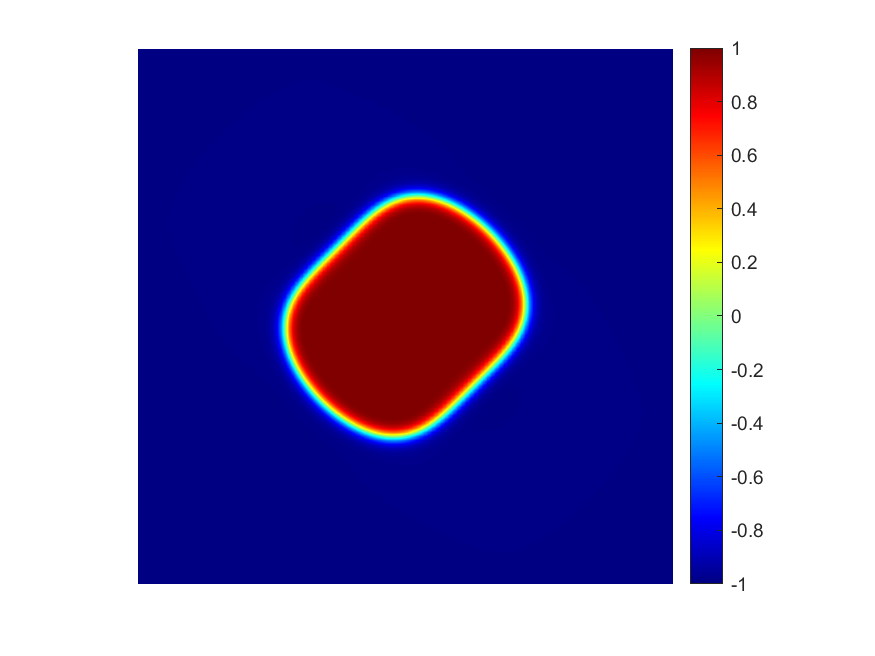}
		\includegraphics[width=0.23\textwidth]{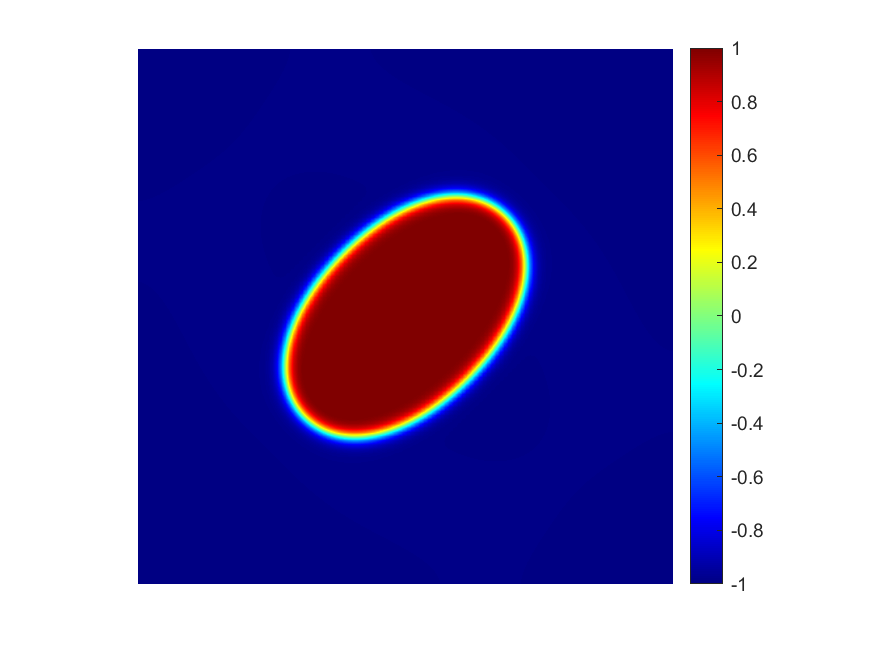}
		\includegraphics[width=0.23\textwidth]{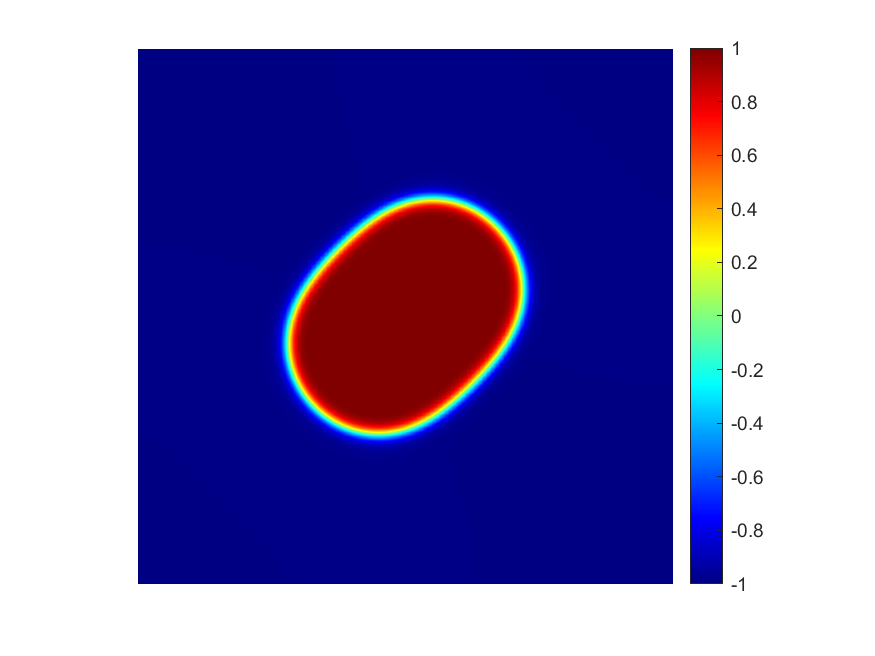}
	}
	\subfigure[{$\phi$ at $t$ =  3.4, 5, 8, 10}]
	{
		\includegraphics[width=0.23\textwidth]{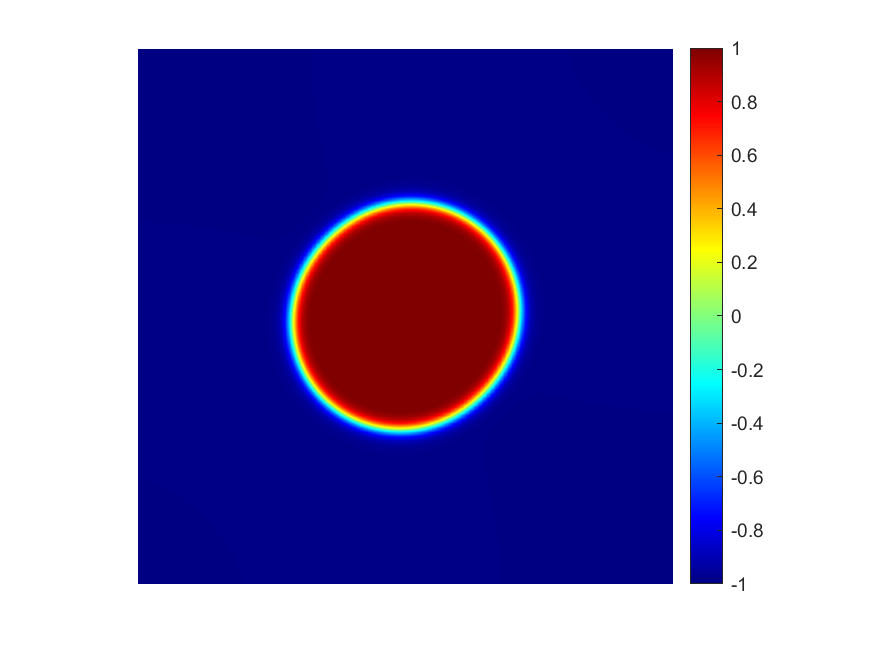}
		\includegraphics[width=0.23\textwidth]{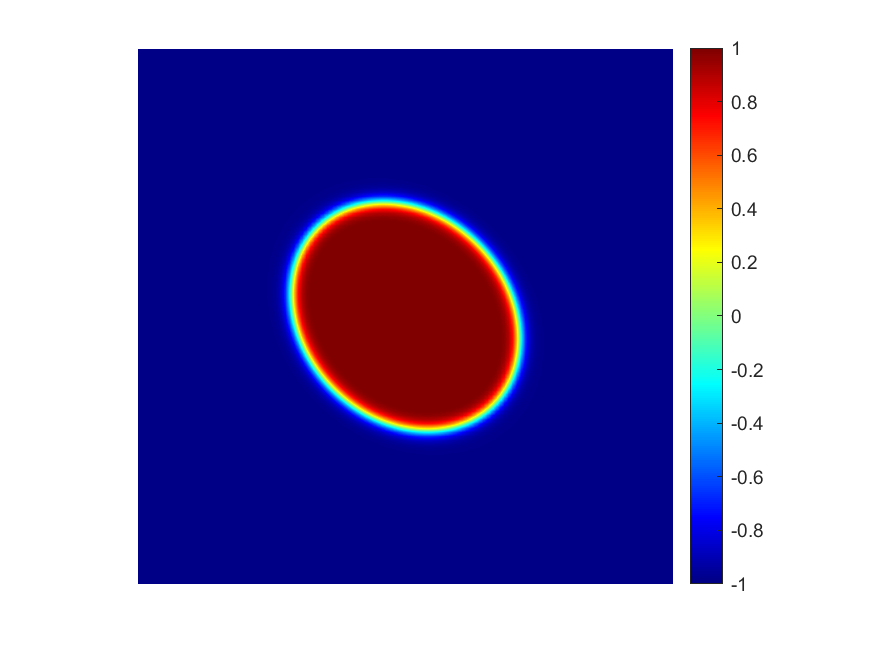}
		\includegraphics[width=0.23\textwidth]{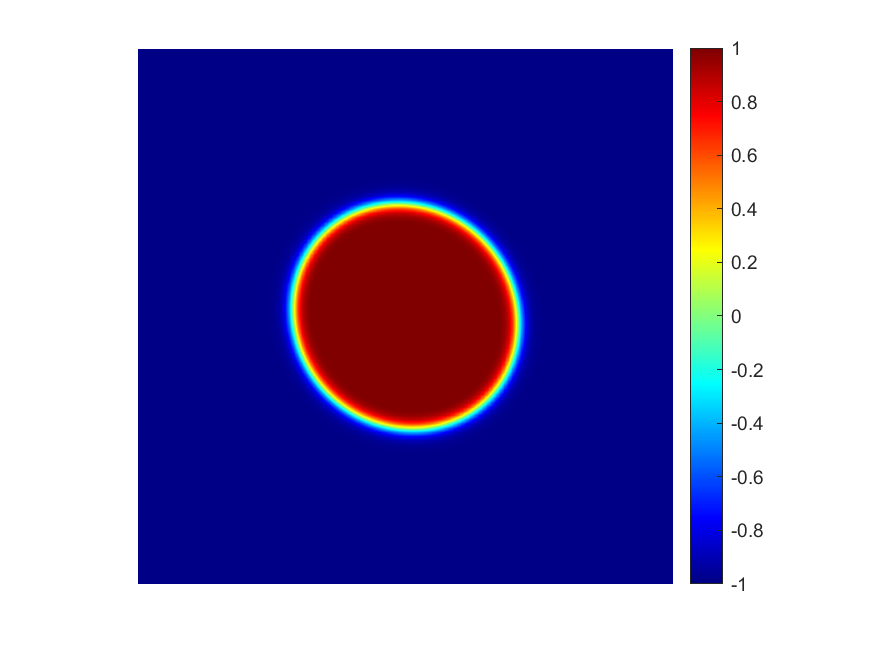}
		\includegraphics[width=0.23\textwidth]{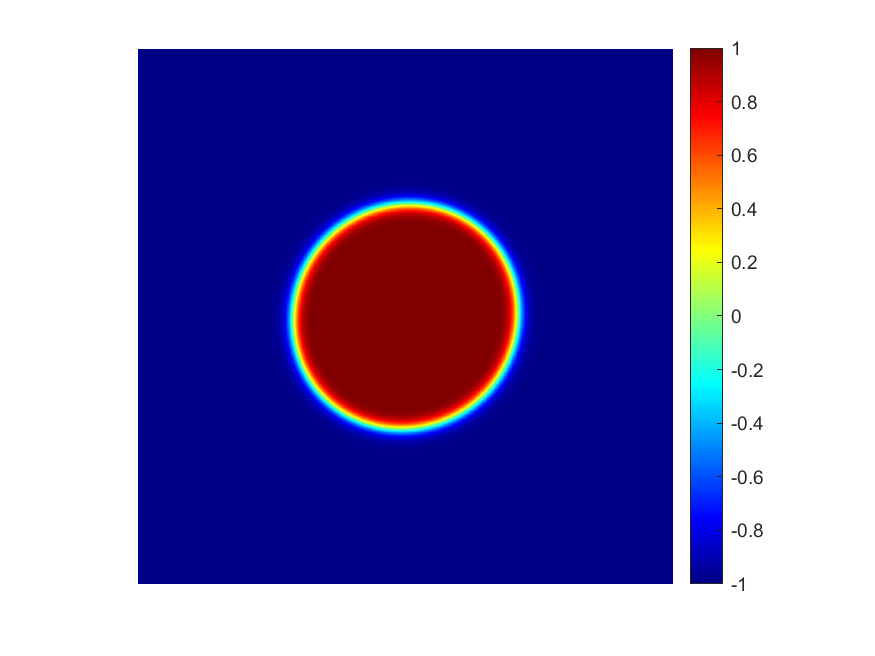}
	}
	\setlength{\abovecaptionskip}{0.01cm} 
	\setlength{\belowcaptionskip}{0.0cm}
	\caption{\small Snapshots of the phase-field variable for the dynamics of bubble merging  at different time instants.}
	\label{fig:tanh:phi}
	\vspace{-10pt}
\end{figure}
\subsection{Bubbles merging driven by surface tension}\label{subsec:e2}
In this example, we simulate the merging of two bubbles driven by surface tension in the domain $\Omega=\left[0, 1\right]^2$, following a setup similar to \cite{guo2014numerical, zhao2021second}. The initial conditions are provided as $\bm{u}(x, y, 0) =\bm{0}$, $p(x, y, 0)={0}$ and
\begin{equation*}
	\phi(x, y, 0) =1-\tanh \frac{-r+\sqrt{(x-x_a)^2+(y-y_a)^2}}{2 s}  -\tanh \frac{-r+\sqrt{(x-x_b)^2+(y-y_b)^2}}{2 s}  
\end{equation*}
with $x_a=0.5 -\frac{r}{\sqrt{2}}$, $y_a=0.5 +\frac{r}{\sqrt{2}}$, $x_b=0.5 +\frac{r}{\sqrt{2}}$, $y_b=0.5 -\frac{r}{\sqrt{2}}$, $r=0.15$ and $s = 10^{-2}$. The parameters are chosen as $M=10^{-2}$,  $\nu=10^{-3}$, $\lambda=10^{-4}$, $\epsilon=10^{-2}$, with uniform meshes $N=128$ and time step $\tau=10^{-3}$.  

The dynamics of bubble merging in a hydrodynamic environment are illustrated in Figs. \ref{fig:tanh:phi}--\ref{fig:tanh:u} at various time instants, which presents the evolution of the phase-field variable $\phi$ and associated velocity fields at different time instants. As expected, the two bubbles eventually merge into a single droplet, which stabilizes into a round shape. Initially, the droplet undergoes a squeezing phase, followed by stretching, with back-and-forth damping oscillations due to the exchange between kinetic energy and Helmholtz free energy. The two bubbles eventually merge into a big elliptical bubble, which subsequently deforms and contorts into a steady state circular bubble. The results computed by the first-order scheme \eqref{scheme:1th:phi}--\eqref{scheme:1th:div} are basically the same as those of the second-order scheme in \cite{zhao2021second}, confirming the reliability of the proposed method. Furthermore, the time evolutions of the energy and mass, shown in Fig. \ref{fig:tanh:E_M}, confirms the energy dissipation and mass conservation properties established in Theorem \ref{thm:energy}.
\begin{figure}  [!htbp] \small
	\centering 
	\subfigure[{$\phi$ at $t$ = 0, 0.2, 0.4, 0.8}]
	{
		\includegraphics[width=0.23\textwidth]{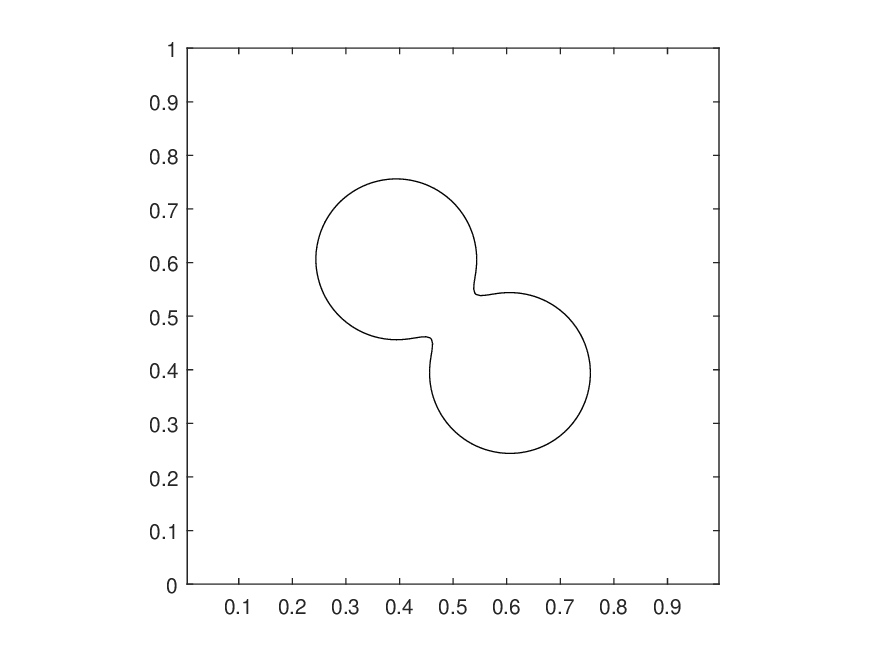}
		\includegraphics[width=0.23\textwidth]{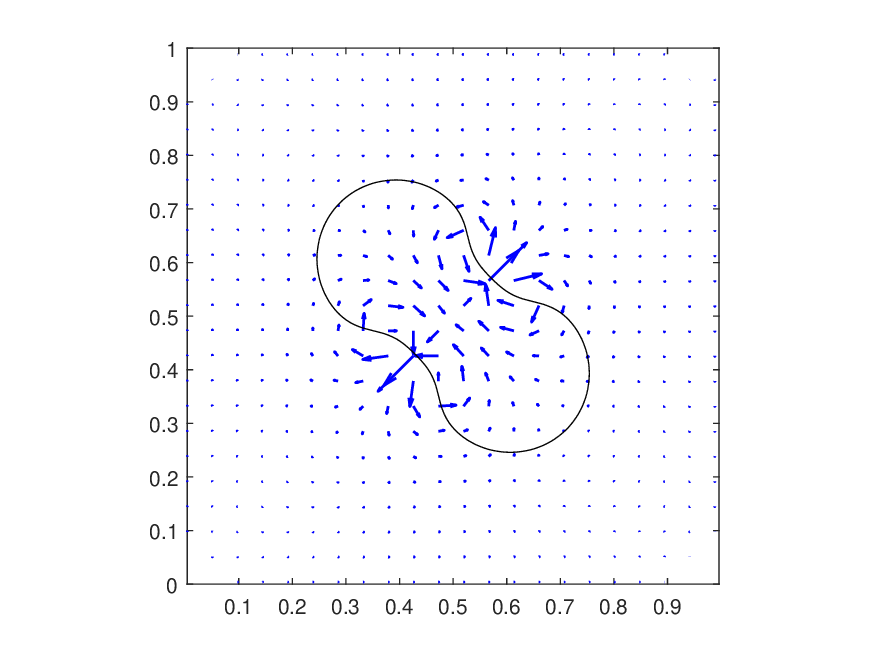}
		\includegraphics[width=0.23\textwidth]{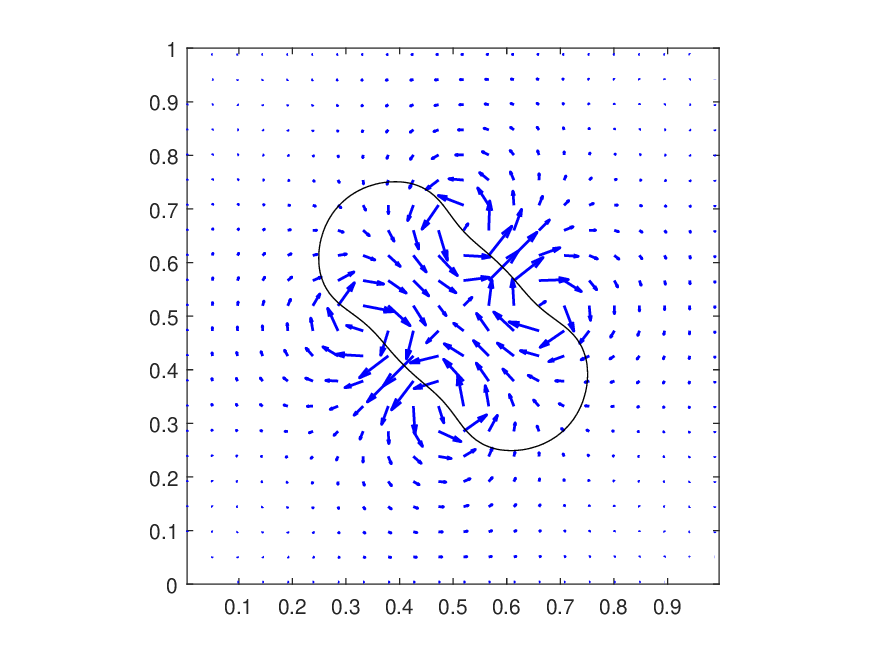}
		\includegraphics[width=0.23\textwidth]{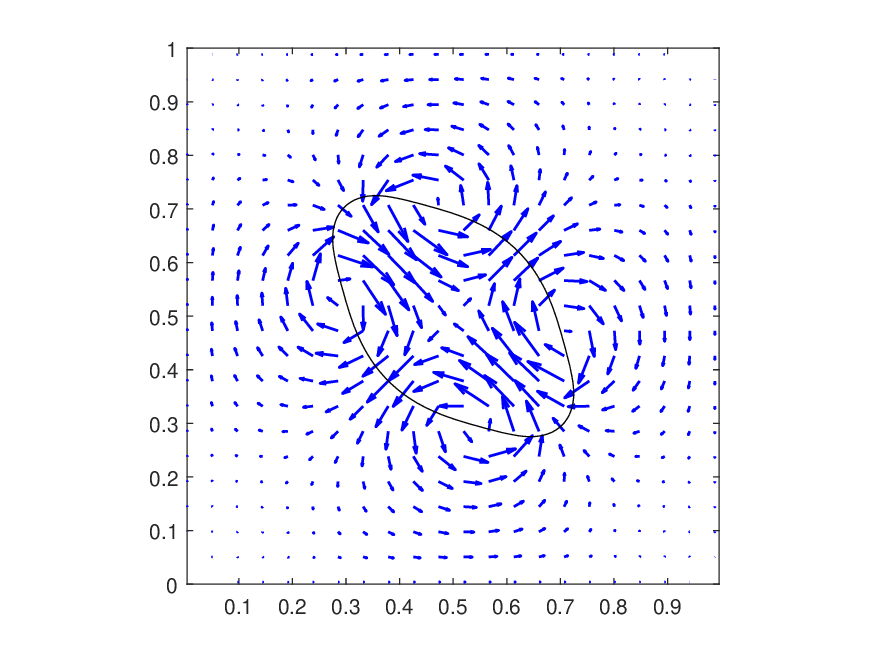}
	}
	\subfigure[{$\phi$ at $t$ = 1.2, 1.5, 2, 3 }]
	{
		\includegraphics[width=0.23\textwidth]{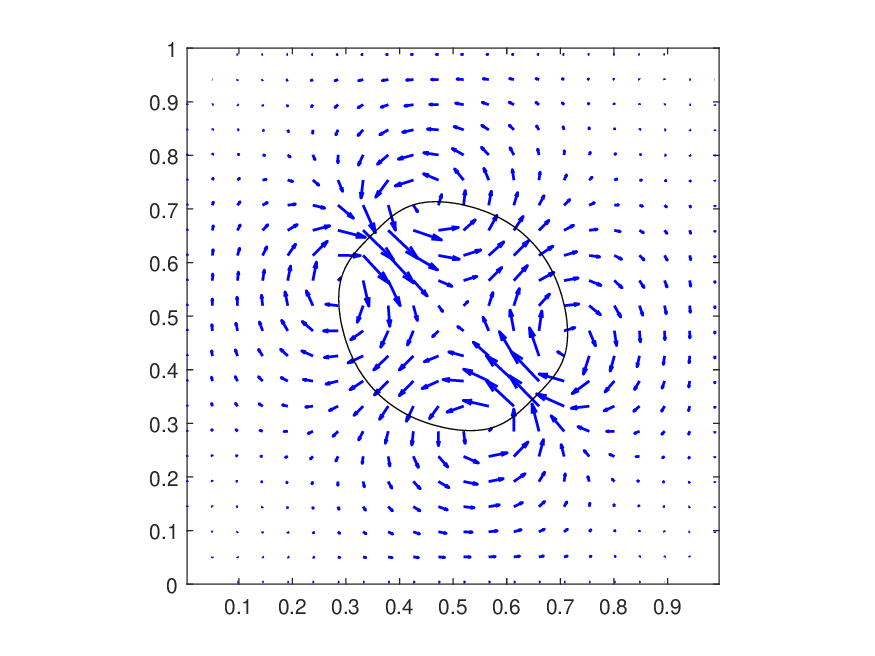}
		\includegraphics[width=0.23\textwidth]{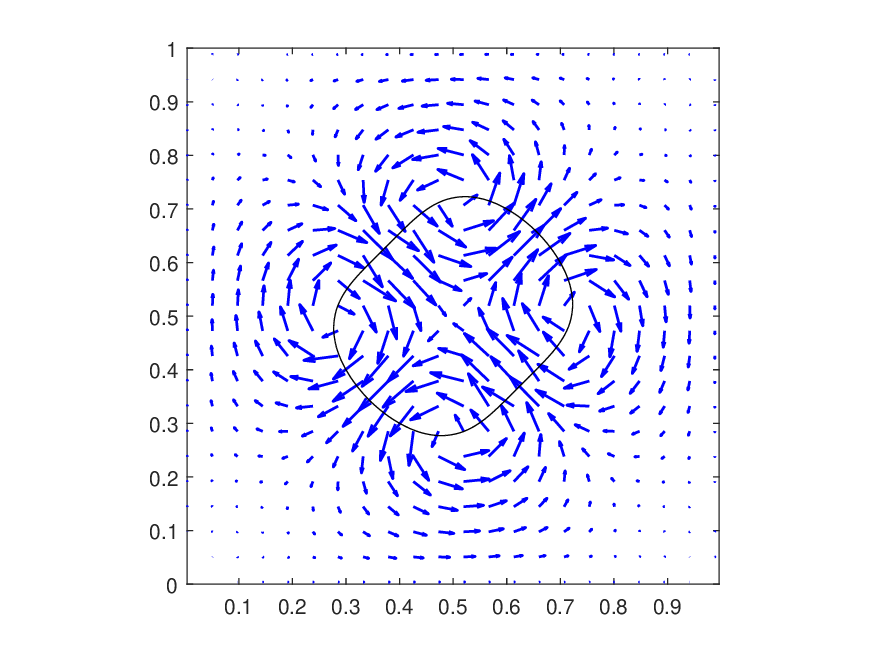}
		\includegraphics[width=0.23\textwidth]{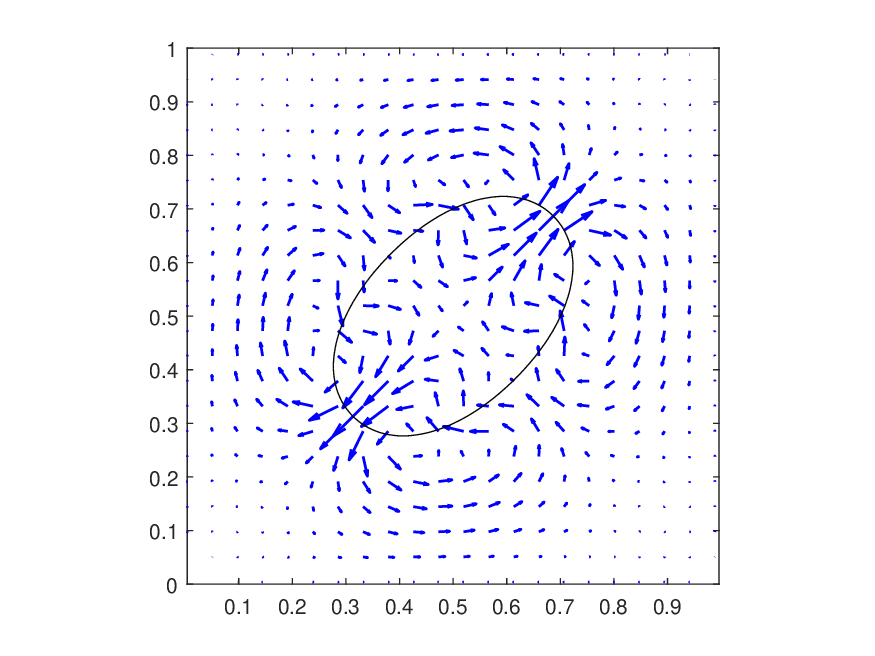}
		\includegraphics[width=0.23\textwidth]{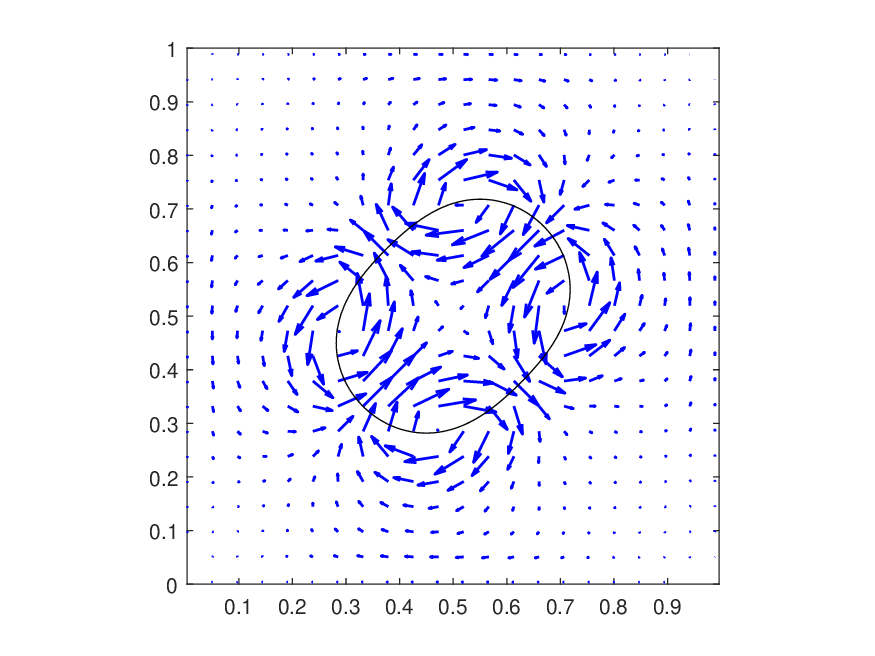}
	}
	\subfigure[{$\phi$ at $t$ = 3.4, 5, 8, 10}]
	{
		\includegraphics[width=0.23\textwidth]{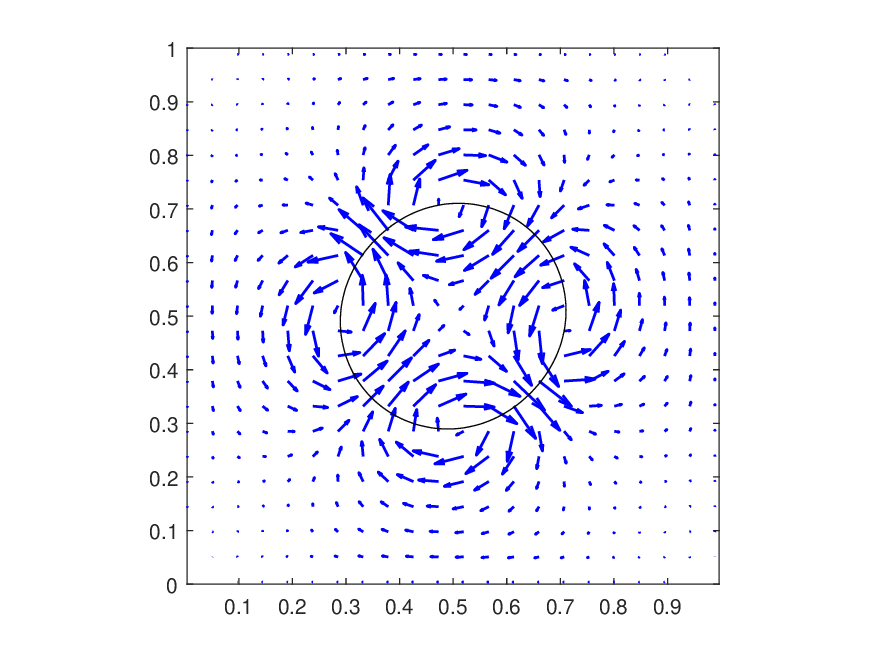}
		\includegraphics[width=0.23\textwidth]{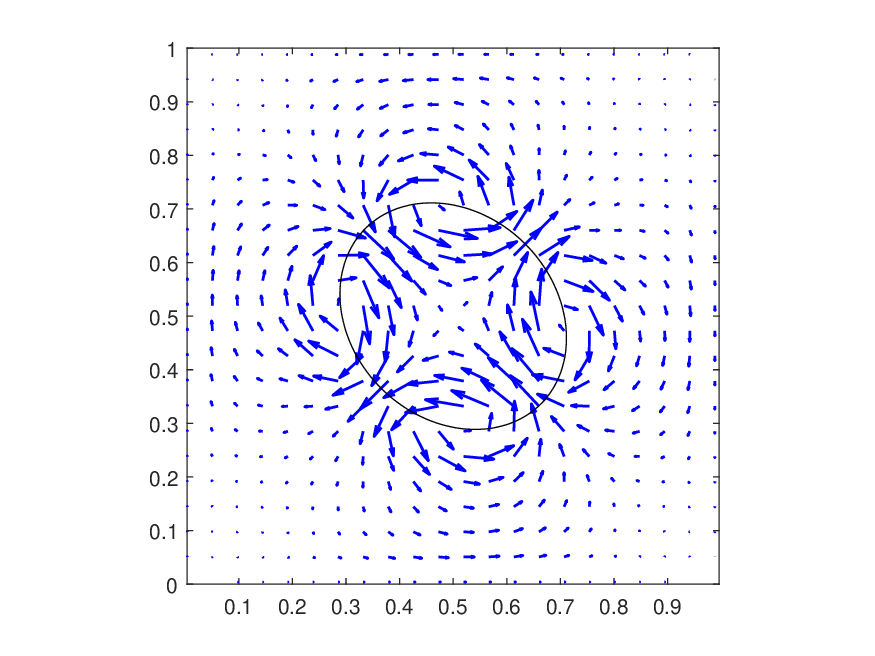}
		\includegraphics[width=0.23\textwidth]{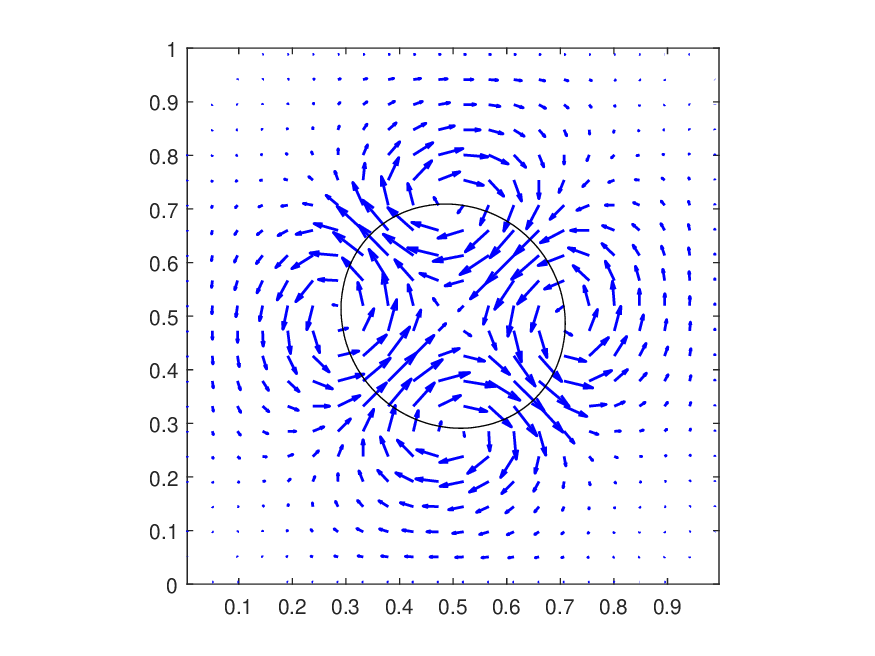}
		\includegraphics[width=0.23\textwidth]{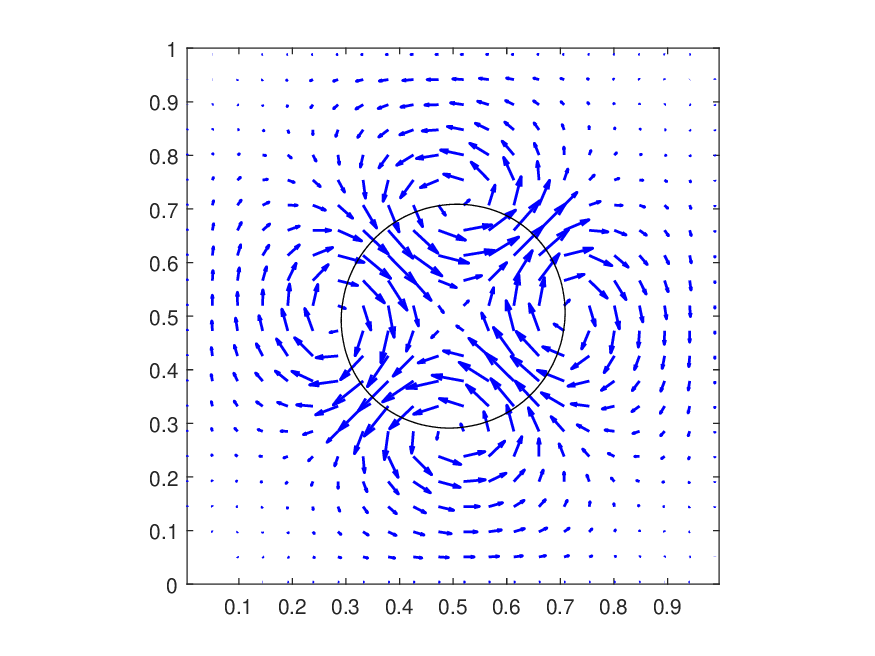}
	}
	\setlength{\abovecaptionskip}{0.01cm} 
	\setlength{\belowcaptionskip}{0.0cm}
	\caption{\small  Visualization of the velocity field  for the dynamics of bubble merging at different time instants.}
	\label{fig:tanh:u}
\end{figure}
 \begin{figure}  [!thbp] \small
	  \vspace{-8pt}
	  \centering 
	  \subfigure
	 	{
		            \includegraphics[width=0.38\textwidth]{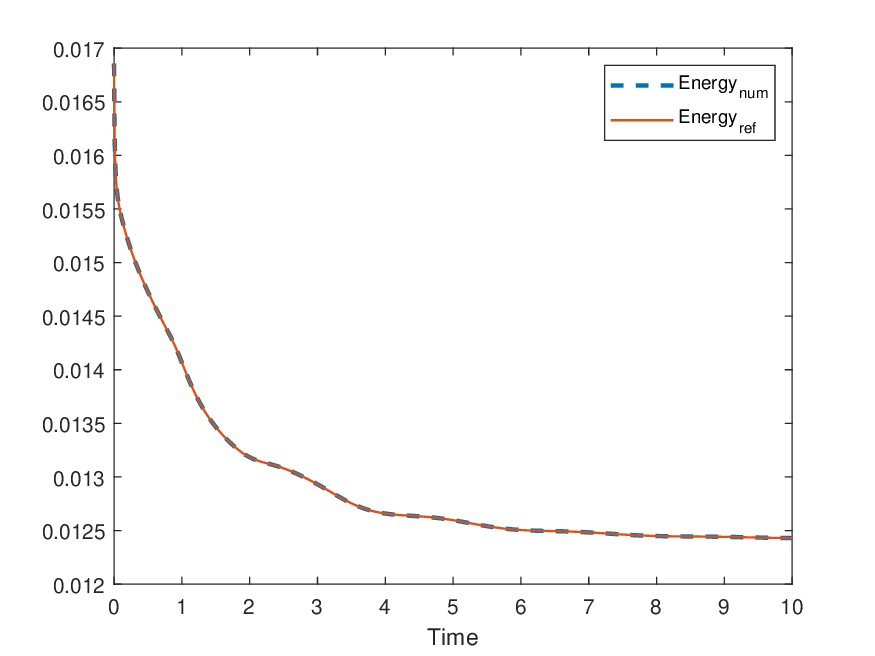}
		            \includegraphics[width=0.38\textwidth]{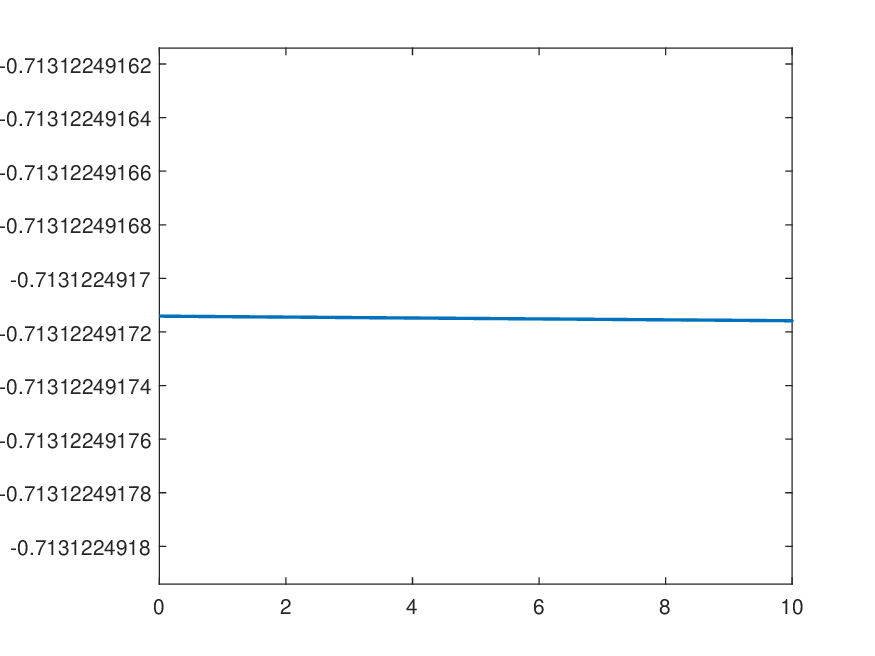}
		 }
	  \setlength{\abovecaptionskip}{0.01cm} 
	 \setlength{\belowcaptionskip}{0.0cm}
	  \caption{\small  Evolutions of energy and mass at different time instants.}	  \label{fig:tanh:E_M}
	  	  \vspace{-8pt}
	 \end{figure}

\subsection{Flow-coupled phase separation}\label{subsec:e3}
In this example, we simulate the flow-coupled nucleation process in the computational domain $\Omega=[0,1]^2$ via the second-order IVS scheme \eqref{scheme:phi:2th}--\eqref{scheme:div:2th}, following the initial conditions given in \cite{li2023global}:
\[
\bm u(x, y,0)=\mathbf{0}, ~ p(x, y, 0)={0}, ~ \phi(x, y,0)=2y-1+ \phi_0,
\]
where $\phi_0$ represents the random distribution values between $-0.01$ and $0.01$.
The phase variable $\phi$ is initially larger near the upper and lower boundaries and smaller at the center of the domain, as shown in Fig. \ref{fig:rand:phi0}. In this simulation, we fix the physical parameters $M=10^{-1}$, $\lambda=10^{-5}$ and $\epsilon=10^{-2}$, and choose the computational parameters $\delta_0 = \beta=0$, $\tau=10^{-3}$, $N=100$.

\begin{figure}  [!htbp] \small
	\centering 
	\subfigure
	{
		\includegraphics[width=0.38\textwidth] {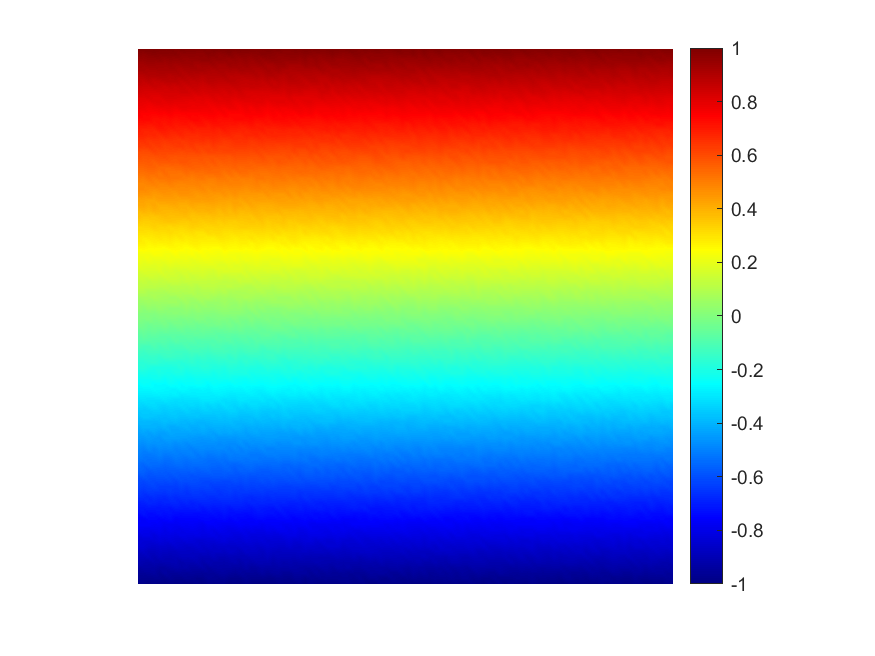}
	}
	\setlength{\abovecaptionskip}{0.01cm} 
	\setlength{\belowcaptionskip}{0.0cm}
	\caption{\small Snapshot of the initial phase-field variable.}
	\label{fig:rand:phi0}
\end{figure}
\begin{figure}[!htbp] 
	\centering
	\includegraphics[width=0.24\textwidth]{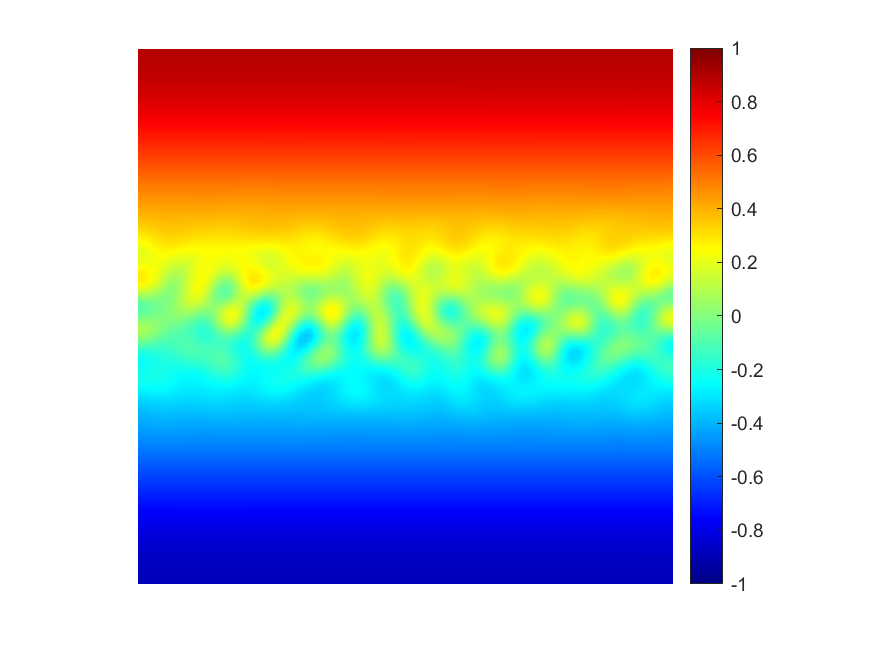}
	\includegraphics[width=0.24\textwidth]{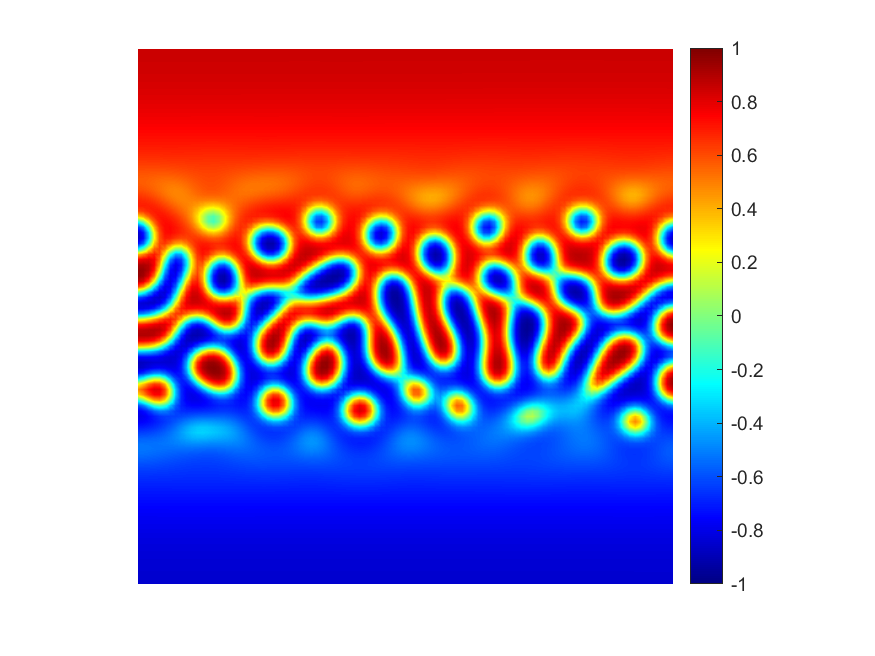}
	\includegraphics[width=0.24\textwidth]{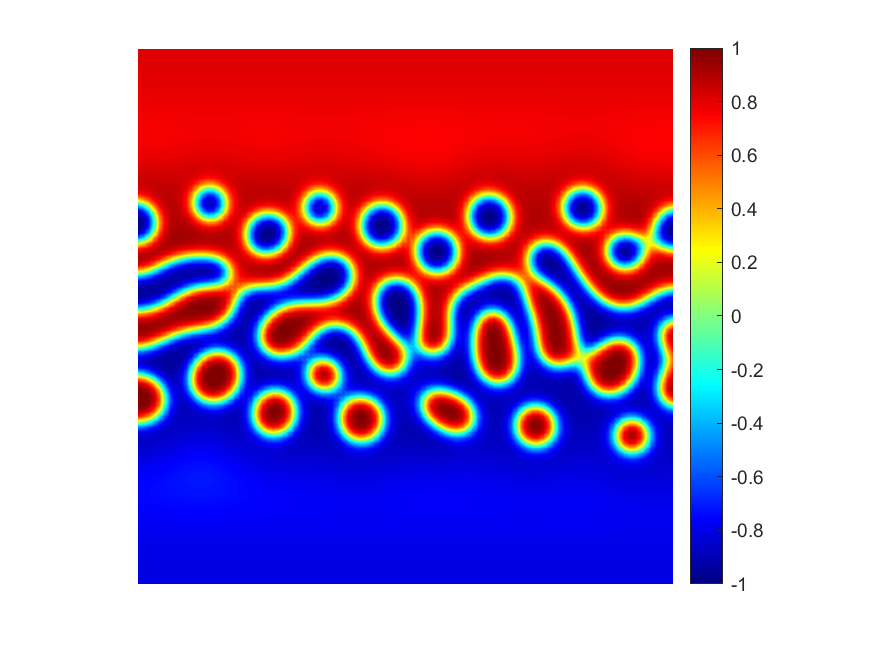}
	\includegraphics[width=0.24\textwidth]{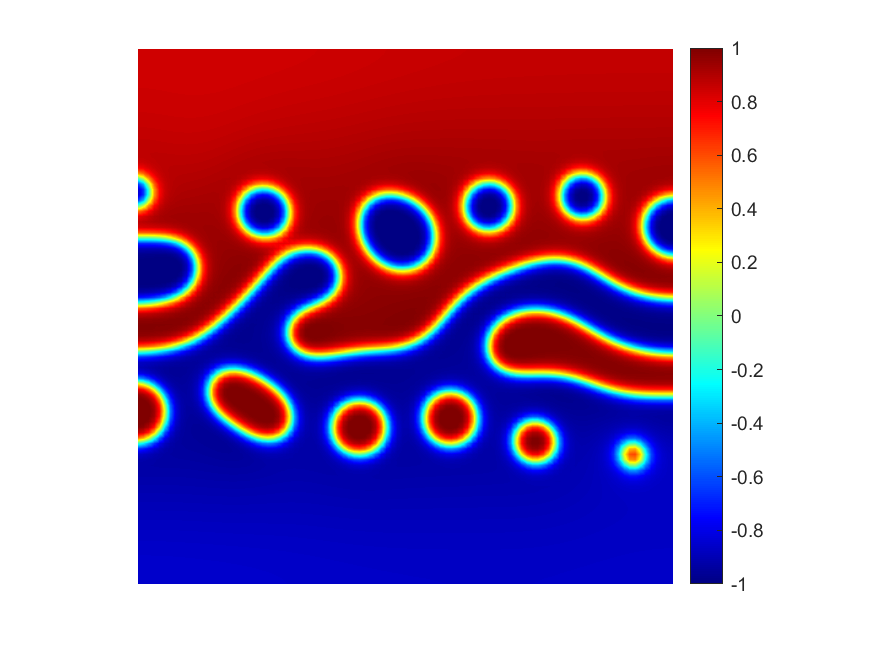}
	\includegraphics[width=0.24\textwidth]{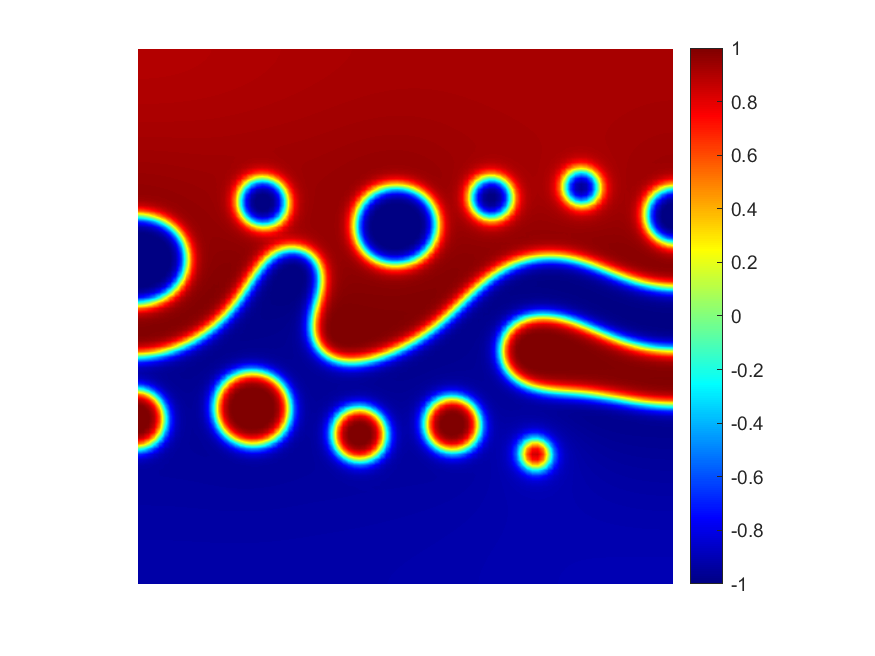}
	\includegraphics[width=0.24\textwidth]{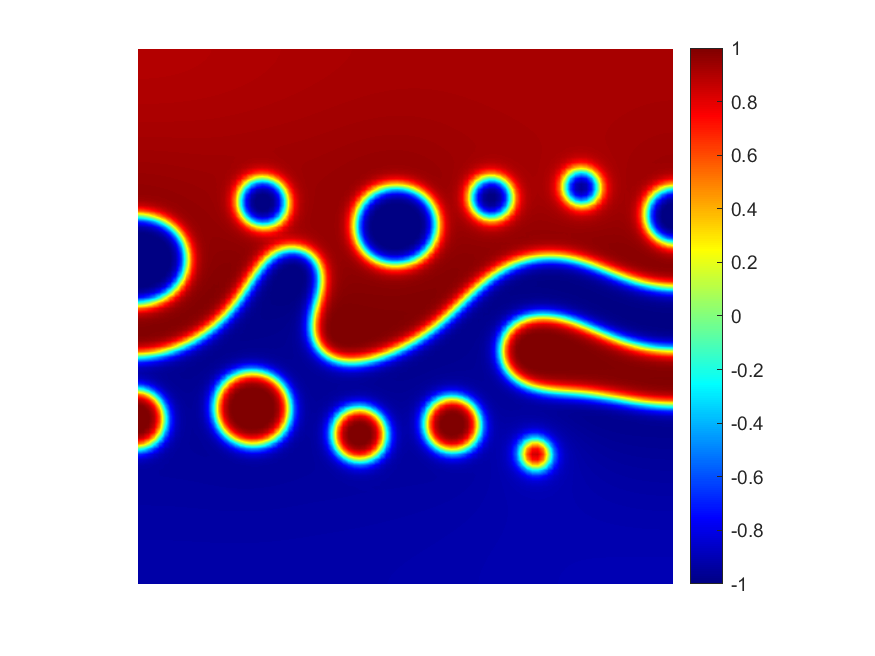}
	\includegraphics[width=0.24\textwidth]{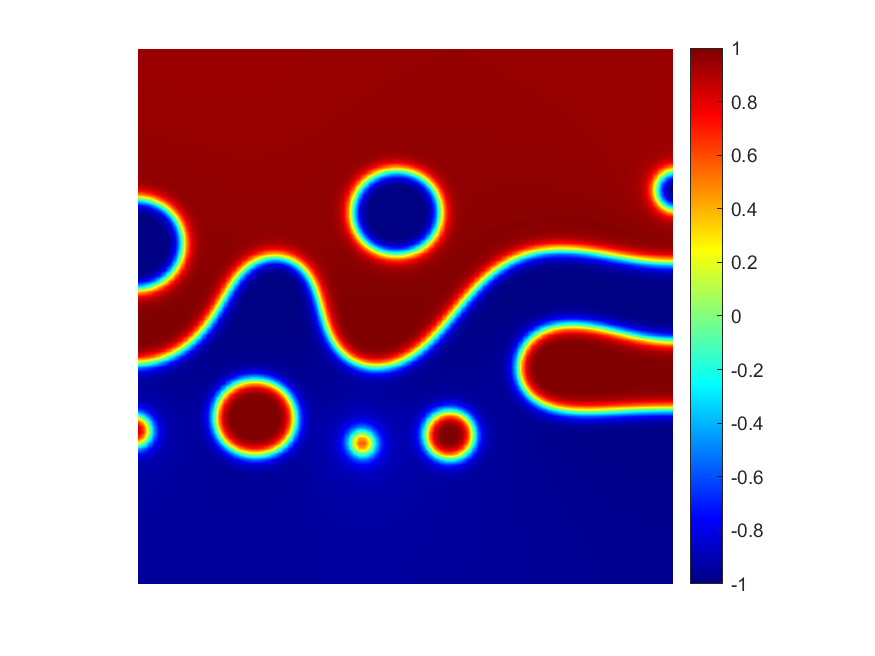}
	\includegraphics[width=0.24\textwidth]{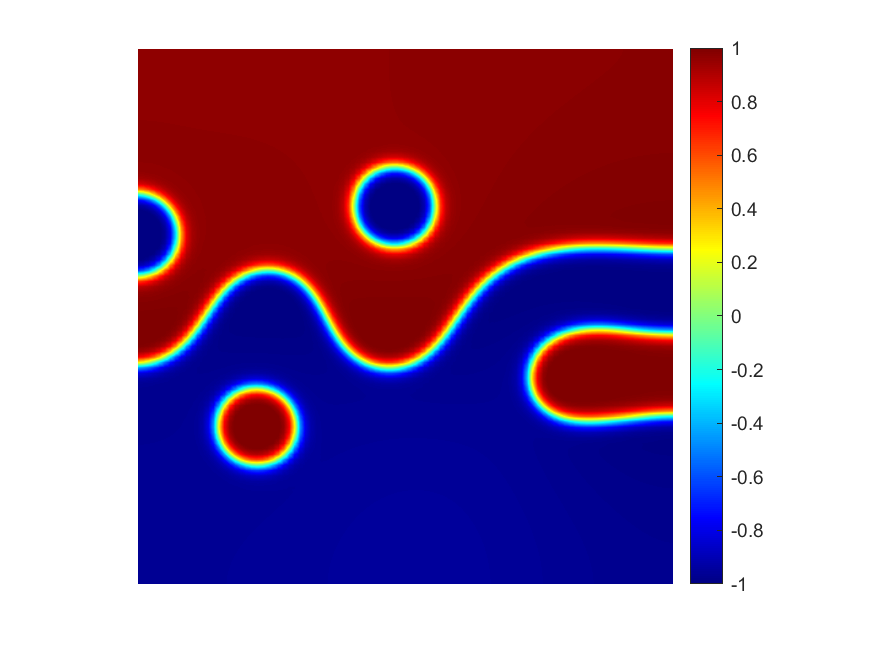}
	\caption{\small  Snapshots of the phase-field variable for the flow-coupled nucleation process with $\nu=1$ at $t= 0.2, 0.5, 1, 3, 5, 10, 15$ and $20$, respectively.}
	\label{fig:rand:phi:nu1}
\end{figure}
\begin{figure}[!htbp] 
	\centering
	\includegraphics[width=0.24\textwidth]{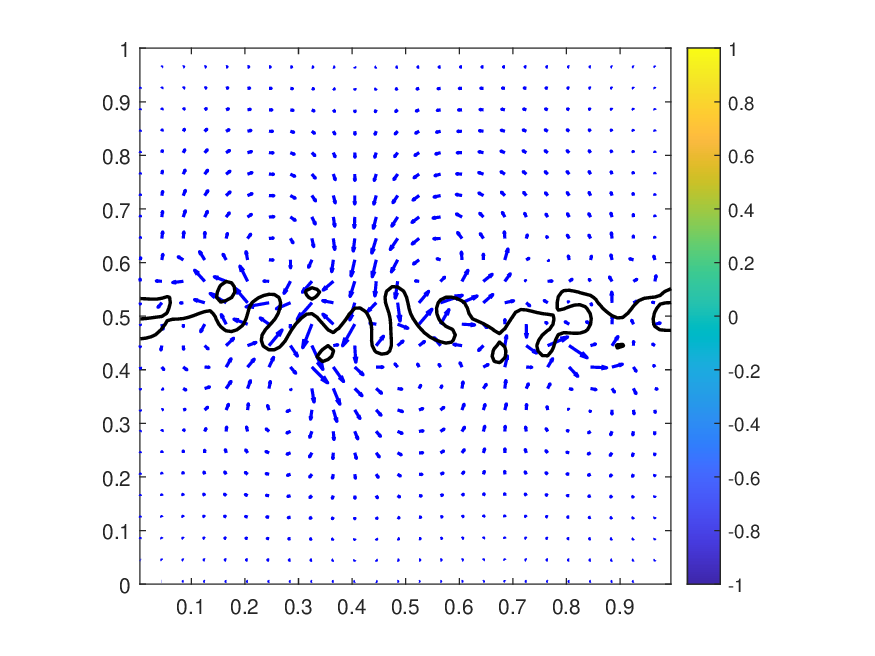}
	\includegraphics[width=0.24\textwidth]{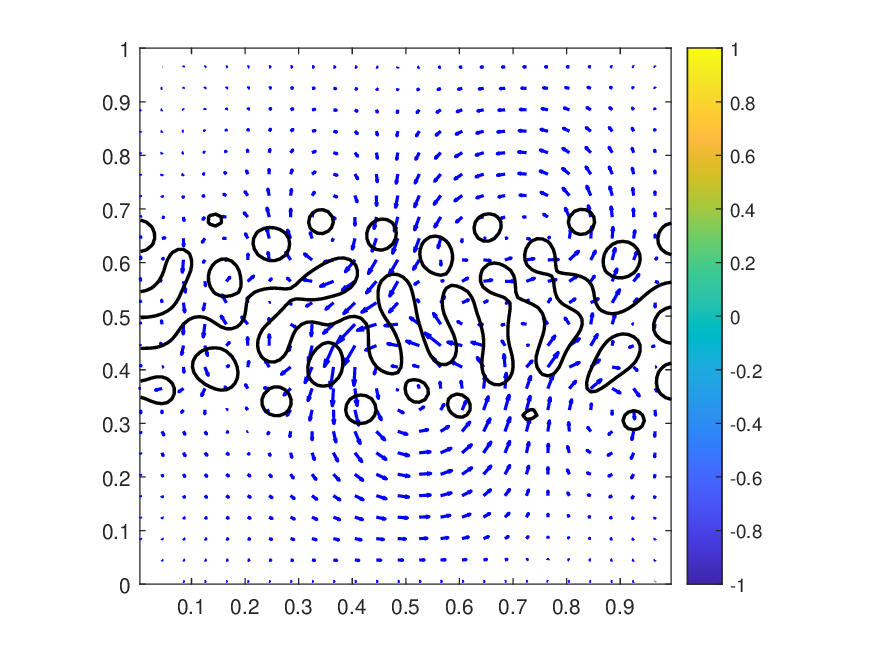}
	\includegraphics[width=0.24\textwidth]{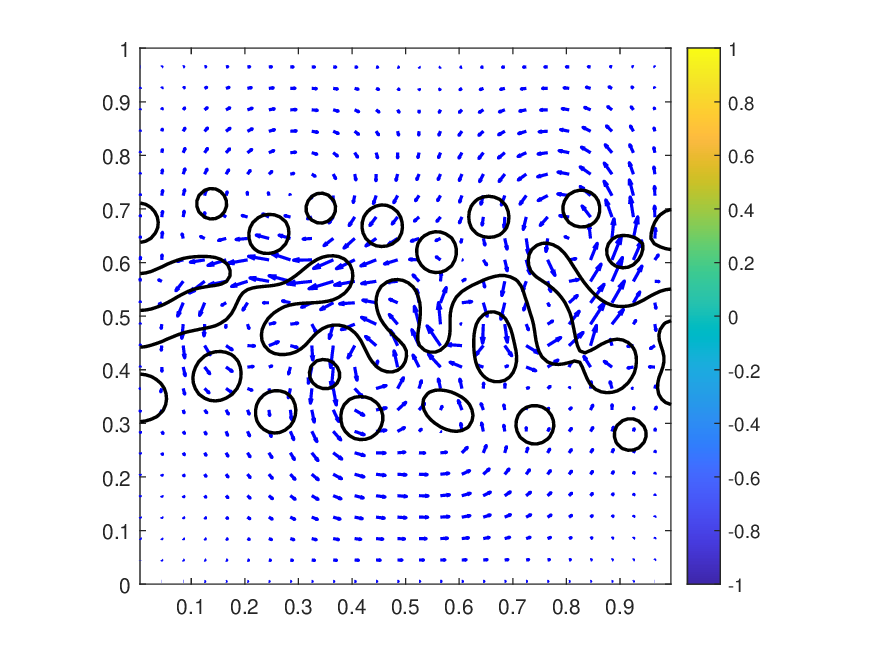}
	\includegraphics[width=0.24\textwidth]{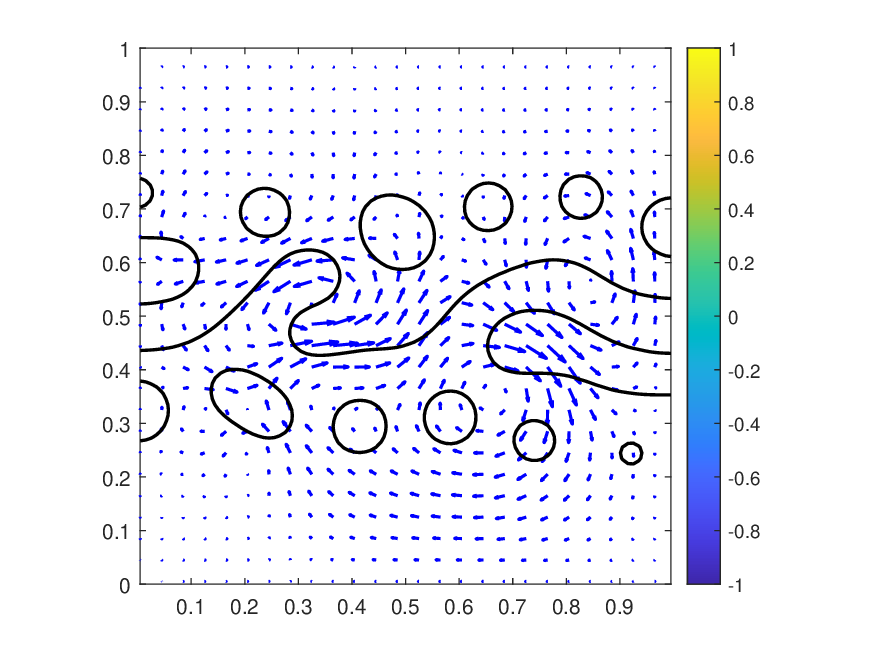}
	\includegraphics[width=0.24\textwidth]{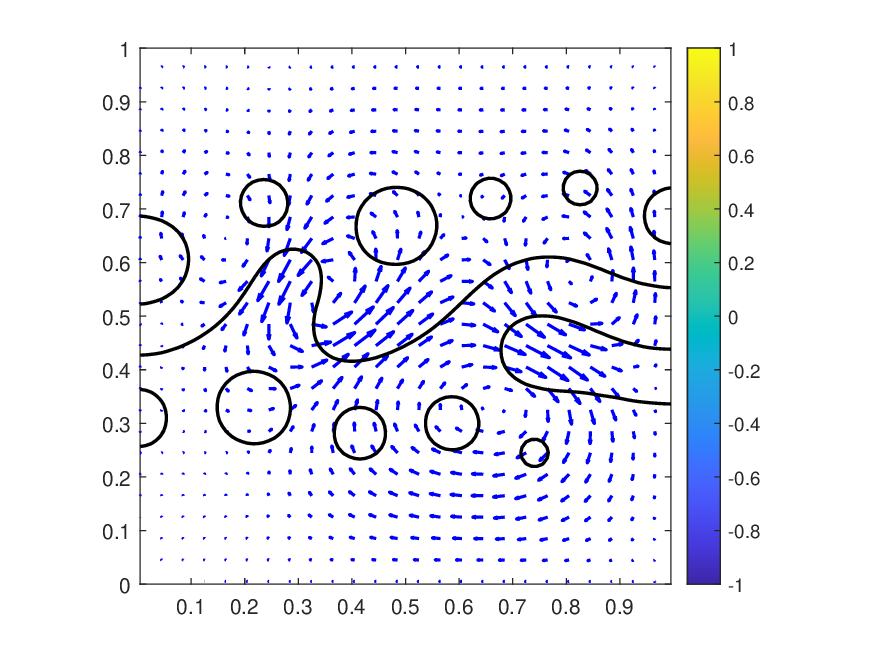}
	\includegraphics[width=0.24\textwidth]{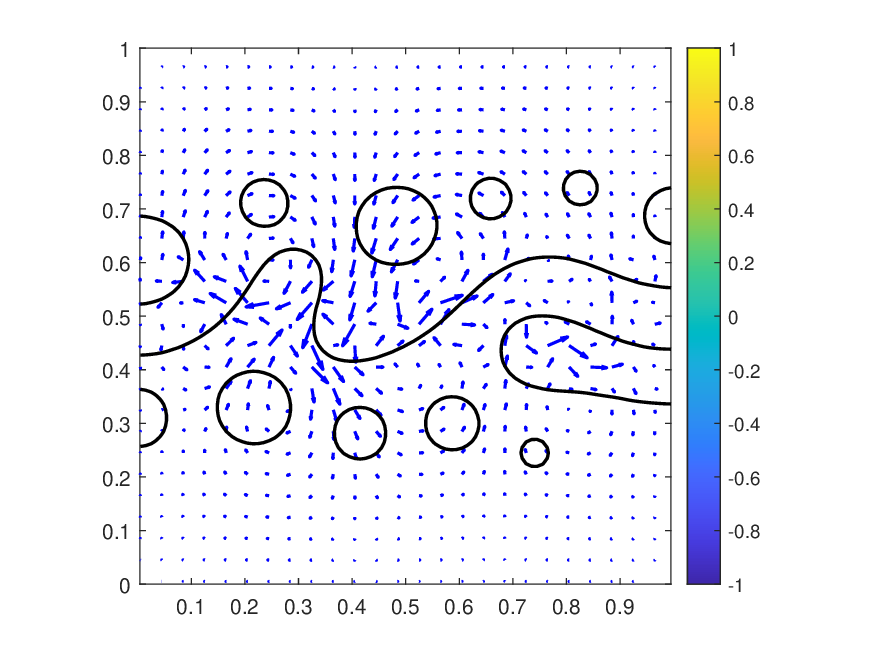}
	\includegraphics[width=0.24\textwidth]{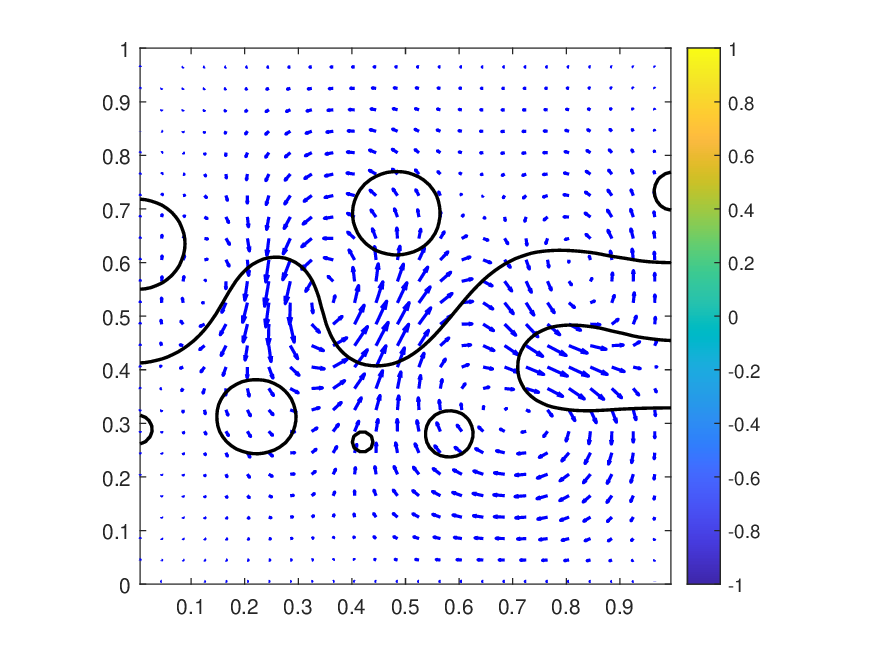}
	\includegraphics[width=0.24\textwidth]{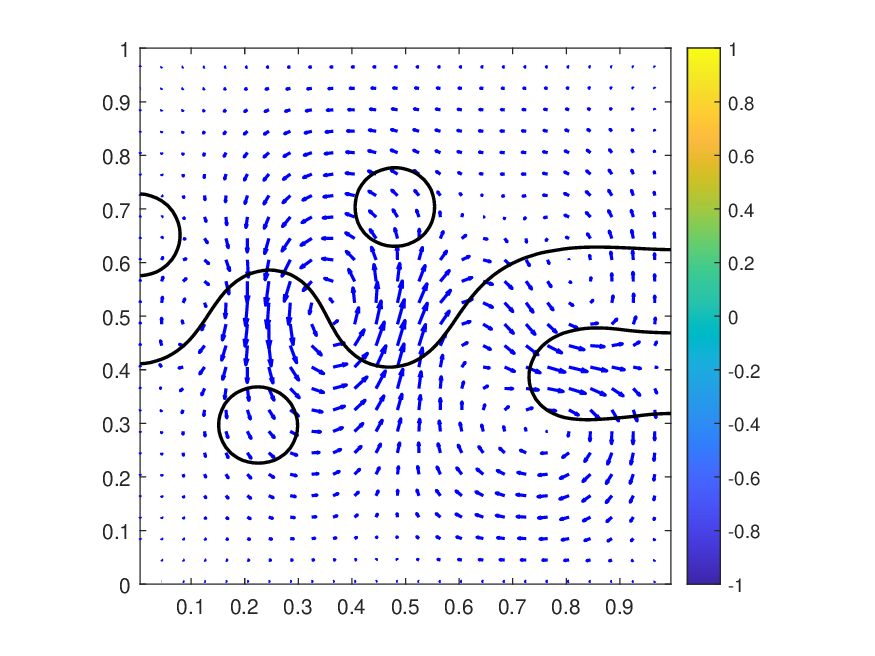}
	\caption{\small  Snapshots of the velocity field  for the flow-coupled nucleation process with $\nu=1$ at $t= 0.2, 0.5, 1, 3, 5, 10, 15$ and $20$ respectively.}
	\label{fig:rand:u:nu1}
\end{figure}
\begin{figure}[!htbp] 
	\centering
	\includegraphics[width=0.24\textwidth]{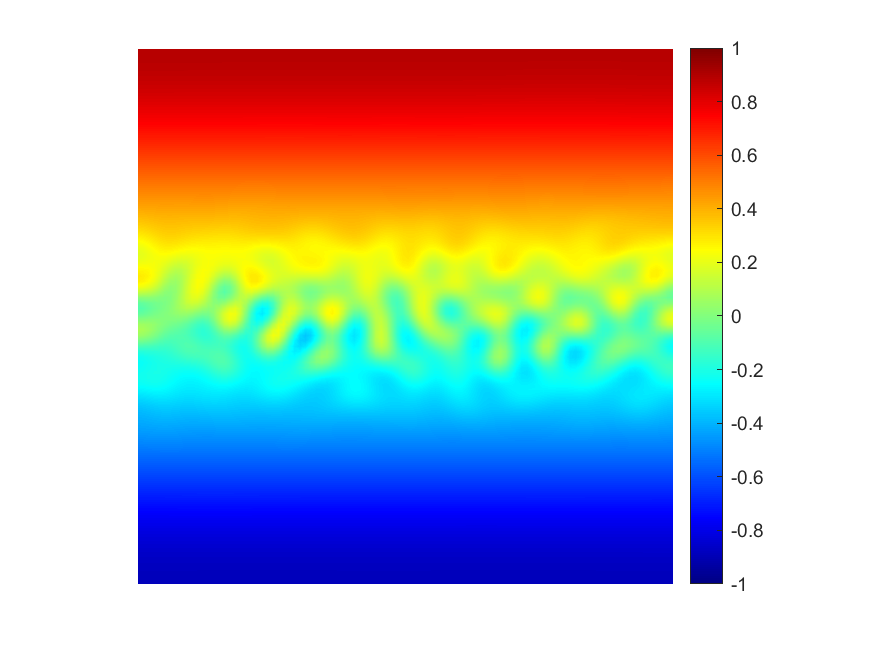}
	\includegraphics[width=0.24\textwidth]{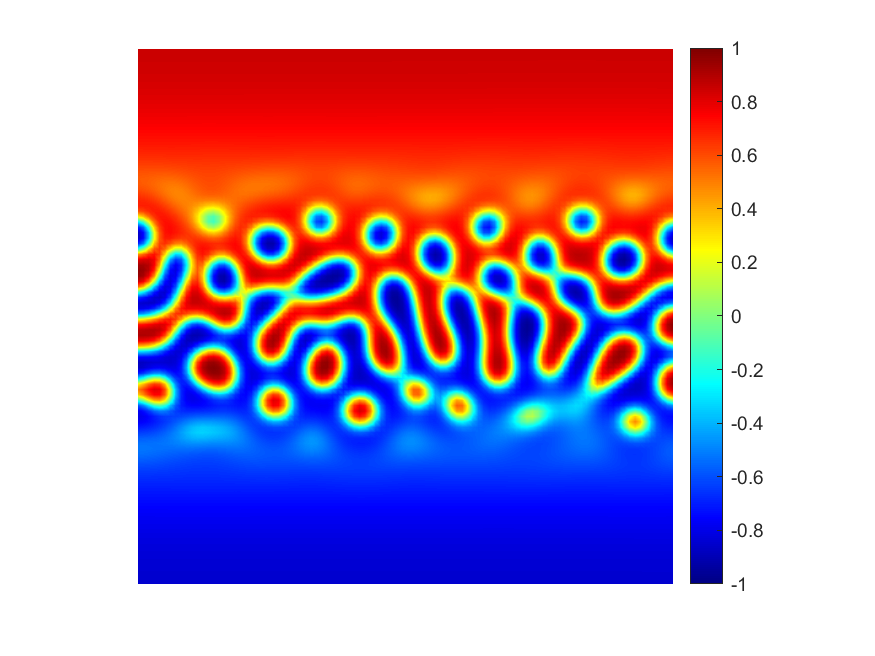}
	\includegraphics[width=0.24\textwidth]{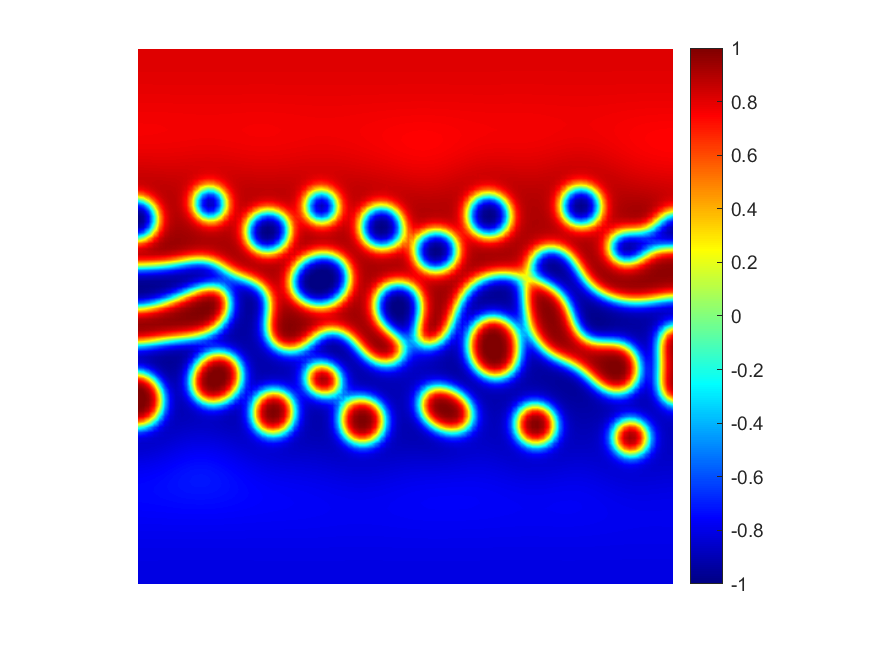}
	\includegraphics[width=0.24\textwidth]{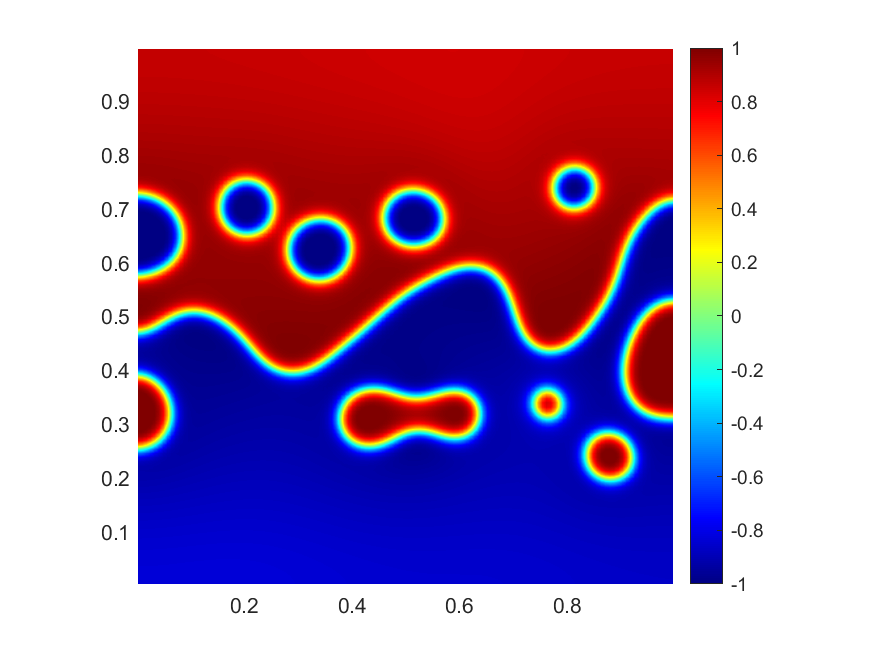}
	\includegraphics[width=0.24\textwidth]{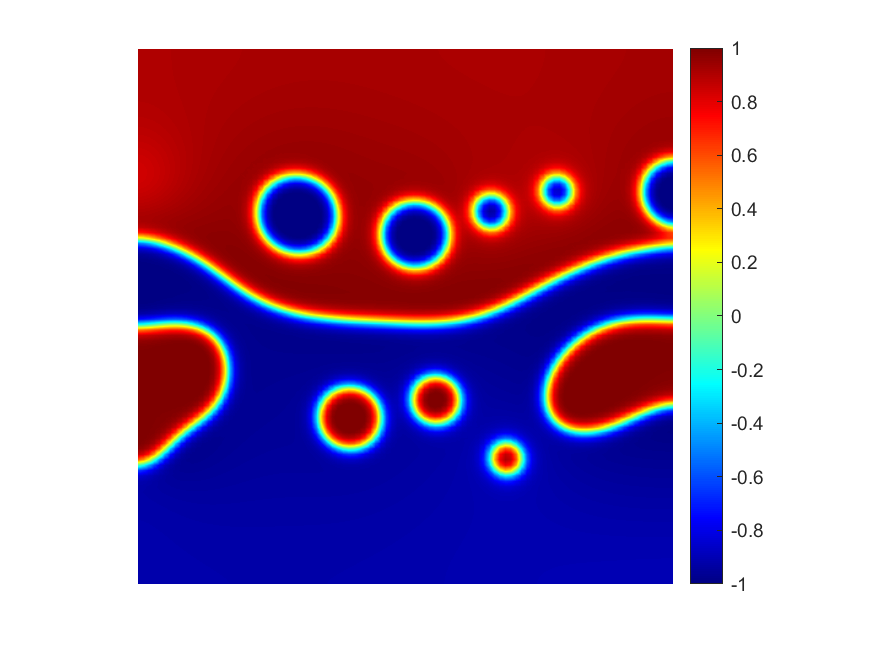}
	\includegraphics[width=0.24\textwidth]{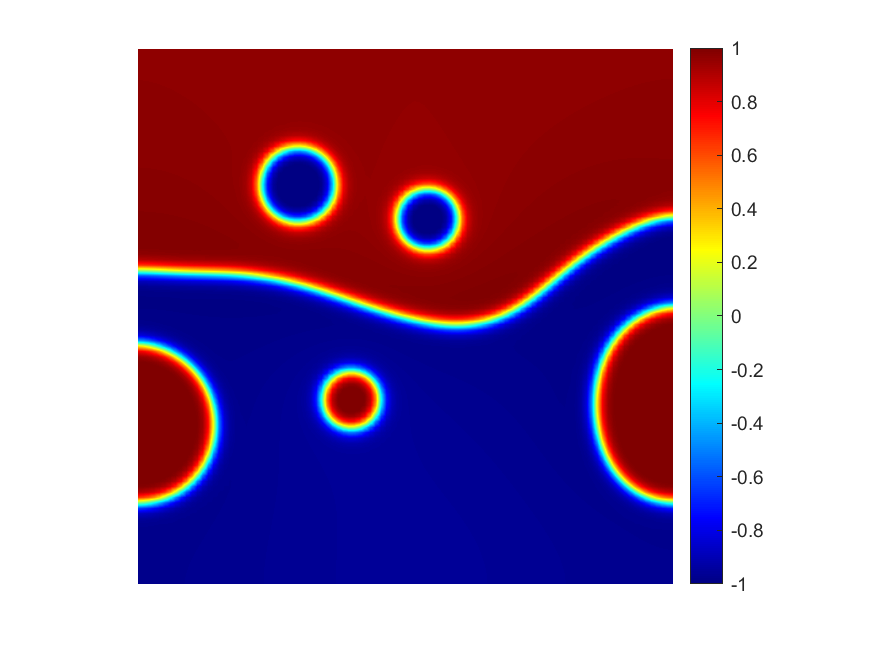}
	\includegraphics[width=0.24\textwidth]{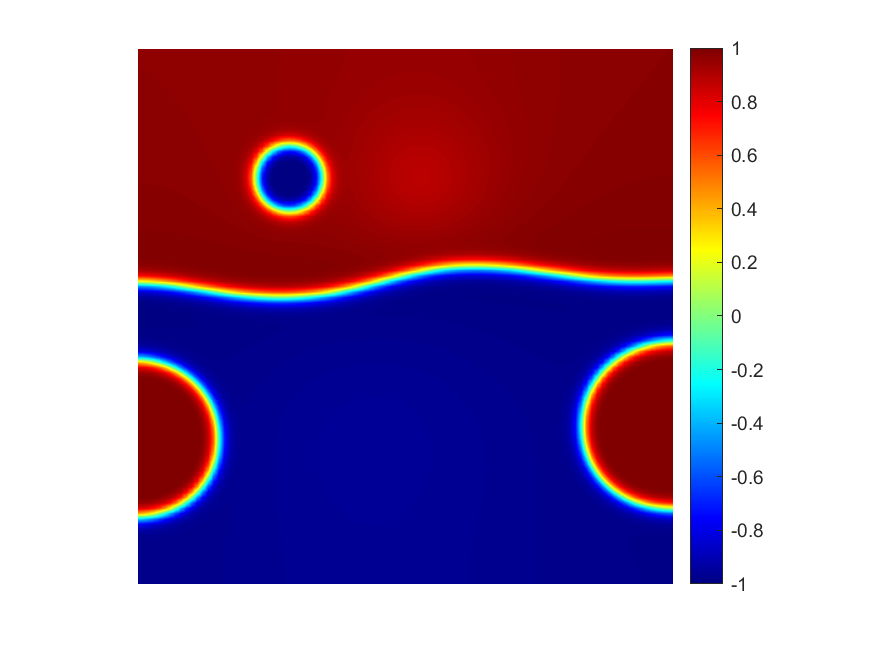}
	\includegraphics[width=0.24\textwidth]{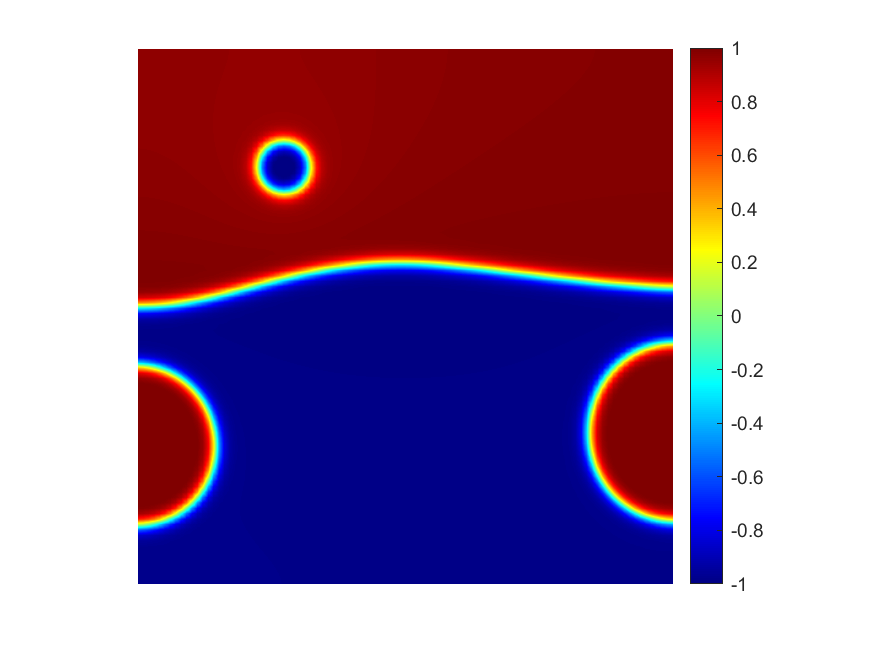}
	\caption{\small  Snapshots of the phase-field variable for the flow-coupled nucleation process with $\nu=10^{-3}$ at $t= 0.2, 0.5, 1, 3, 5,10, 15$ and $20$, respectively.}
	\label{fig:rand:phi:nu1e3}
\end{figure}
Numerical results for the viscosity parameter $\nu =1$ are shown in Figs. \ref{fig:rand:phi:nu1} and \ref{fig:rand:u:nu1}, which display the evolution of the phase variable $\phi$ and the velocity field at various time instants. We observe that the spinodal decomposition is more pronounced in the middle of the domain. Over time, the droplets gradually merge and decrease in number as the interfacial length shrinks due to energy dissipation \cite{li2023global}. For comparison, we also perform the simulation with identical parameters but a smaller viscosity $\nu=10^{-3}$. The profiles of the phase variable and the velocity field at the same time instants are plotted in Figs. \ref{fig:rand:phi:nu1e3} and \ref{fig:rand:u:nu1e3}. Although the droplets eventually disappear in both scenarios, the process of flow-coupled nucleation occurs more rapidly with smaller viscosity. 


\begin{figure}[!htbp] 
	\centering
	\includegraphics[width=0.24\textwidth]{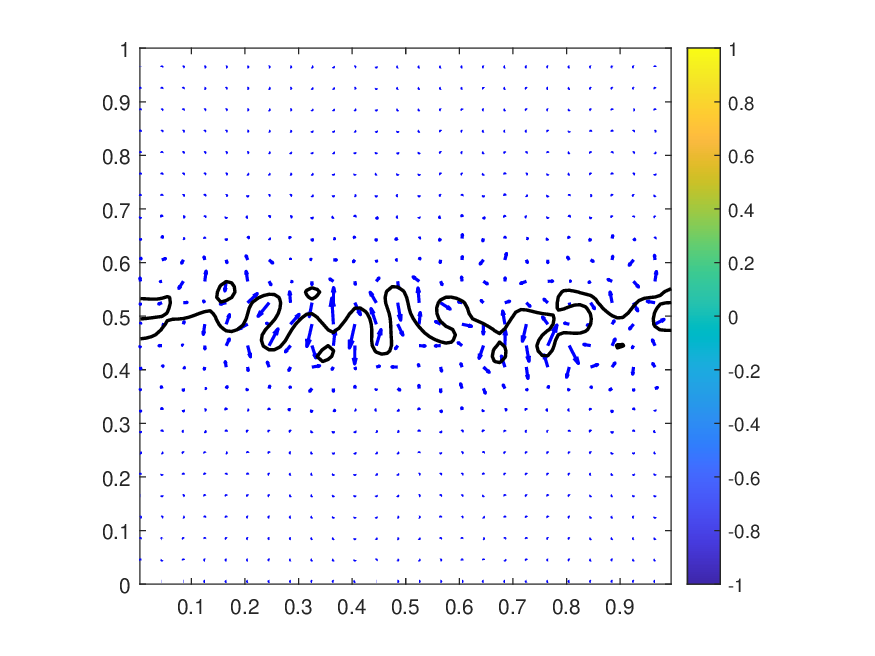}
	\includegraphics[width=0.24\textwidth]{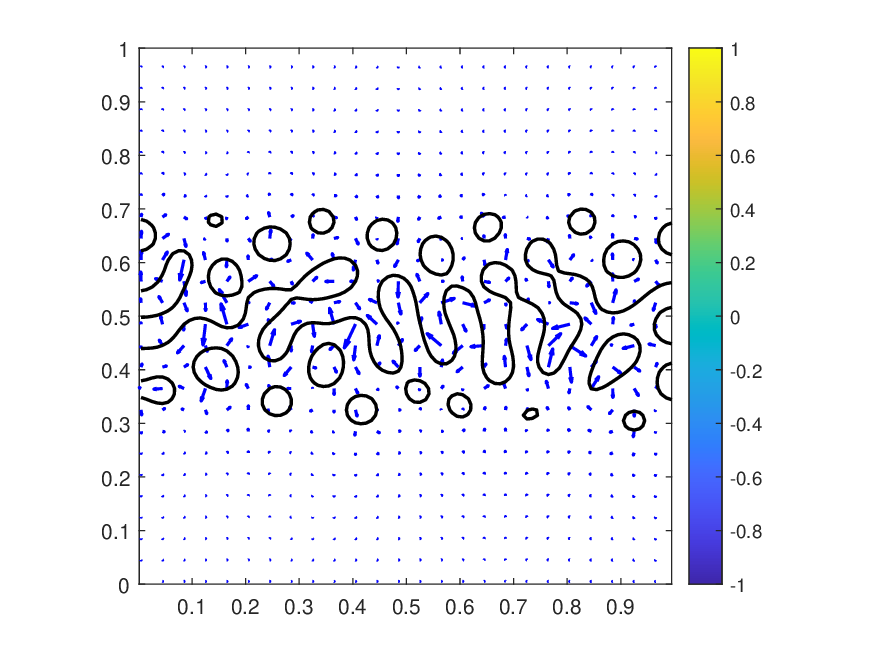}
	\includegraphics[width=0.24\textwidth]{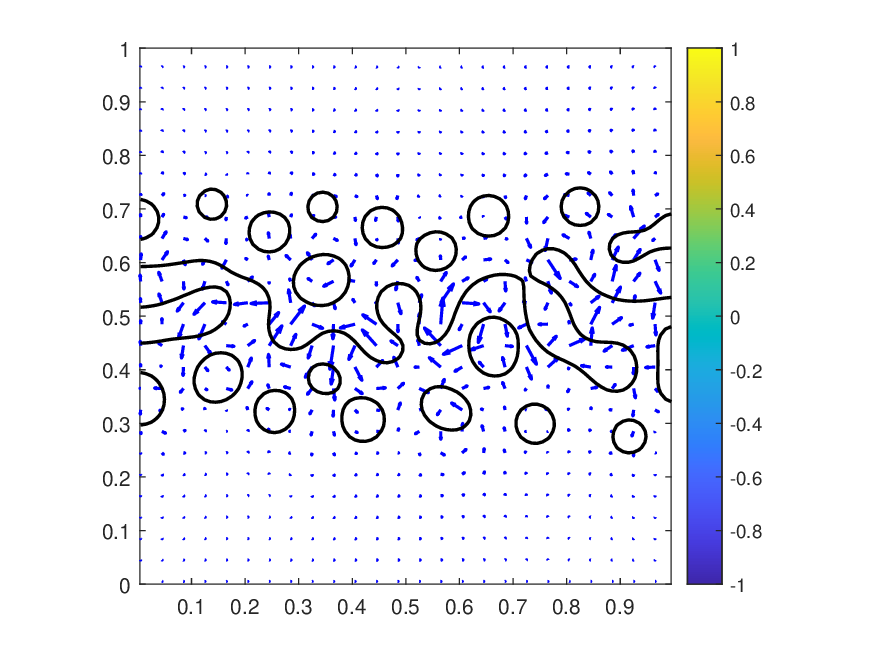}
	\includegraphics[width=0.24\textwidth]{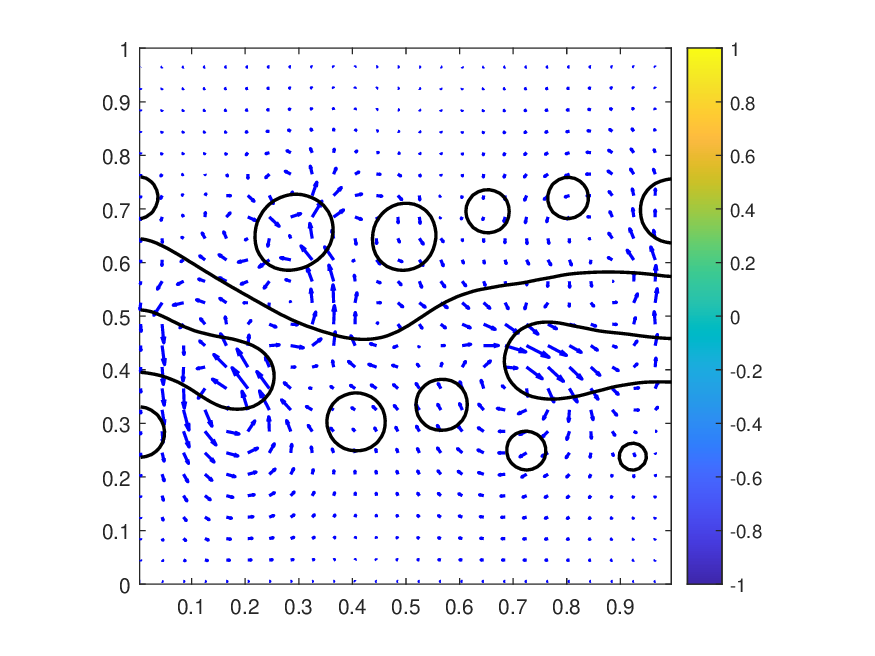}
	\includegraphics[width=0.24\textwidth]{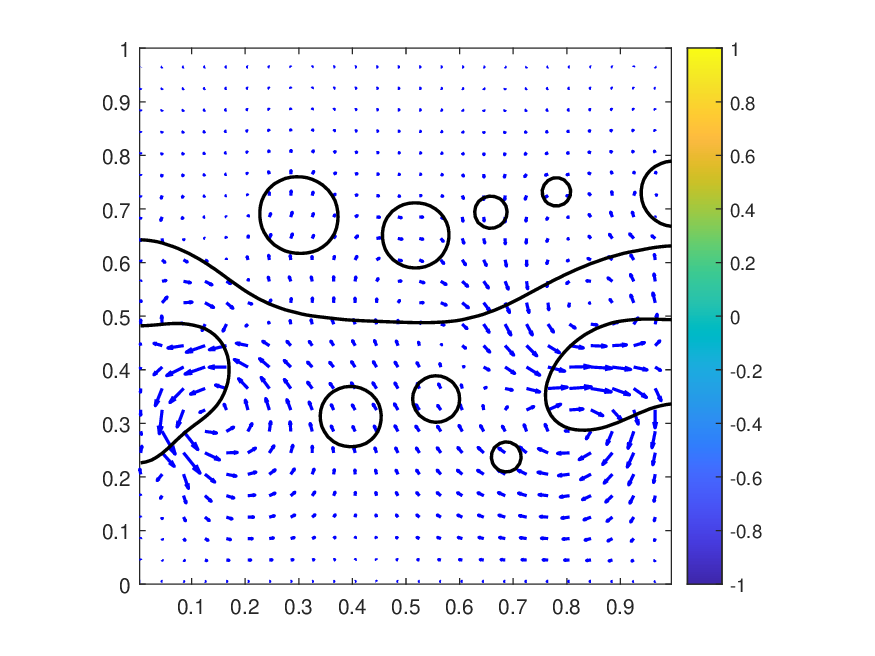}
	\includegraphics[width=0.24\textwidth]{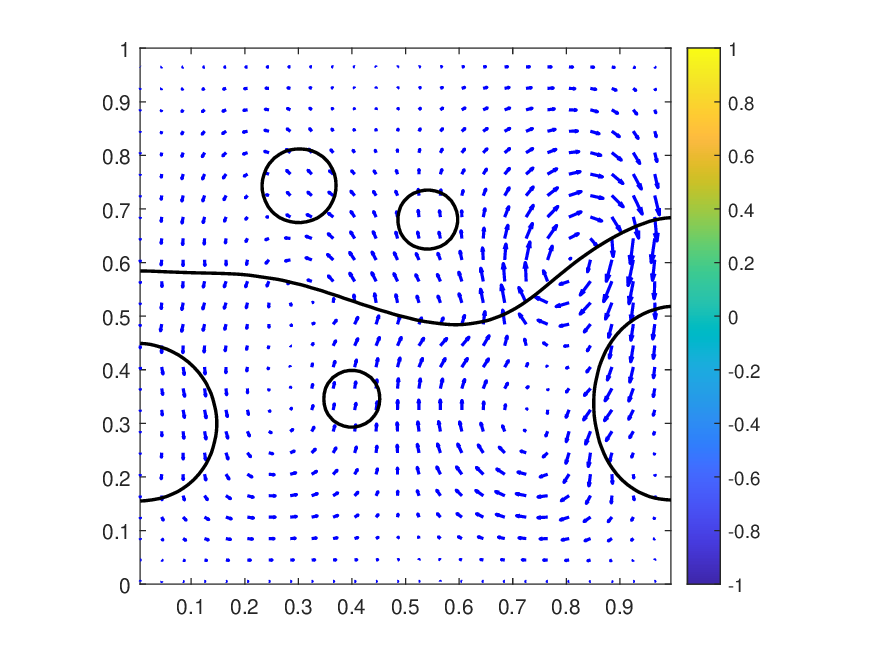}
	\includegraphics[width=0.24\textwidth]{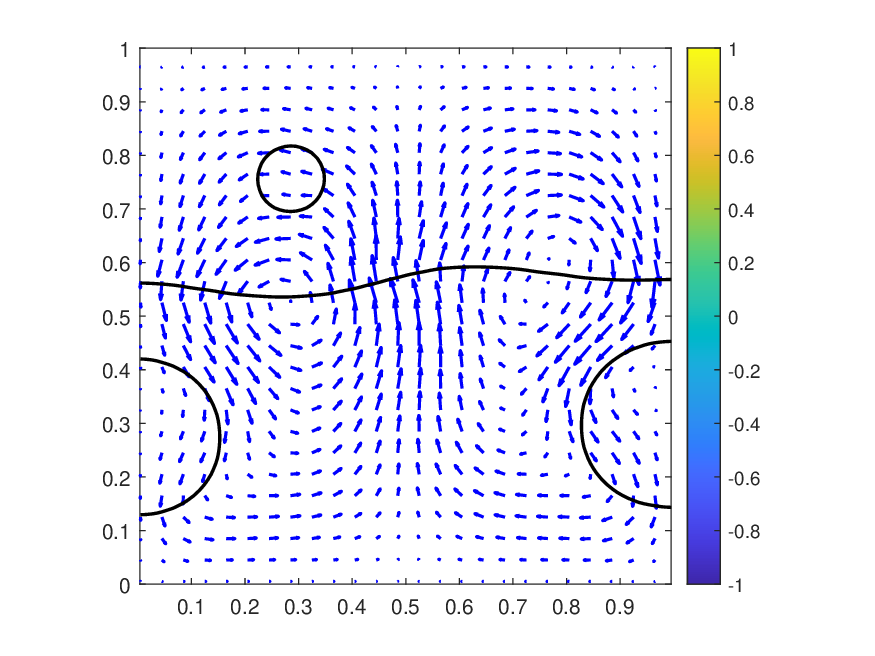}
	\includegraphics[width=0.24\textwidth]{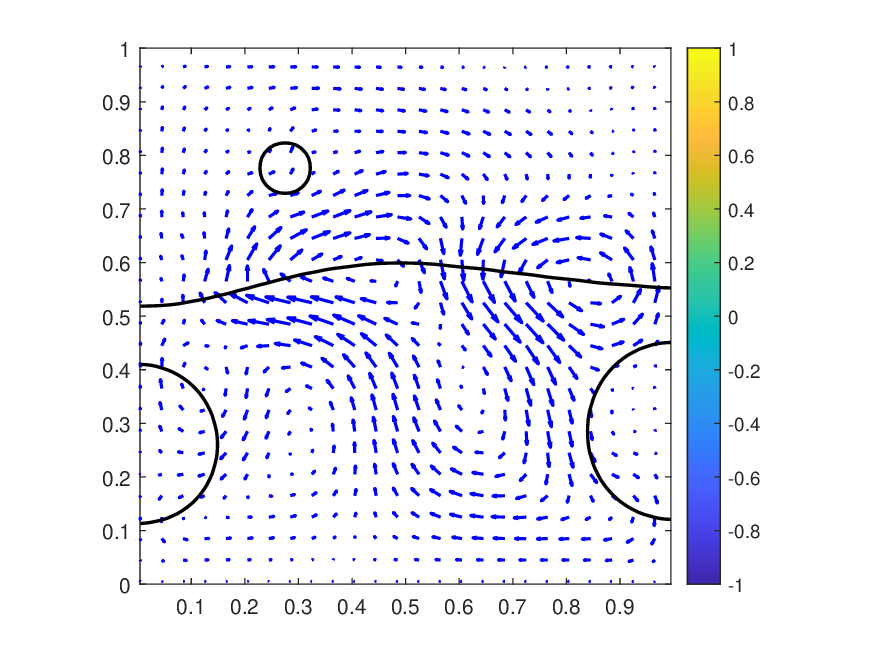}
	\caption{\small  Snapshots of the velocity field  for the flow-coupled nucleation process with $\nu=10^{-3}$ at $t= 0.2, 0.5, 1, 3, 5, 10, 15$ and $20$, respectively.}
	\label{fig:rand:u:nu1e3}
\end{figure}
\begin{figure}  [!htbp] \small
	\centering 
	\subfigure[$\nu = 1$]
	{
		\includegraphics[width=0.23\textwidth]{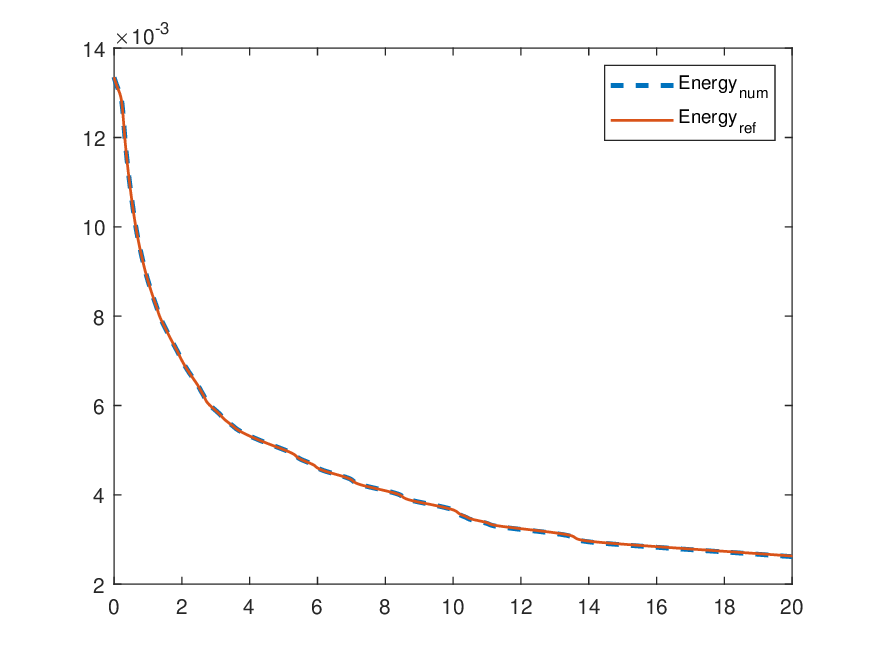}
		\includegraphics[width=0.23\textwidth]{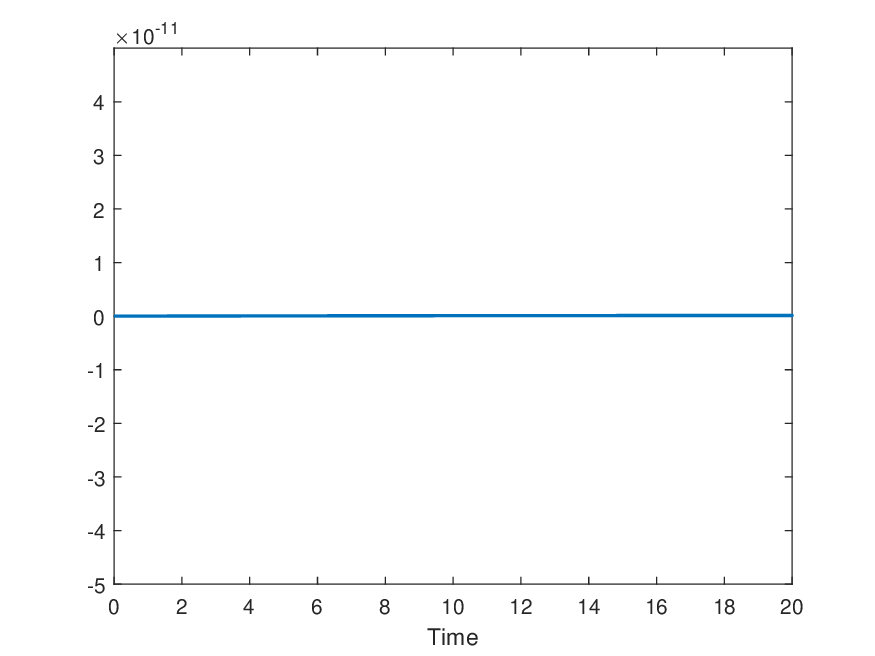}
	}
	\subfigure[$\nu=10^{-3}$]
	{
		\includegraphics[width=0.23\textwidth]{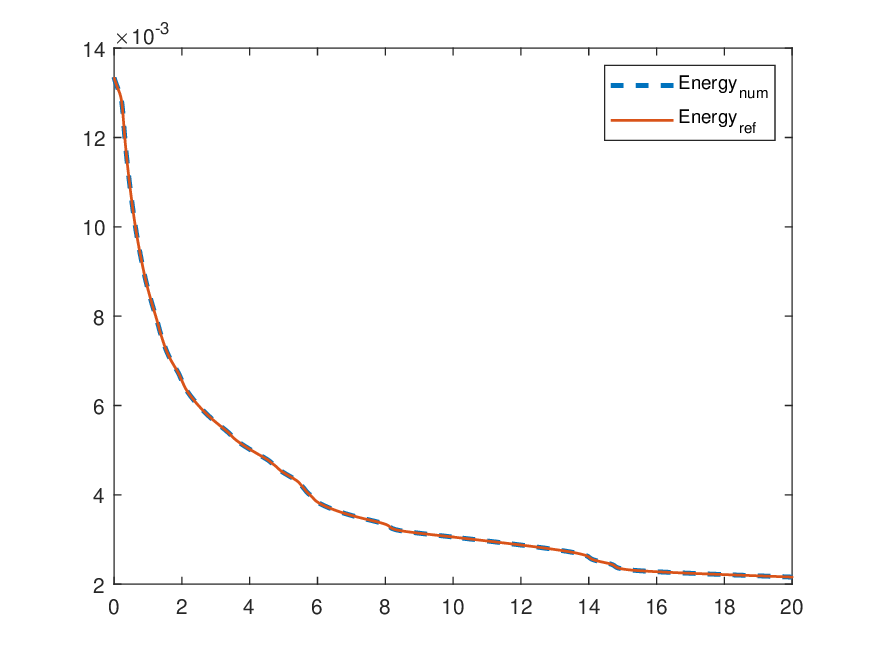}
		\includegraphics[width=0.23\textwidth]{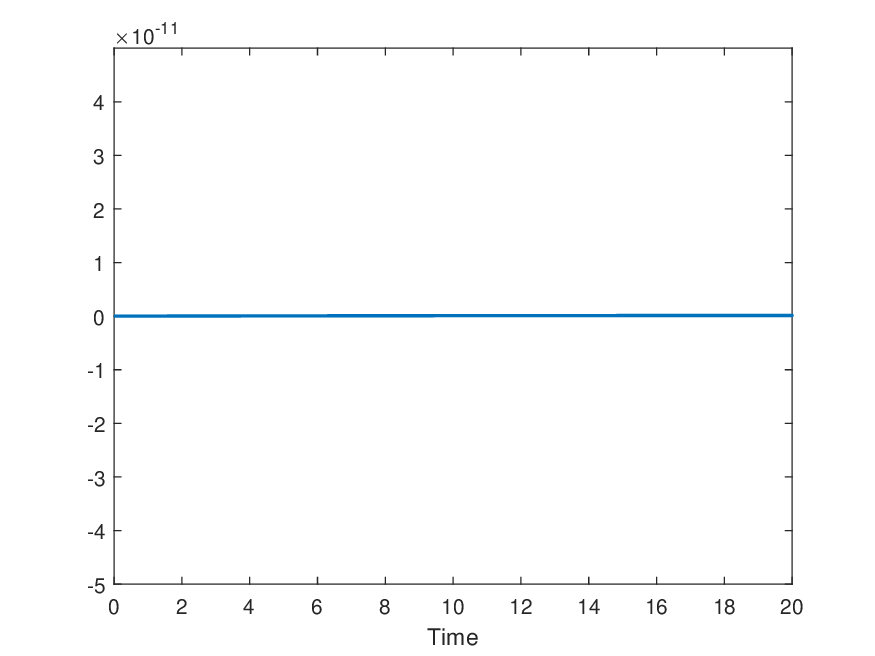}
	}
	\setlength{\abovecaptionskip}{0.01cm} 
	\setlength{\belowcaptionskip}{0.0cm}
	\caption{\small  Evolutions of energy and mass.}
	\label{fig:rand:E_M}
\end{figure}
We also plot the energy decay curves and mass errors for $\nu =1$ and $\nu = 10^{-3}$ in Fig. \ref{fig:rand:E_M}, respectively. As shown, the modified energy \eqref{energy:mod:2th}, obtained via the second-order IVS scheme combined with the relaxation method, closely matches the original energy \eqref{model:McEc:e2} in both cases. Moreover, both energies exhibit monotonic dissipation over time, confirming that the system evolves toward a dynamically steady state. Additionally, the mass remains conserved with at least 11-digit precision in both cases, verifying that the proposed scheme adheres to the mass conservation law, as established in Theorem \ref{thm:energy}.

\subsection{Buoyancy-driven flow} 
When the density difference between the surrounding fluid and the droplet is small,  we can reformulate the variable density momentum equation using the Boussinesq approximation \cite{lee2012comparison} as follows:
\[
\bm u_t+\bm u \cdot \nabla \bm u -\nu \Delta \bm u+\nabla p =- \phi \nabla \mu + \rho(\phi)\bm g,
\]    
where $\rho(\phi) \bm g$ is a buoyancy term with $\rho(\phi)=\chi(\phi - \bar{\phi})$. 
The proposed structure-preserving IVS method remains applicable in this situation. In what follows, we employ the second-order IVS scheme \eqref{scheme:phi:2th}--\eqref{scheme:div:2th} to test two distinct examples with the initial conditions:
\[
\bm u(x, y,0)=\mathbf{0}, ~ p(x, y, 0)={0}, ~ \phi(x,y,0) = \tanh \big(\frac{r-\sqrt{(x-x_a)^2+(y-y_b)^2}}{\epsilon} \big) 
\]
and other parameters will be illustrated in detail in each example.

\subsubsection{Bubble rising} 
We set $\Omega=[0,1]^2$ and let $\bar{\phi}$ represent the spatially averaged order parameter. The parameters used in our simulation are as follows:
\[
r =0.15, ~ x_a = 0.5, ~ y_b = 0.25, ~ \chi = 5, ~ \bm g =(0, 10)^{T}, ~ M=10^{-2},  \]
\[ 
\lambda=10^{-3}, ~ \epsilon =10^{-2}, ~ \beta = \delta_0 = 0, ~ \nu =1, ~\tau=5 \times 10^{-4}, ~ N=200. 
\]

\begin{figure}[!htbp] \small
	\centering
	\includegraphics[width=0.24\textwidth]{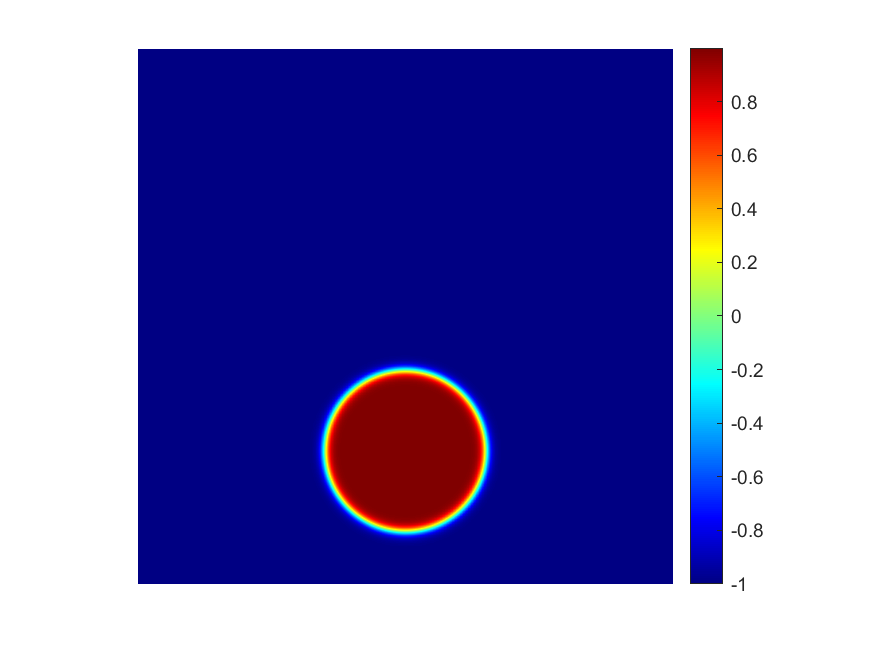}
	\includegraphics[width=0.24\textwidth]{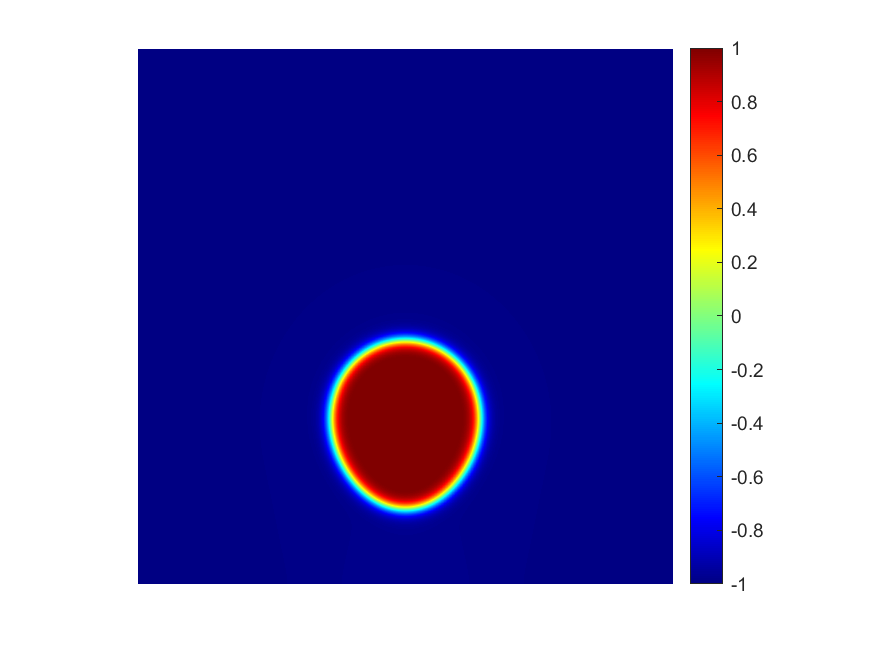}
	\includegraphics[width=0.24\textwidth]{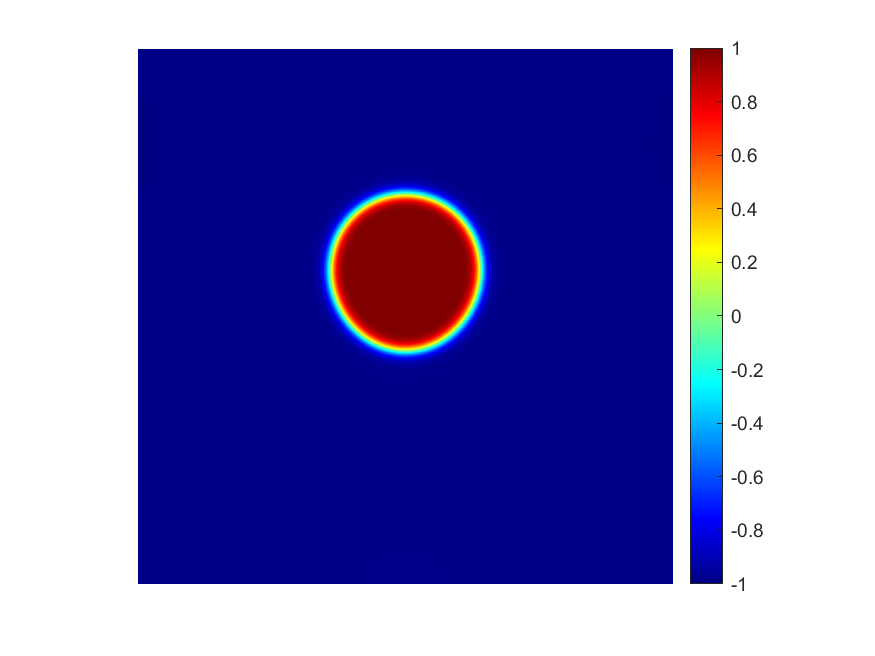}
	\includegraphics[width=0.24\textwidth]{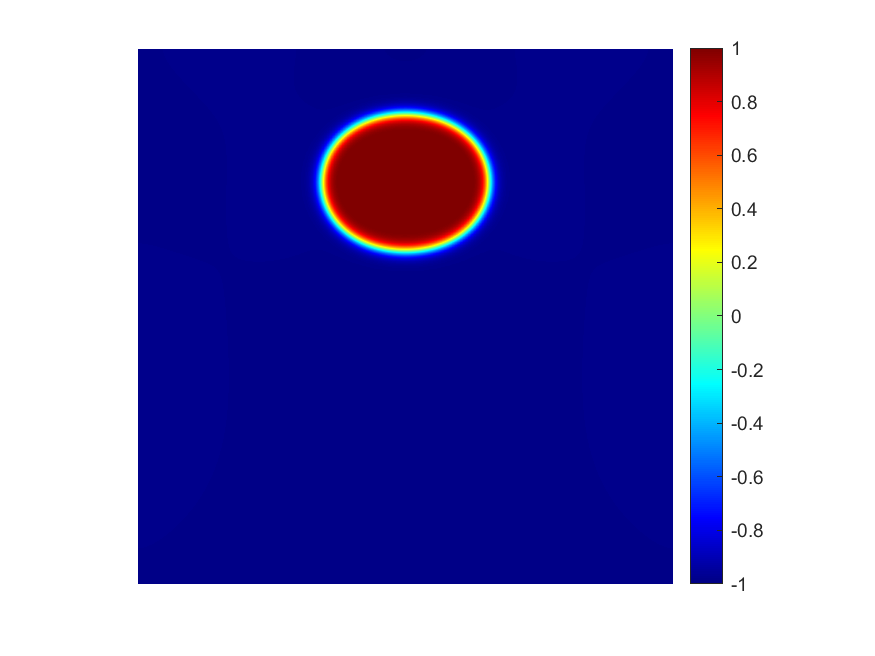}
	\includegraphics[width=0.24\textwidth]{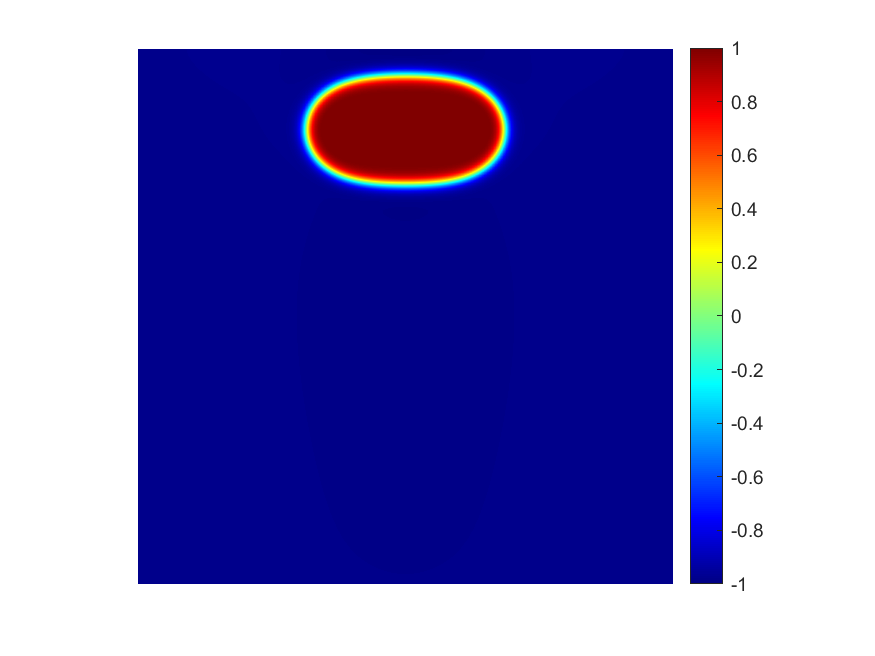}
	\includegraphics[width=0.24\textwidth]{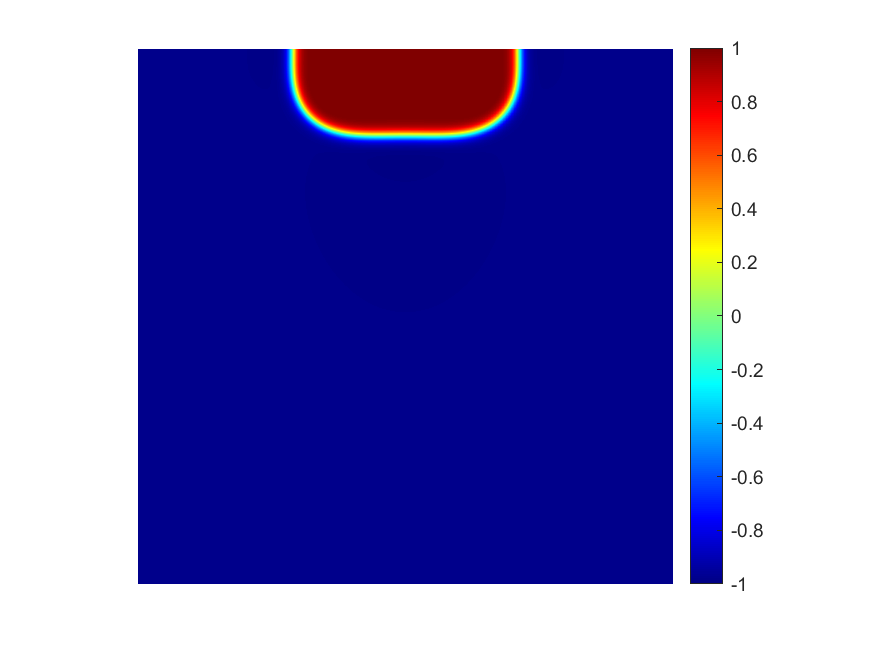}
	\includegraphics[width=0.24\textwidth]{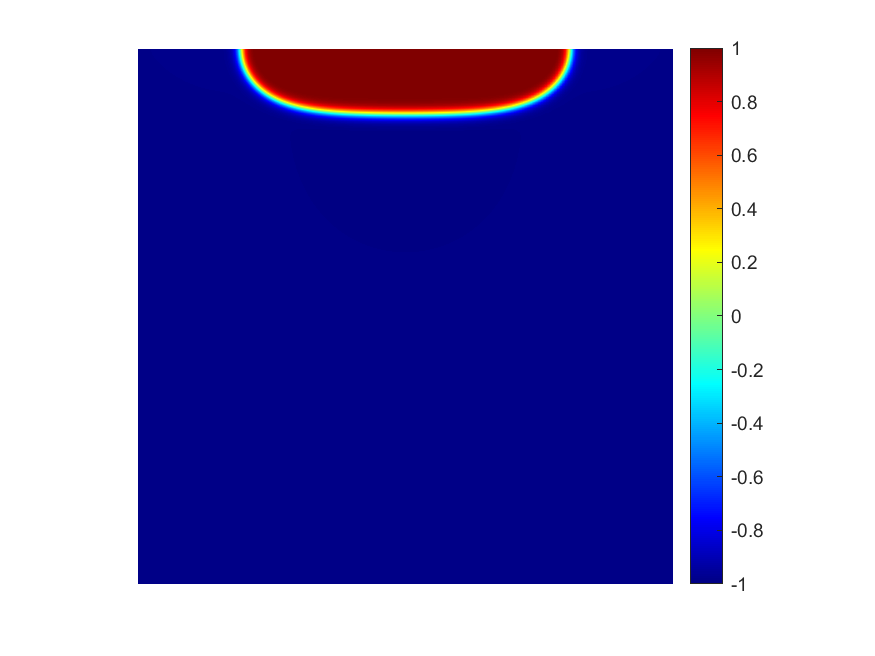}
	\includegraphics[width=0.24\textwidth]{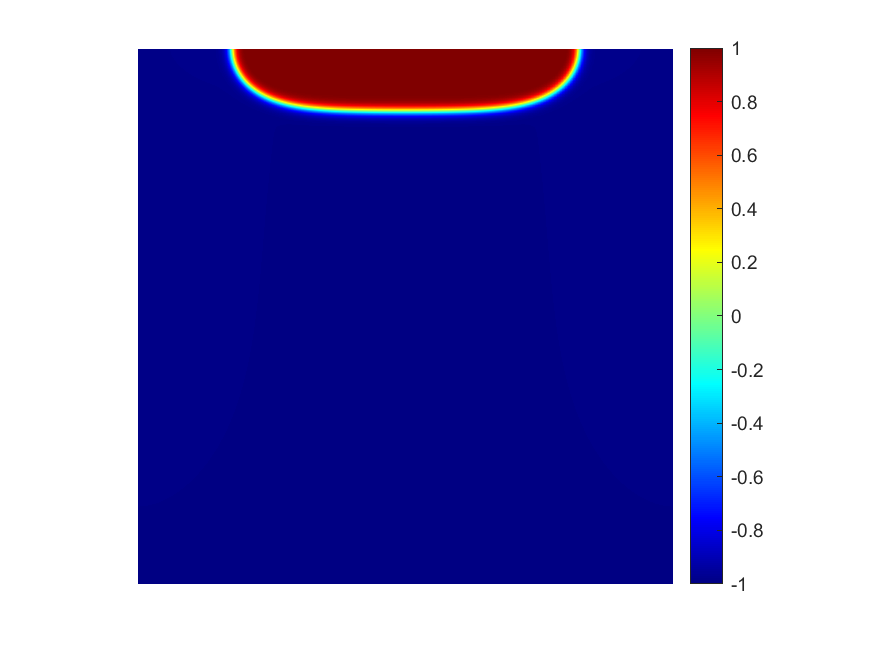}
	\caption{\small  Snapshots of the phase-field variable for the bubble rising process  at $t=0, 0.2,  1, 1.5, 2, 3, 5$ and $12$,  respectively.}
	\label{fig:rising:phi}
\end{figure}
\begin{figure}[!htbp] \small
	\centering
	\includegraphics[width=0.24\textwidth]{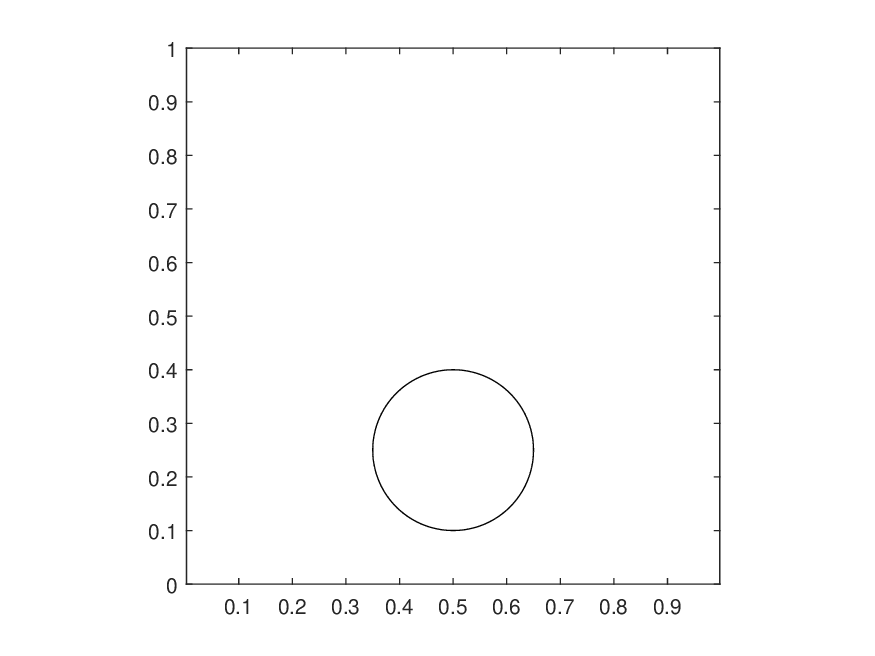}
	\includegraphics[width=0.24\textwidth]{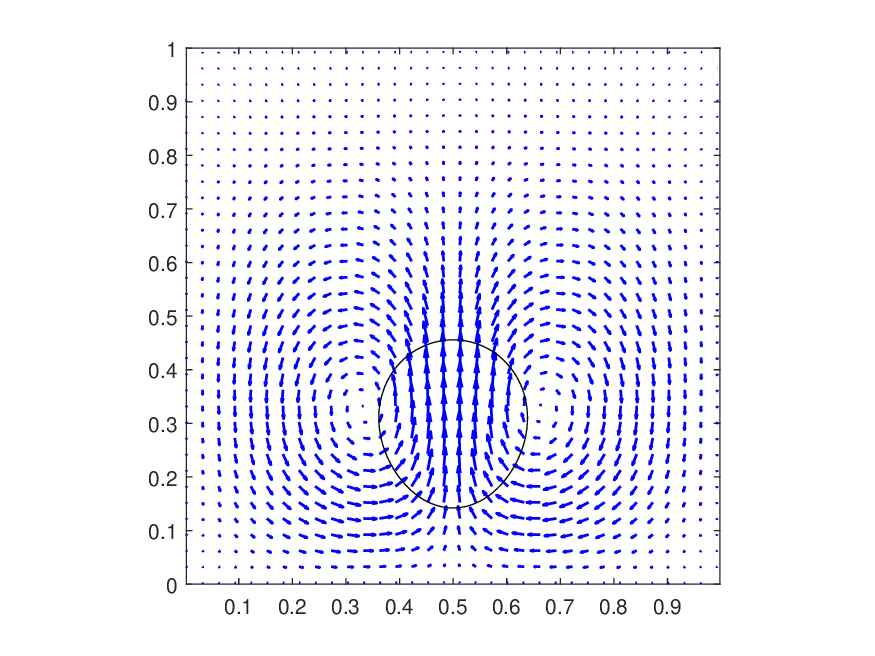}
	\includegraphics[width=0.24\textwidth]{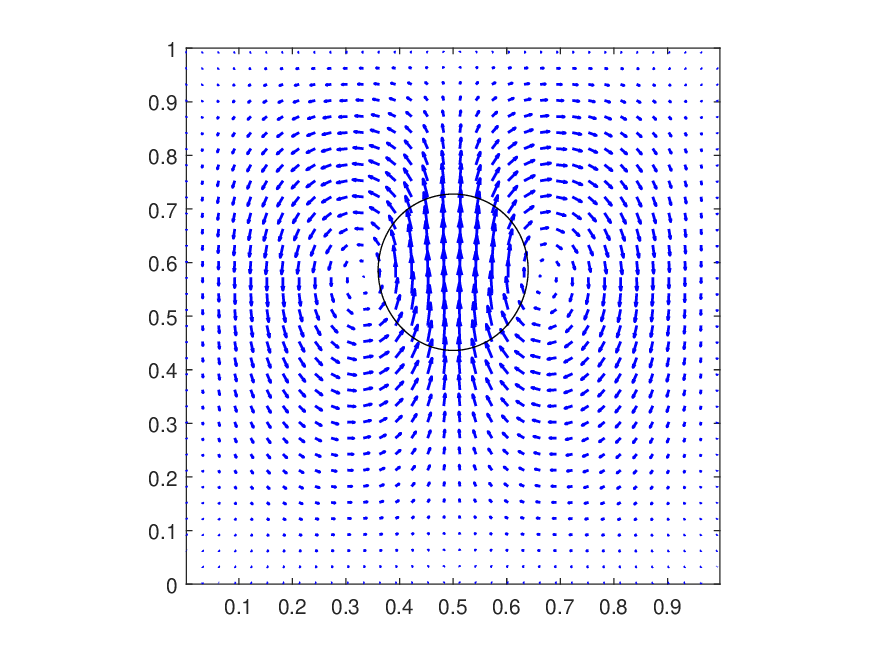}
	\includegraphics[width=0.24\textwidth]{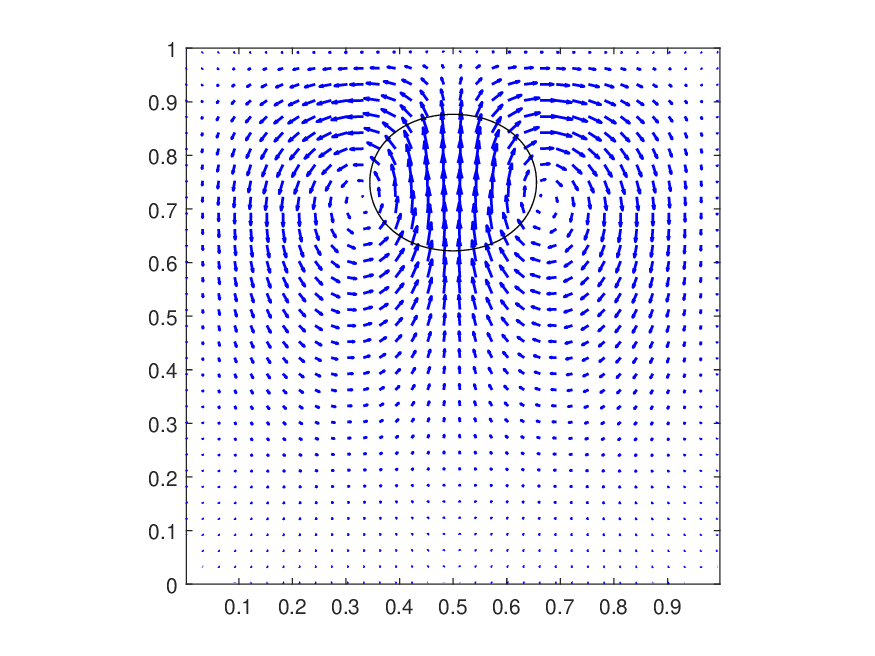}
	\includegraphics[width=0.24\textwidth]{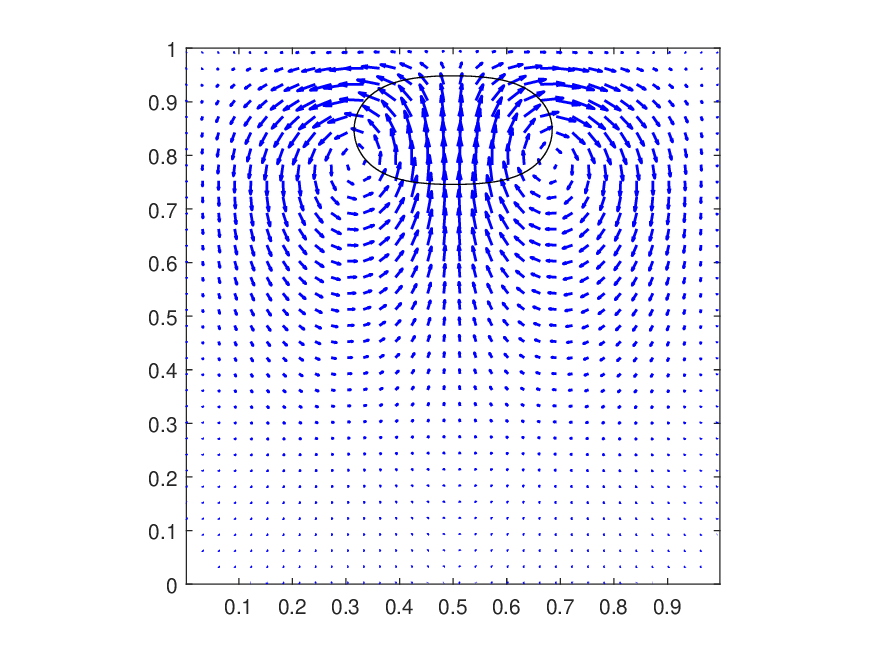}
	\includegraphics[width=0.24\textwidth]{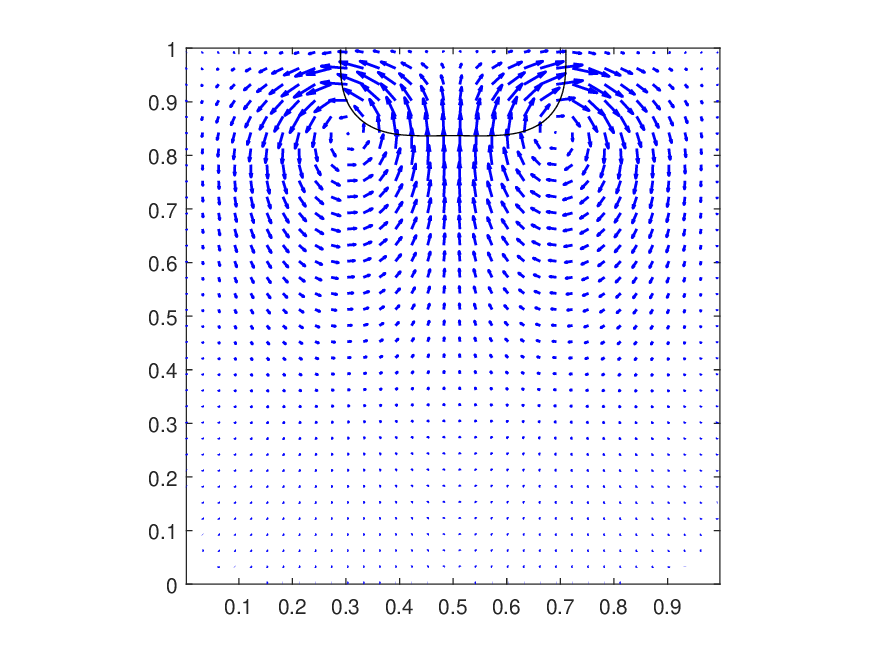}
	\includegraphics[width=0.24\textwidth]{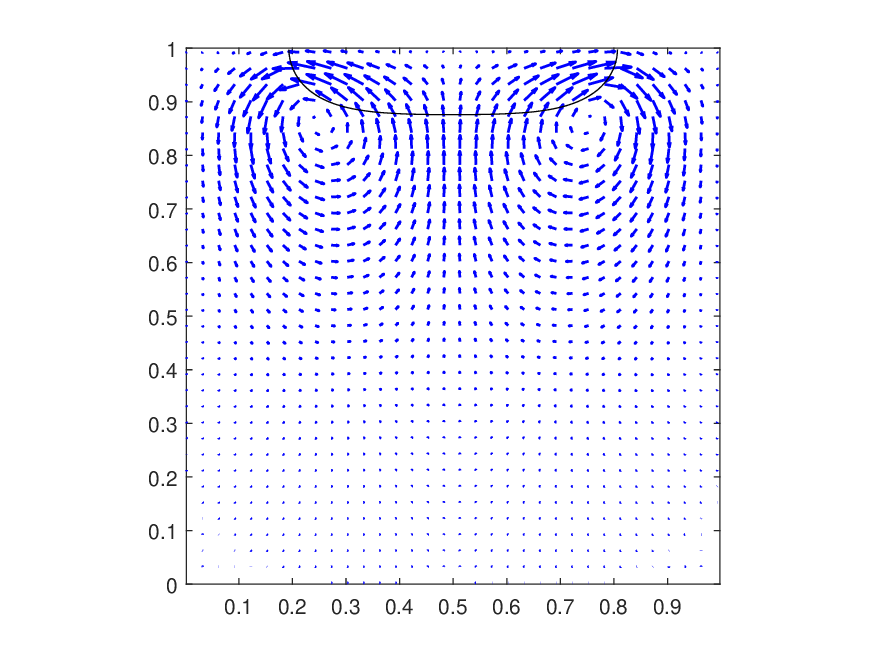}
	\includegraphics[width=0.24\textwidth]{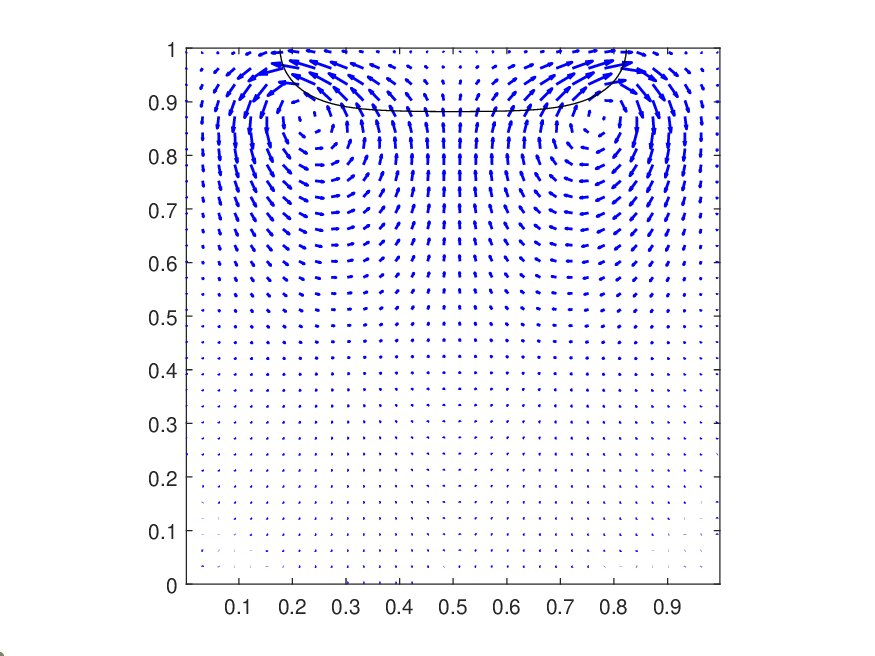}
	\caption{\small  Snapshots of the the velocity field  for the bubble rising process  at $t=0, 0.2,  1, 1.5, 2, 3, 5$ and $12$, respectively.}
	\label{fig:rising:u}
\end{figure}
In Figs. \ref{fig:rising:phi}--\ref{fig:rising:u}, we present snapshots of the phase-field variable and the corresponding velocity field at different time instants. Initially, the bubble appears as a circular shape near the bottom of the domain. During the ascent, the shape of the bubble transitions from a circle to an ellipse due to its lower density compared to the surrounding fluid, allowing it to rise further. Upon reaching the upper wall, the bubble gradually stretches horizontally until it attains a stable state.

\subsubsection{Dripping droplet} 
We also simulate the dripping dynamics of a liquid drop that falls from the top wall under gravity force. We set the computational domain $\Omega=[0,1]\times[0,2]$ with the parameters:
\[
r =0.32, ~ x_a = 0.5, ~ y_b = 2.1, ~\chi=1, ~  \bm g =(0, -10)^{T}, ~ \bar{\phi}=0,  \]
\[ 
M=10^{-2}, ~ \lambda=10^{-3}, ~ \epsilon =10^{-2}, ~ \beta = \delta_0 = 0,   ~\tau=4 \times 10^{-4}, ~ h=4 \times 10^{-3}. 
\]

\begin{figure}  [!htbp] \small
	\vspace{-0.4cm}
	\centering 
	\subfigure
	{
		\includegraphics[width=0.38\textwidth]{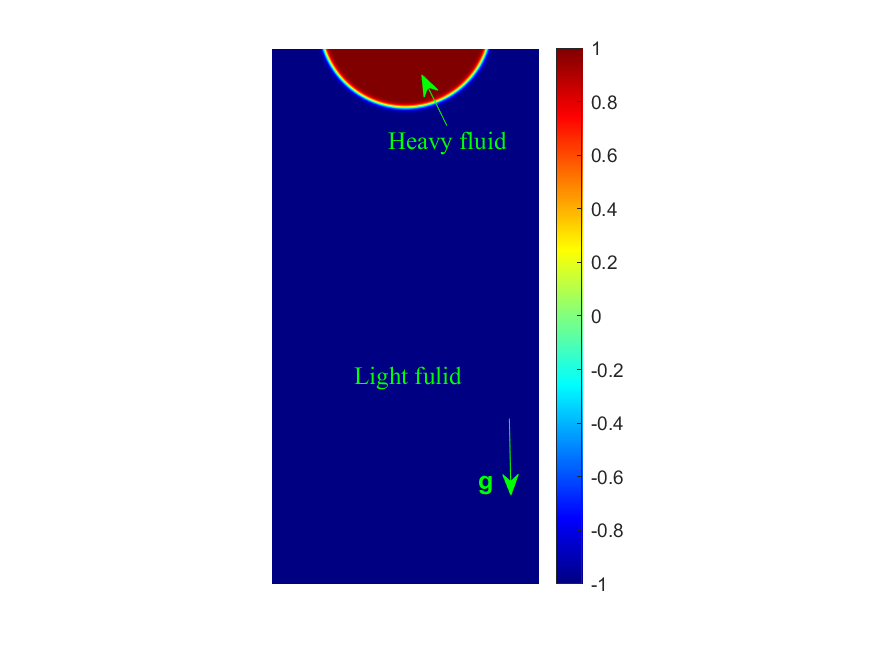}
	}
	\setlength{\abovecaptionskip}{0.01cm} 
	\setlength{\belowcaptionskip}{0.0cm}
	\caption{\small Snapshot of the initial phase-field variable.}
	\label{fig:dripping:phi0}
\end{figure}
Initially, the droplet with heavier density is attached to the upper wall, see Fig. \ref{fig:dripping:phi0}. In Figs. \ref{fig:dripping:Re10}--\ref{fig:dripping:Re50}, we plot the snapshots of the phase-field variable with different viscosities  $\nu =10^{-1}$ and $2 \times 10^{-2}$, respectively.
In both simulations, the droplet initially elongates under the influence of gravity, then pinches off and descends to the bottom of the domain. Notably, as the Reynolds number increases, the pinch-off behavior occurs earlier, and the droplet descends more quickly. 

\begin{figure}  [!htbp] \small
	\centering 
	\includegraphics[width=0.28\textwidth]{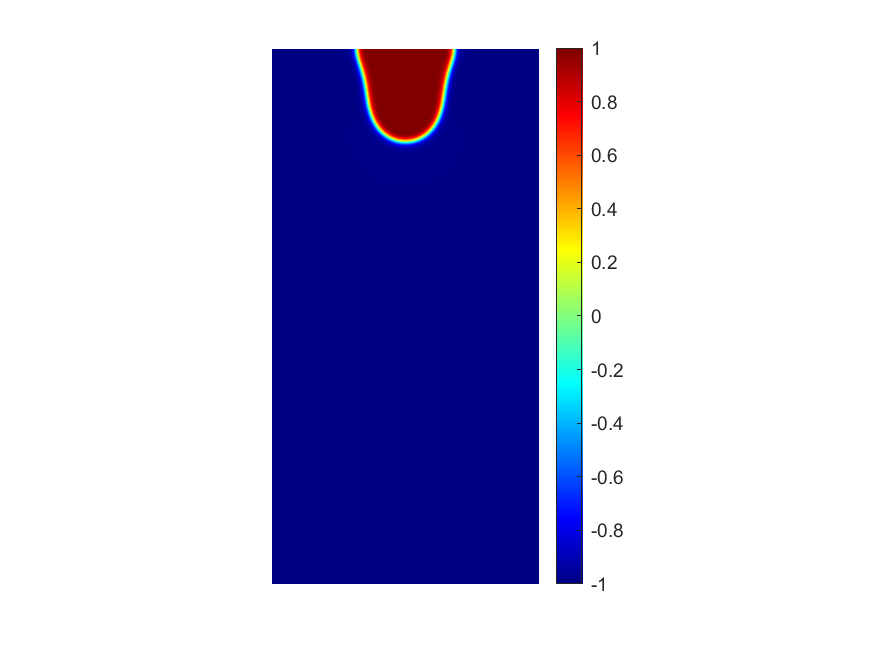} \hspace{-0.08\textwidth}
	\includegraphics[width=0.28\textwidth]{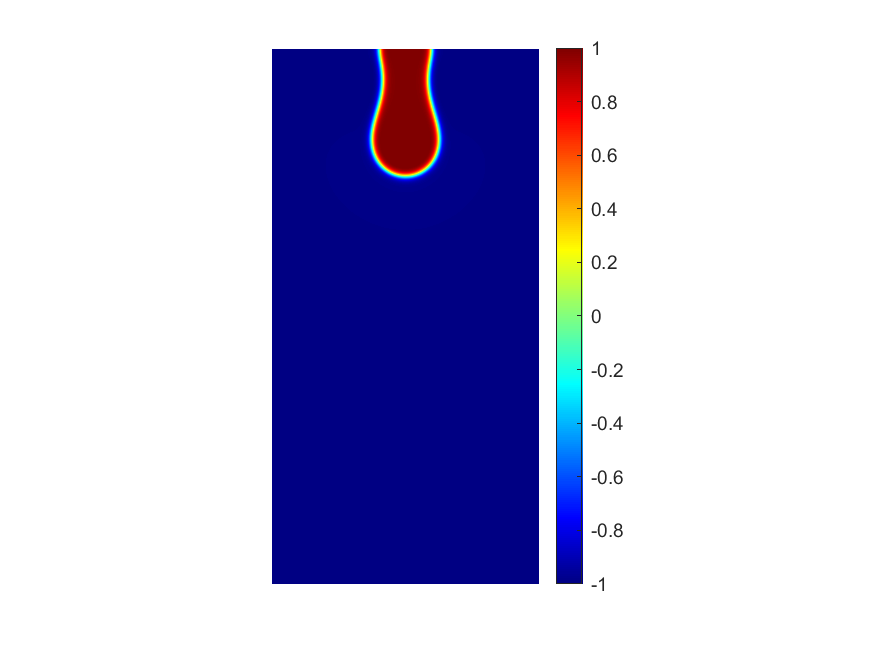} \hspace{-0.08\textwidth}
	\includegraphics[width=0.28\textwidth]{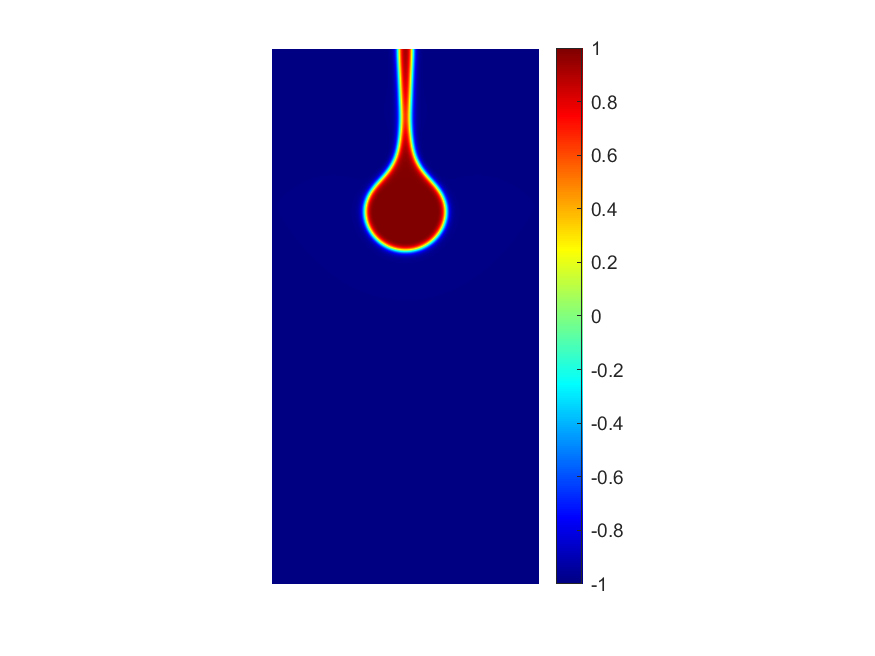}\hspace{-0.08\textwidth}
	\includegraphics[width=0.28\textwidth]{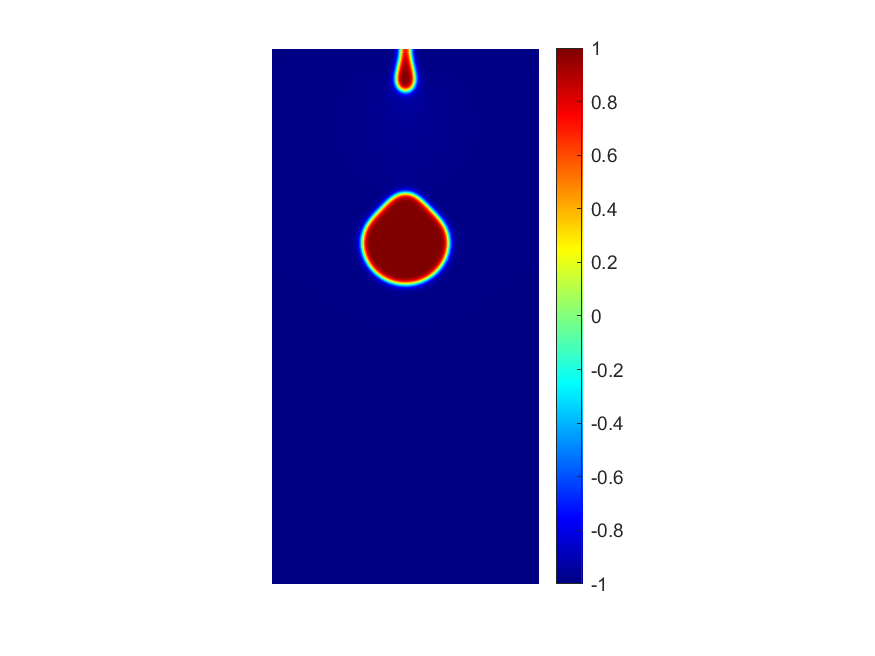}
	\setlength{\abovecaptionskip}{0.01cm} 
	\setlength{\belowcaptionskip}{0.0cm}
	\caption{\small Snapshots of the phase-field variable with $\nu=10^{-1}$ for the droplet falling at $t=0.5, 0.8, 1.3$ and 1.5, respectively.}
	\label{fig:dripping:Re10}
\end{figure}
\begin{figure}  [!thbp] \small
	\centering 
	\includegraphics[width=0.28\textwidth]{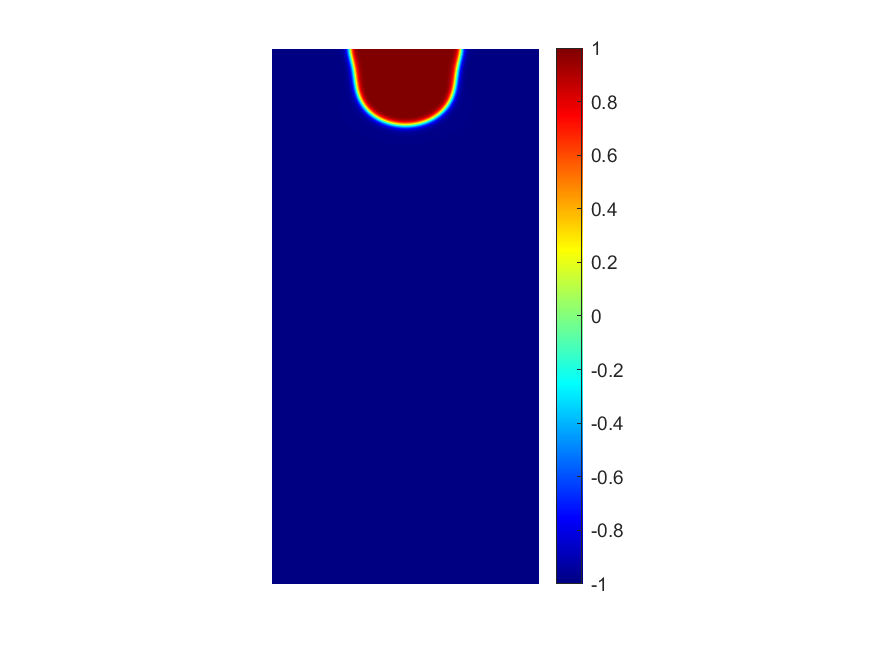} \hspace{-0.08\textwidth}
	\includegraphics[width=0.28\textwidth]{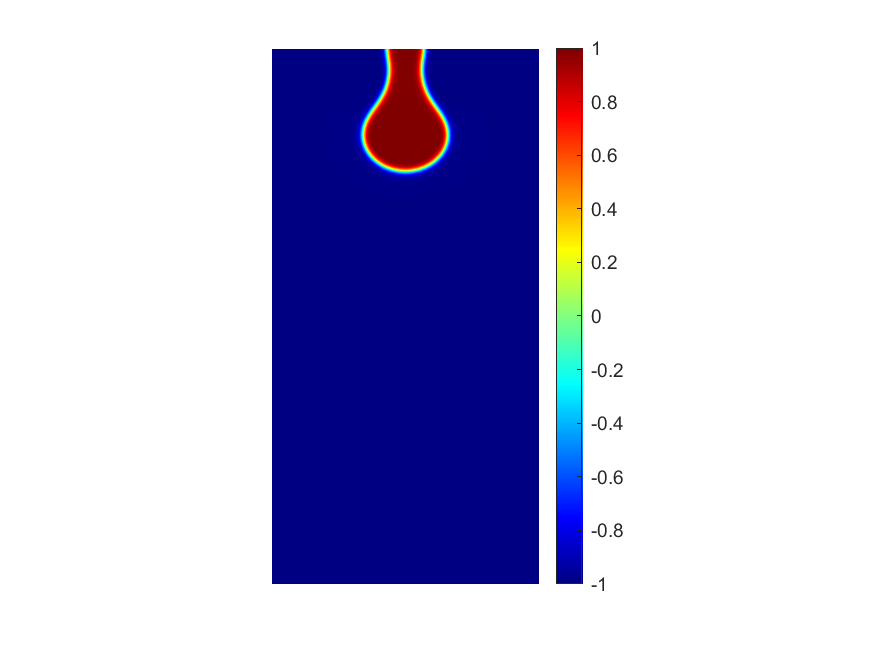} \hspace{-0.08\textwidth}
	\includegraphics[width=0.28\textwidth]{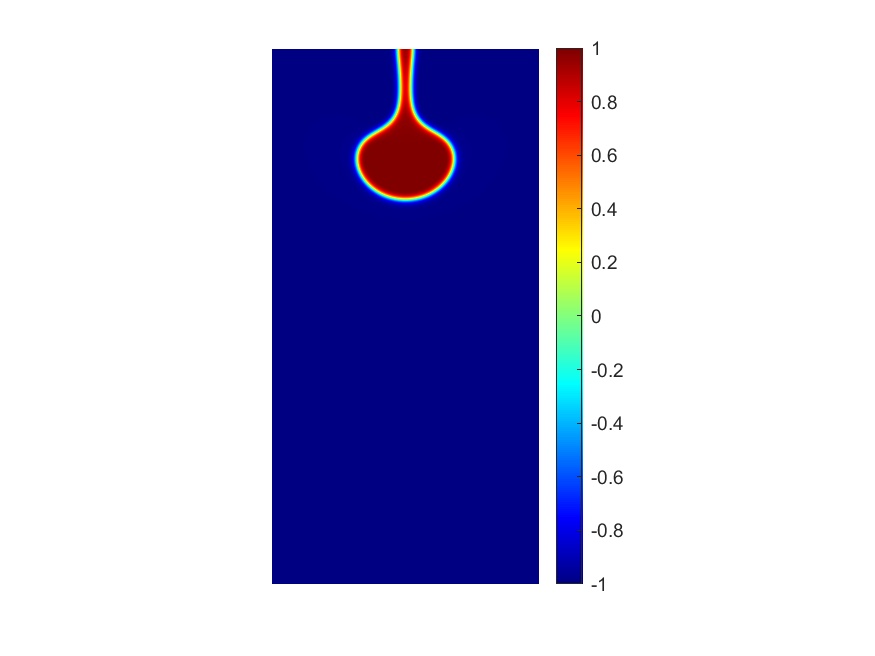} \hspace{-0.08\textwidth}
	\includegraphics[width=0.28\textwidth]{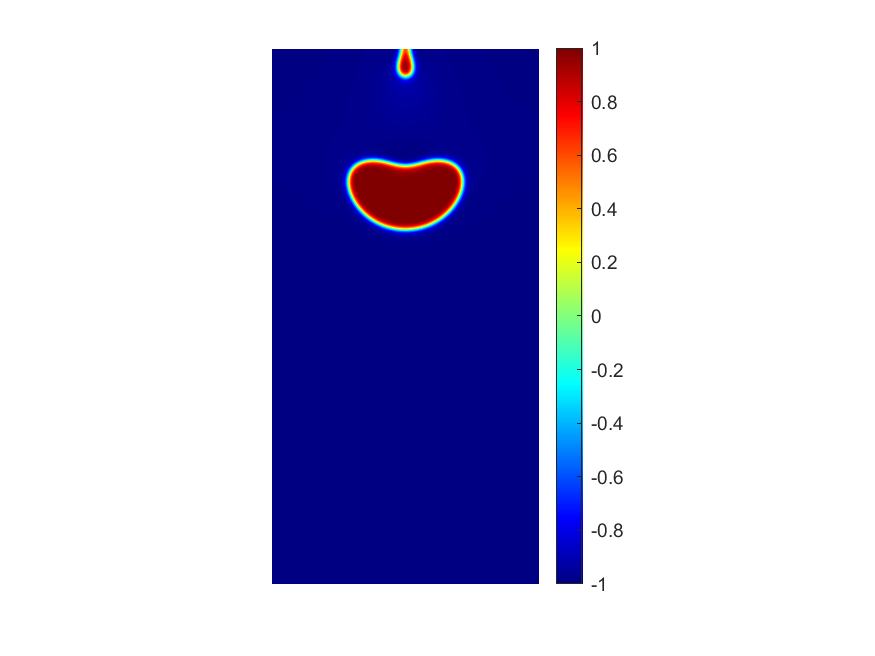} 
	\setlength{\abovecaptionskip}{0.01cm} 
	\setlength{\belowcaptionskip}{0.0cm}
	\caption{\small Snapshots of the phase-field variable with $\nu=2 \times 10^{-2}$  for the droplet falling process at $t=0.2, 0.4, 0.5$ and 0.6, respectively.}
	\label{fig:dripping:Re50}
\end{figure}

\section{Conclusions}\label{conclusions}
In this paper, we have developed two structure-preserving IVS schemes for the CHNS system, which are linear, decoupled, mass-conserving, and unconditionally energy stable. Moreover, the proposed schemes require only solving a few constant-coefficient linear equations, along with a single linear equation with one unknown,  making these algorithms efficient and easy to implement. The nonlinearities are handled by combining the SAV method with the ZEC approach, and the incremental viscous splitting technique is to address the coupling of pressure and velocity. Meanwhile, a well user-defined, time-dependent parameter is introduced to enhance the accuracy of the velocity and pressure approximations. 
Specifically, we have established first-order error estimates for the phase-field variable in the $\ell^{\infty}(0,T;H^1)$-norm, for the chemical potential in the $\ell^{2}(0,T;H^1)$-norm, for the velocity in both the $\ell^{\infty}(0,T;\bm H^1)$- and $\ell^2(0,T;\bm H^2)$-norms, and for the pressure in the $\ell^2(0,T;H^1)$-norm. 
Ample numerical experiments are provided to show the reliability and efficiency of the developed schemes in modeling the phase separation, bubbles merging, bubble rising and dripping dynamics of a liquid drop.  

Due to the complexity, this work only focuses on the time semi-discretization scheme, leaving the fully discrete error analysis under suitable spatial discretizations for future research. Furthermore, a rigorous error analysis of the second-order IVS scheme remains an open problem, and we also would like to address it in future work.

%


\begin{acknowledgement} 
The work was partially supported by the National Natural Science Foundation of China (Nos. 12131014, 12271302), by the Natural Science Foundation of Shandong Province (Nos. ZR2024MA023, ZR2024JQ030), by the Fundamental Research Funds for the Central Universities (Nos. 202264006, 202561097), and by the OUC Scientific Research Program for Young Talented Professionals.
\end{acknowledgement}
\bibliographystyle{spmpsci}   
\bibliography{Ref_AVVS}

\end{document}